\documentclass[a4paper]{amsart}

\usepackage{amsmath}
\usepackage{amssymb}
\usepackage[foot]{amsaddr}
\usepackage{overpic}
\usepackage{microtype}
\usepackage{color}
\usepackage{tikz-cd}
\usepackage{multicol}
\usepackage{contour}
\contourlength{1pt}
\usepackage{enumerate}
\usepackage{enumitem}
\setlist[enumerate, 1]{label=(\roman*)}
\usepackage{sidecap}
\usepackage{aligned-overset}
\usepackage{physics}

\usepackage{hyperref}

\newcommand{\R}{\mathbb{R}}
\newcommand{\Z}{\mathbb{Z}}
\newcommand{\ff}{\mathfrak{f}}

\newcommand{\cc}{\mathfrak{c}}
\newcommand{\bb}{\mathfrak{b}}
\newcommand{\aaa}{\mathfrak{a}}
\newcommand{\la}{\langle}
\newcommand{\ra}{\rangle}
\newcommand{\sv}[1]{\begin{pmatrix}#1\end{pmatrix}}
\newcommand{\svs}[1]{\left(\!\begin{smallmatrix}#1\end{smallmatrix}\!\right)}

\newcommand{\dash}{\discretionary{-}{}{-}\penalty1000\hskip0pt}
\newcommand{\mlput}[1]{\hbox to 0pt{\hss{#1}}}
\newcommand{\mrput}[1]{\hbox to 0pt{{#1}\hss}}
\newcommand{\mcput}[1]{\hbox to 0pt{\hss{#1}\hss}}
\newcommand{\M}{\mathcal{M}}

\def\be{\begin{equation}}
\def\ee{\end{equation}}

\definecolor{airforceblue}{rgb}{0.36, 0.54, 0.66}
\definecolor{myred}{rgb}{0.9, 0.3, 0.3}
\definecolor{myblue}{rgb}{0.3, 0.3, 0.6}

\DeclareMathOperator{\area}{area}

\DeclareMathOperator{\myspan}{span}
\DeclareMathOperator{\sgn}{sgn}
\DeclareMathOperator{\SO}{SO}

\newtheorem{thm}{Theorem}
\newtheorem{prop}[thm]{Proposition}
\newtheorem{cor}[thm]{Corollary}
\newtheorem{lem}[thm]{Lemma}
\newtheorem{ex}[thm]{Example}
\newtheorem{rem}[thm]{Remark}
\theoremstyle{definition}
\newtheorem{defn}[thm]{Definition}

\makeatletter\def\@captionfont{\normalfont\footnotesize}\makeatother

\makeatletter
\def\paragraph{\medskip\@startsection{paragraph}{4}%
  \z@\z@{-\fontdimen2\font}%
  {\normalfont\bfseries}}
\makeatother

\title{Isometric surfaces in isotropic 3-space}

\author{Christian M{\"u}ller$^1$ \and Helmut Pottmann$^1$}
\address{\begin{minipage}{\linewidth}
  $^1$TU Wien, Institute for Discrete Mathematics and Geometry, 
  A-1040 Vienna\\
  \phantom{$^1$}\emph{Email:} cmueller@geometrie.tuwien.ac.at
\end{minipage}
}

\address{\begin{minipage}{\linewidth}
  \phantom{$^1$}\emph{Email:} pottmann@geometrie.tuwien.ac.at
\end{minipage}
}

\begin{document}

\begin{abstract}
  While the notion of isometric deformations of surfaces is
  straightforward for surfaces with Euclidean metric, a corresponding
  notion in isotropic space has been missing. By making Gauss'
  Theorema Egregium a necessary condition we develop a sensible notion of
  isometric surfaces in isotropic space. The well-known examples in
  Euclidean space, like isometries within the associated family of minimal
  surfaces, Bour's theorem, and Minding isometries, find their natural
  analogues in isotropic space. We also include an extensive treatment of
  infinitesimal flexibility, or infinitesimal deformation, of surfaces. We
  prove results for the isotropic displacement diagrams in analogy to its
  well-known counterparts in Euclidean space culminating in the existence
  of an isotropic Darboux wreath consisting of six surfaces. We show
  several interesting relations for special parametrizations involving
  K{\oe}nigs and Voss nets of smooth and discrete surfaces within the
  Darboux wreath and we encounter surfaces of constant Gaussian and mean
  curvature. At several occasions, we point to connections to statics as
  the isotropic space is a natural language to describe the Airy stress
  function.
\end{abstract}

\maketitle

\section{Introduction}

The geometry of isotropic 3-space $I^3$ is a Cayley-Klein geometry based
on a group of affine transformations which preserve the isotropic
semi-norm $\|(x, y, z)\|_i := \sqrt{x^2 + y^2}$. K.~Strubecker
\cite{strubecker1,strubecker2,strubecker3,strubecker4}
has studied this geometry systematically. While at first sight the metric
in $I^3$ is too degenerate to result in an elegant and non-trivial theory,
he showed that this is not true, since one can define so-called replacing
invariants which make isotropic geometry interesting. An example is
provided by the Gaussian curvature of surfaces. While the Riemannian
viewpoint would yield vanishing Gaussian curvature, one can define a
Gaussian mapping to a non-degenerate parabolic sphere and in this way
obtain an interesting curvature theory and a Gaussian curvature which
shares close similarities with its counterpart in Euclidean 3-space $E^3$.
Strubecker's pioneering work stimulated a lot of further research on the
geometry in $I^3$, most of which is found in the monograph by
Sachs~\cite{Sachs:1990}. A comparison of the large body of results in
isotropic geometry with their Euclidean counterparts leads to a view of
isotropic geometry as a structure preserving simplification
of Euclidean geometry. This makes isotropic geometry an attractive tool
for the solution of difficult Euclidean problems. The main idea is to
first study the isotropic counterpart and then transform it to the
Euclidean version. 

The present paper has been motivated by such a hard problem, namely the
design of flexible quadrilateral meshes. Such meshes are mechanisms when
considering their faces as rigid bodies and the edges joining two faces
as rotational joints. Already the classification of such meshes with 
$3 \times 3$ faces is a very difficult problem~\cite{izmestiev-2017}, and
apart from special cases addressed below it is not known how to combine
these small meshes to larger ones. Hence, we came up with the idea to
first study the isotropic counterpart. However, surprisingly there is no
prior work on isometric surfaces in $I^3$, a gap that will be filled in
the present paper.

Looking just at the metric, in isotropic 3-space $I^3$ any two surfaces
are locally isometric to each other (see
Figure~\ref{fig:wrong-isometry}~(a)). To obtain non-trivial isometries, we
additionally require agreement of the isotropic Gaussian curvature at
corresponding points of two isometric surfaces. Based on this definition,
we find isotropic counterparts to well-known results on Euclidean
isometries and remarkable relations to the work of Sauer~\cite{sauer:1970}
on infinitesimal Euclidean isometries and concepts of graphic statics and
mechanics, namely Airy stress functions associated to 2D systems in
equilibrium \cite{strubecker:1962,Millar2022,vouga-2012-sss}. Moreover,
we provide the necessary theoretical framework for the study of isotropic
flexible meshes \cite{flexible-isotropic-2025}. Fortunately, these meshes turn out to be rather easily
constructed and serve well as initial guesses for numerical optimization
towards Euclidean flexible meshes (see Section~\ref{sec:future}).

\subsection{Related work}

Isometries of surfaces in Euclidean space $E^3$ have a long history of
geometric research. We point to the survey by Sabitov~\cite{sabitov92}.
Infinitesimal flexibility of a mechanism has mostly been studied in the
general context of bar
and joint frameworks. We refer to Connelly and Guest~\cite{connelly+2022}
for a textbook introduction.
Sauer~\cite{sauer:1970} has an extensive treatment of isometries, in particular
infinitesimal ones, in his book on difference geometry, which is an early
precursor of discrete differential geometry~\cite{bobenko-2008-ddg}. There,
a focus is on flexible quad meshes~\cite{Schief2008}, which recently received 
increasing interest in connection with rigid origami structures. For
major progress in this area we refer to Izmestiev~\cite{izmestiev-2017} and
He and Guest~\cite{he-guest}. 

Isotropic geometry recently received interest within the structural design
community. Already Strubecker~\cite{strubecker:1962} pointed to the fact that
the graph of the Airy stress function of a planar continuum in equilibrium is best
studied within isotropic geometry. One obtains a complete analogoy between mechanical
invariants of the planar stress state and isotropic curvatures of the stress surface.
This concept has been extended for usage in the design of structures which are 
in equilibrium under vertical external loads~\cite{Millar2022,TELLIER-linearWeingarten,vouga-2012-sss}.

\subsection{Contributions and overview}

We start in Section~\ref{sec:basics} by introducing the basic concepts of
isotropic space, in particular curvatures, metric duality, transformation
of contact elements under dualities, support functions and various
relations among them.
The key notion of our paper is the isometry between two surfaces in
isotropic space that we introduce in Section~\ref{sec:isodef}.
We emphasize the sensibility of our notion of isometry on the basis of
three well-known examples, that is, three classes of families of
isometric deformations of surfaces (associated family of minimal surfaces,
Bour's theorem, and Minding isometries) in
Section~\ref{ssec:simpleexamples}.
Infinitesimal isometries are studied in Section~\ref{sec:inf} where one
focus lies on a sensible development of displacement diagrams and the
investigation of infinitesimal isometries of ruled surfaces. 
The displacement diagrams can nicely be inserted into a Darboux wreath of
length six (Sec.~\ref{sec:darboux}). The Darboux wreath contains
interesting surfaces and parametrizations if the special properties like
being a Q-net is imposed on one of its members (Sec.~\ref{sec:special}).
We also consider discrete nets and their infinitesimal isometries for
special cases in Section~\ref{sec:discrete} where we encounter discrete
K{\oe}nigs nets, Voss nets, A-nets, linear Weingarten surfaces, minimal
surfaces, cmc surfaces and constant Gaussian curvature surfaces in the
Darboux wreath.

\section{Basic Concepts in Isotropic Geometry}
\label{sec:basics}

We start with explaining basic notions and properties of isotropic space
$I^3$. It can be seen as a simplified version of \emph{Euclidean Space}
$E^3$. While $E^3$ is endowed with an actual metric, the Euclidean metric,
the isotropic space comes with a pseudometric which ``ignores'' the third
dimension compared to the Euclidean metric.

\subsection{The isotropic space $I^3$}

The standard model of three dimensional isotropic geometry $I^3$ is $\R^3$
endowed with the pseudometric (see, e.g.,~\cite{Sachs:1990})
\begin{equation}
  \label{eq:distance}
  d((x_1, y_1, z_1), (x_2, y_2, z_2)) = \sqrt{(x_1 - x_2)^2 + (y_1 -
  y_2)^2},
\end{equation}
the \emph{isotropic distance}.
In geometric terms, the isotropic distance between two points is measured
as Euclidean distance in the top view.
\begin{defn}
  The orthogonal projection into the $xy$-plane is called 
  \emph{top view}. In coordinates it is the map 
  \begin{equation*}
    p = (x, y, z) \mapsto \tilde p := (x, y, 0) \cong (x, y).
  \end{equation*}
\end{defn}

Two planes $\pi_1, \pi_2$ that are not parallel to the $z$-axis can be
represented by equations 
$z = u_i x + v_i y + w_i$ with $i = 1, 2$. They intersect in an
\emph{isotropic angle} defined as 
\begin{equation*}
  \psi(\pi_1, \pi_2) 
  = \sqrt{(u_2 - u_1)^2 + (v_2 - v_1)^2}.
\end{equation*}
\emph{Isotropic congruence transformations} preserve isotropic distances
and angles and are represented by those volume-preserving affine
transformations which appear in the top view as Euclidean congruence
transformations.

It should be noted that isotropic geometry can be interpreted as
Cayley-Klein geometry in complex extended 3-dimensional projective space.
The \emph{absolute}, i.e., the object that is left invariant under
isotropic transformations, is the plane at infinity (in homogeneous
coordinates $x_0 = 0$) and the pair of conjugate complex lines (in
homogeneous coordinates $(x_1 + i x_2) (x_1 - i x_2) = 0$) in the plane at
infinity.

In Cayley-Klein geometry, spheres are not always defined by a metric but
also as real irreducible quadric containing the absolute. Isotropic unit
spheres therefore come in two types. Expressed in standard form they are
of 
\begin{equation*}
  \text{\emph{parabolic type} }
  z = \frac{1}{2} (x^2 + y^2),
  \quad
  \text{and \emph{cylindrical type} }
  x^2 + y^2 = 1.
\end{equation*}
The latter one is also the set of points with constant isotropic distance
to the origin.

\emph{Isotropic lines} are straight lines which pass trough the absolute
of $I^3$. The only real point of the absolute is $(0 : 0 : 0 : 1)$ in
homogeneous coordinates. Isotropic lines are therefore parallel to the
$z$-axis. The distance between any two points of an isotropic line is
zero. \emph{Isotropic planes} are planes which are also parallel to the
$z$-axis.

\emph{Admissible} surfaces in $I^3$ are regular surfaces without isotropic
tangent planes.
A smooth Monge patch (see, e.g.,~\cite{oprea-2007}), i.e., a surface
parametrization of the form $(u, v) \in U \mapsto F(u, v) = (u, v, f(u, v))$ 
over an open domain $U \subset \R^2$ always describes an admissible
surface as there are no vertical tangent planes. We will often simply
write $(u, v, f)$ for $(u, v, f(u, v))$.

In Euclidean geometry the Gaussian curvature can be defined by the metric
on the surface. Since the metric in $I^3$ induces on every admissible
surface the metric of the plane $\R^2$, the Gaussian curvature vanishes
for every surface.
Consequently, in~\cite{strubecker2} the definition of the Gaussian
curvature for surfaces in $I^3$ is defined as so-called \emph{relative
curvature} with the parabolic unit sphere as relative Gauss image (see
Figure~\ref{fig:wrong-isometry}). The Gauss image is defined by parallel
tangent planes.

\begin{defn}
  \label{defn:gaussimage}
  Let us consider the isotropic sphere of parabolic type $S$ with equation
  $2 z = x^2 + y^2$ as unit sphere. Then the \emph{isotropic Gauss image
  $\sigma(g)$} of an admissible surface $g(u, v) = (u, v, f(u, v))$
  defined by parallel tangent planes is given by 
  \begin{equation*}
    \sigma(g(u, v)) = (f_u, f_v, (f_u^2 + f_v^2)/2).
  \end{equation*}
\end{defn}

For a general admissible surface parametrization 
$g : U\subset \R^2 \to I^3$ the \emph{isotropic Gaussian curvature} (or
\emph{relative curvature}) is given by (see, e.g.,~\cite{Sachs:1990})
\begin{equation}
  \label{eq:general-K}
  K = \frac{
    \det(g_u, g_v, g_{uu})
    \det(g_u, g_v, g_{vv})
    -
    \det(g_u, g_v, g_{uv})^2
  }
  {\tilde g_u^2 \tilde g_v^2 - \la \tilde g_u, \tilde g_v\ra^2}, 
\end{equation}
and the isotropic mean curvature is given by
\begin{equation}
  \label{eq:general-H}
  H = \frac{
    g_u^2 \det(g_u, g_v, g_{vv})
    - 2 \la \tilde g_u, \tilde g_v\ra \det(g_u, g_v, g_{uv}) 
    + g_v^2 \det(g_u, g_v, g_{uu})
  }
  {2 (\tilde g_u^2 \tilde g_v^2 - \la \tilde g_u, \tilde g_v\ra^2)},
\end{equation}
where $\tilde g$ denotes the top view of $g$.
If the surface is given as graph of a function 
$F(u, v) = (u, v, f(u, v))$, the Gaussian and mean curvatures simply read
\begin{equation}
  \label{eq:graphhk}
  K = f_{uu} f_{vv} - f_{uv}^2
  \quad\text{and}\quad
  H = \frac{1}{2} (f_{uu} + f_{vv}).
\end{equation}

Note that isotropic minimal surfaces, i.e., surfaces with vanishing
isotropic mean curvature, correspond to harmonic functions $f$ (which
solve the Laplace equation $\Delta f = 0$).

The \emph{principal curvatures} $\kappa_1, \kappa_2$ are solutions to the
equation
\begin{equation*}
  \kappa^2 - 2 H \kappa + K = 0,
\end{equation*}
and therefore the eigenvalues of the Hessian 
\begin{equation*}
  \nabla^2 f = 
    \begin{pmatrix}
      f_{uu} & f_{uv} \\
      f_{uv} & f_{vv}
    \end{pmatrix}.
\end{equation*}
Its eigenvectors are the \emph{principal directions}.
Since the Gauss, mean and principal curvatures only depend on the second
derivative they are equal to the corresponding curvatures of the
osculating paraboloid. We obtain the osculating paraboloid from the Taylor
expansion up to its second order. It is therefore instructive to compute
the curvatures for a paraboloid.

\begin{figure}[t]
  \begin{overpic}[width=.32\textwidth]{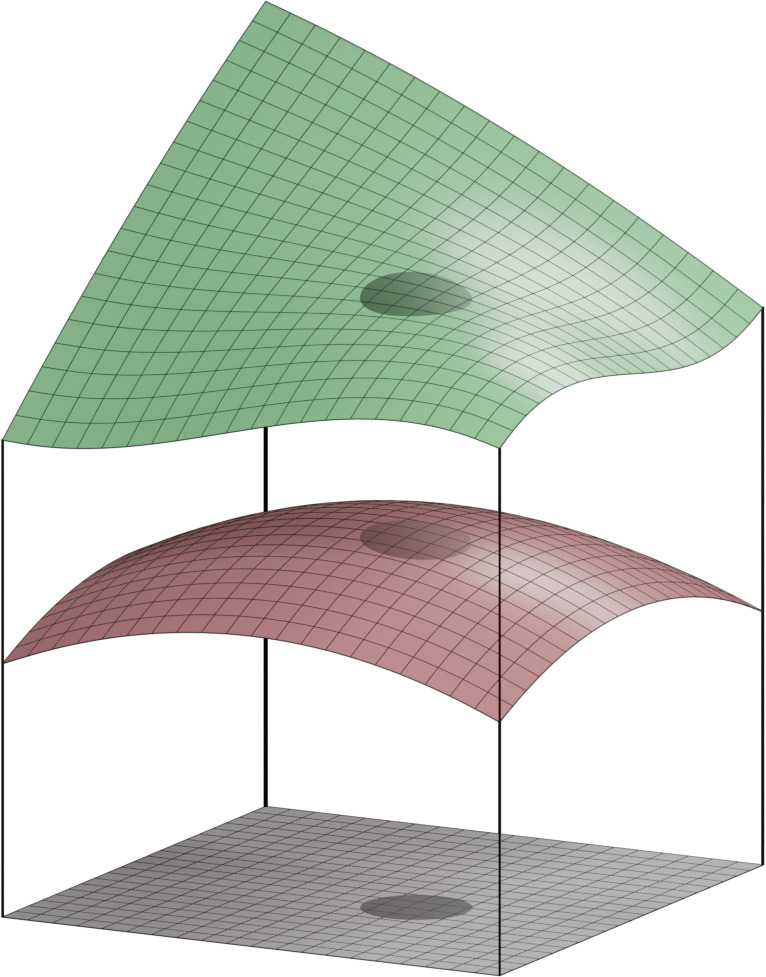}
    \put(0,0){\small(a)}
  \end{overpic}
  \hfill
  \begin{overpic}[width=.67\textwidth]{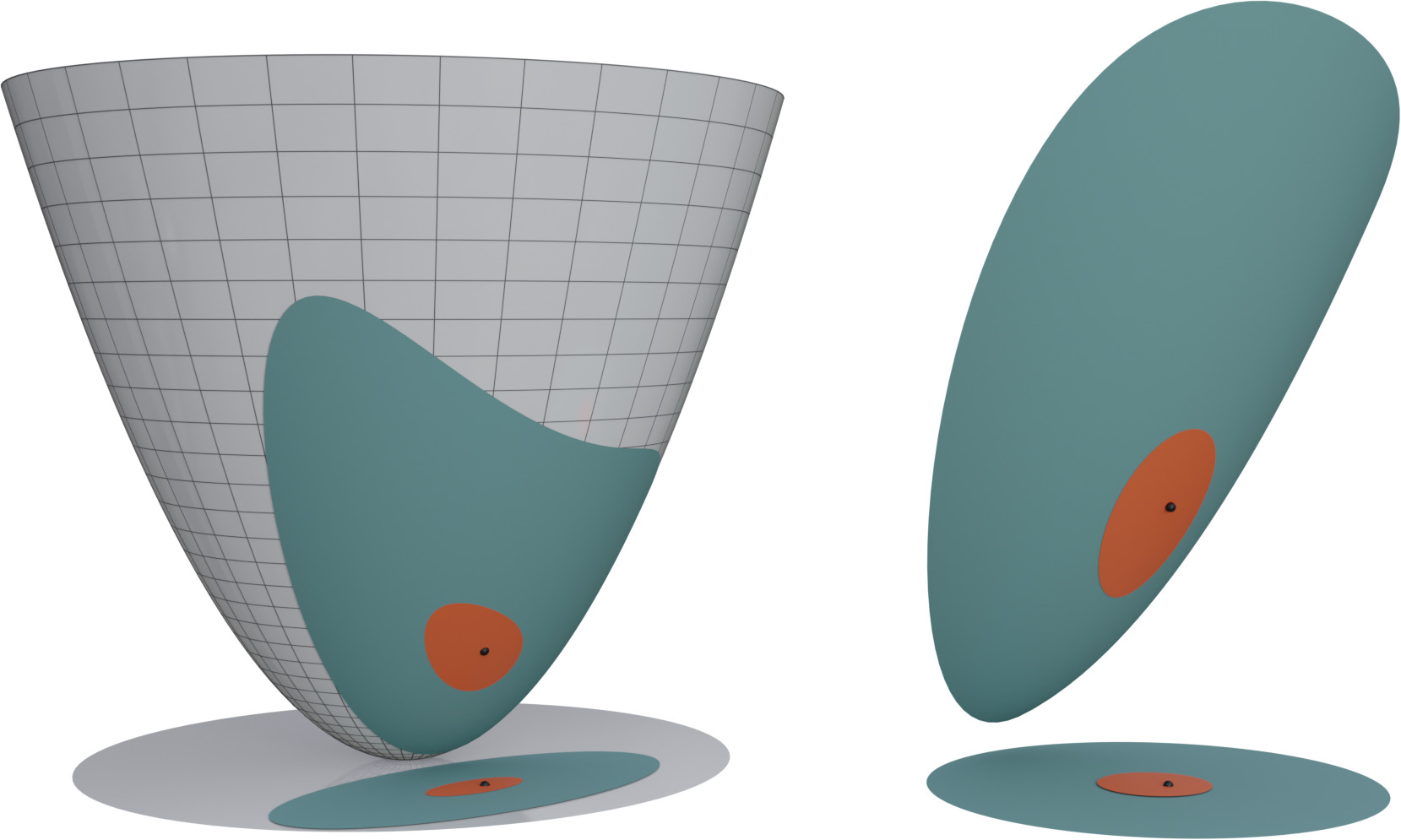}
    \put(85,23){\contour{white}{\small$p$}}
    \put(92,43){\contour{white}{\small$g$}}
    \put(85,3){\contour{white}{\small$\tilde p$}}
    \put(92,3){\contour{white}{\small$\tilde g$}}
    \put(36,13){\contour{white}{\small$\sigma(p)$}}
    \put(36,23){\contour{white}{\small$\sigma(g)$}}
    \put(40,43){\contour{white}{\small$S$}}
    \put(2,0){\small(b)}
    \put(65,0){\small(c)}
  \end{overpic}
  \caption{\emph{(a)} 
  \emph{Metric isometry} of surfaces is not a sensible notion in isotropic
  geometry. As distances are measured in the top view any two surfaces are
  isometric to each other in \emph{just the metrical sense} if their top
  view is congruent. 
  \emph{(b-c)} The isotropic unit sphere $S$ (of parabolic type)
  is a Euclidean paraboloid \emph{(b)}. The Gauss image $\sigma(g)$
  of a surface $g$ \emph{(c)} to $S$ is obtained by the correspondence
  of parallel tangent planes. The ratio of the areas of the top views of
  the red domains around a point $p$ converges to the isotropic Gaussian
  curvature at $p$ as the diameter of the domain goes to zero.}
  \label{fig:wrong-isometry}
\end{figure}
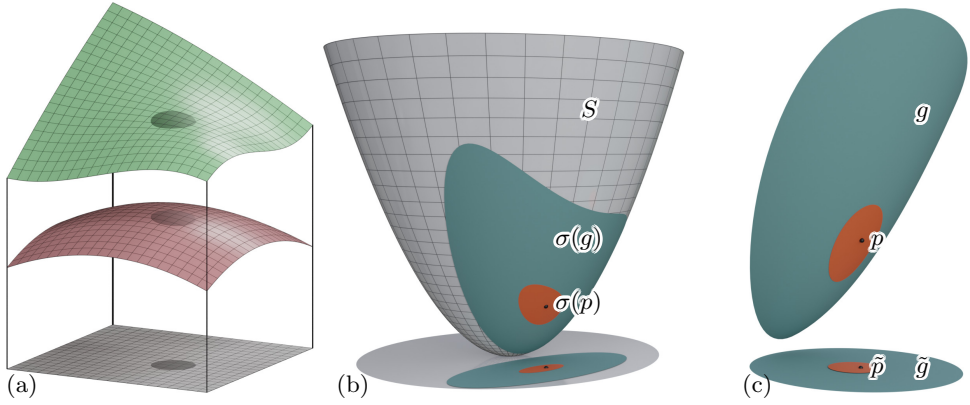

\begin{ex}
  \label{ex:paraboloid}
  Let us consider a paraboloid with equation $z = (a x^2 + b y^2)/2$. Its
  Gaussian and mean curvature are everywhere $K = a b$ and $H = (a + b)/2$.
  The principal curvatures are everywhere $\kappa_1 = a$ and $\kappa_2 =
  b$.
\end{ex}

\subsection{Metric duality and parallel points} 
\label{ssec:metricdual}

Projective correlations in projective $3$-space are linear maps which map
points to planes and vice versa while preserving incidence. Straight lines
are mapped to straight lines. In isotropic space $I^3$ there are
correlations which preserve metric quantities. Such a correlation in $I^3$
is called \emph{metric duality} and can be realized in the following two
ways as discussed below. 

The first realization of metric duality is by the \emph{polarity $\delta$
with respect to the isotropic unit sphere} $S \colon 2 z =
x^2 + y^2$. It maps a point with coordinates $(u, v, w)$ to its polar
plane and vice versa:
\begin{equation*}
 \delta \colon (u, v, w) \longleftrightarrow u x + v y - z - w = 0.
\end{equation*}
Note that a point $p = (u, v, w)$ lies on the plane $\delta(p)$ if and
only if $p$ lies on the isotropic unit sphere $S$.

The second way of realizing metric duality in $I^3$ is by the \emph{null
polarity} $\nu$ which maps a point $(u, v, w)$ to a plane via
\begin{equation}
  \label{eq:nu}
 \nu \colon (u, v, w) \longleftrightarrow  v x - u y + z - w =0, 
\end{equation}
and vice versa.
Note that each point $p$ always lies on its dual plane $\nu(p)$ which
characterizes \emph{null polarities} among correlations.

After projective extension of $I^3$, note that both dualities, $\delta$
and $\nu$, map vertical planes to points at infinity and vice versa.
Furthermore, both metric dualities preserve metric quantities under
dualization as detailed in Lemma~\ref{lem:metricduality}. A proof can be
found in~\cite{Sachs:1990}.

\begin{lem}
  \label{lem:metricduality}
  The metric dualities $\delta, \nu$ map two points $p_1, p_2$ with
  isotropic distance $l$ to two planes with intersection angle $l$ and
  vice versa: 
  \begin{equation*}
    d(p_1, p_2) 
    = \psi(\delta(p_1), \delta(p_2)) 
    = \psi(\nu(p_1), \nu(p_2)).
  \end{equation*}
\end{lem}

Two planes $\delta(p_1), \delta(p_2)$ are parallel if and only if the
corresponding points $p_1, p_2$ lie on the same isotropic line (i.e.,
vertical line). This is the reason why such two points are called parallel
points which turns out to be a useful notion.

\begin{defn}
  \label{def:parallelpoints}
  Two points $p_1, p_2 \in I^3$ are called \emph{parallel} if they lie on
  the same isotropic line.
\end{defn}

Parallel points have the same top view which implies that the isotropic
distance~\eqref{eq:distance} of two parallel points is zero.
In some cases a so called ``replacing invariant'', a pseudo-distance, is
introduced to measure the vertical distance between two parallel points.

\subsection{Contact elements}
\label{subsec:contactelements}

A key notion in our paper is the one of a contact element.
\begin{defn}
  A \emph{contact element} is a pair consisting of a point and an incident
  plane (see Figure~\ref{fig:contact-element}).
\end{defn}
A contact element with a non-isotropic plane can be represented by the
\emph{contact point} $(x, y, z)$ and the Euclidean normal vector $(p, q,
-1)$. We denote such a non-isotropic contact element by the quintuple 
  \begin{equation*}
  (x, y, z, p, q).
  \end{equation*}
Since correlations preserve incidence, they map contact elements to
contact elements. 
\begin{lem} 
  \label{lemma1}
  The metric dualities $\delta$ and $\nu$ map between contact elements in the
  following way:
  \begin{align*}
    \delta \colon \ (u, v, w, p, q) 
    &\longleftrightarrow 
    (p, q, p u + q v - w, u, v),
    \\
    \nu \colon \ (u, v, w, p, q)  
    &\longleftrightarrow
    (q, -p, w - p u - q v, -v, u).
  \end{align*}
\end{lem}
\begin{proof}
  Let us consider the contact element $(u, v, w, p, q)$. This contact
  element consists of the point $(u, v, w)$ and the plane 
  $p x + q y - z - (p u + q v - w) = 0$.
  The metric duality $\delta$ maps the point $(u, v, w)$ to the plane $u x
  + v y - z - w = 0$ whose Euclidean normal vector is $(u, v, -1)$.
  Furthermore, $\delta$ maps the above plane to the point 
  $(p, q, p u + q v - w)$.
  Consequently, the image of the above contact element under $\delta$ is 
  $(p, q, p u + q v - w, u, v)$. The proof for $\nu$ works analogously.
\end{proof}

\begin{SCfigure}[2.9][t]
  \begin{overpic}[width=.25\textwidth]{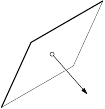}
    \put(40,54){\small$(x, y, z)$}
    \put(60,7){\small$(p, q, -1)$}
  \end{overpic}
  \caption{A contact element is a plane with an incident point 
  $(x, y, z)$. If the Euclidean normal vector of a non-isotropic plane is
  $(p, q, -1)$ then we represent the contact element by the quintuple 
  $(x, y, z, p, q)$.}
  \label{fig:contact-element}
\end{SCfigure}

Let us consider an isotropic contact element $E$. 
A Euclidean rotation about the $z$-axis by the angle of $\pi/2$ and
subsequent Euclidean reflection at the plane $z = 0$ maps the contact
element $\nu(E)$ to the contact element $\delta(E)$.
The rotation is an isotropic (orientation preserving) motion,
whereas the reflection is an orientation reversing congruence
transformation. It reverses the sign of certain signed invariants that we
will not describe in more details. 
\begin{lem}
  \label{lem:topviewoflines}
  A straight line $L$ and its image $\nu(L)$ have parallel top
  views, while $L$ and $\delta(L)$ have orthogonal top views.
\end{lem}
\begin{proof}
  Let $L$ be spanned by two points $p_1 = (u_1, v_1, w_1)$, 
  $p_2 = (u_2, v_2, w_2) \in \R^3$.
  The two planes $\nu(p_1), \nu(p_2)$ have Euclidean normal vectors
  $(-v_1, u_1, -1)$ and $(-v_2, u_2, -1)$, respectively.
  The direction of the line of intersection of these two planes is parallel
  to the cross product of their Euclidean normal vectors
  \begin{equation*}
    \sv{-v_1\\ \hphantom{-}u_1\\ -1} \times \sv{-v_2\\ \hphantom{-}u_2\\ -1} 
    =
    \sv{u_2 - u_1\\ v_2 - v_1\\ *},
  \end{equation*}
  and therefore parallel to $p_2 - p_1$ in the top view. On the other hand
  the Euclidean normal vectors of $\delta(p_1)$ and $\delta(p_2)$ read 
  $(u_1, v_1, -1)$ and $(u_2, v_2, -1)$, respectively. Their cross product
  is
    \begin{equation*}
    \sv{\hphantom{-}u_1\\ \hphantom{-}v_1\\ -1} \times 
    \sv{\hphantom{-}u_2\\ \hphantom{-}v_2\\ -1} 
    =
    \sv{\hphantom{-}v_2 - v_1\\ -(u_2 - u_1)\\ *}.
    \end{equation*}
  Therefore, the top view of the intersection line is orthogonal to $p_2 -
  p_1$.
\end{proof}

\paragraph{Graph representation and its dual}

Typically surfaces are considered as sets of points expressed with
the help of a parametrization, an equation, explicitly or implicitly. 
The so called \emph{dual representation} however describes a surface as
being enveloped by its tangent planes. 
Let us recall the classical fact that any generic two-parameter family of
planes envelopes a surface.
\begin{lem}
  \label{lem:dualrep}
  In $E^3$ let $\tau(u, v) \colon \la x, n(u, v)\ra - d(u, v) = 0$ denote
  the equations of a two-parameter family of planes with Euclidean normal
  vectors $n$ and oriented distance $d/\|n\|$ to the origin. Let
  furthermore $n, n_u, n_v$ be linearly independent.
  Then these planes envelope a surface $e(u, v)$ which is given by
  \begin{equation*}
    e(u, v)
    =
    N^{-1} D,
    \ \text{where}\ 
    N(u, v)
    =
    \begin{pmatrix}
      n, n_u, n_v
    \end{pmatrix}^\top
    \in\R^{3 \times 3}
    \ \text{and}\ 
    D(u, v)
    =
    \begin{pmatrix}
      d
      ,
      d_u 
      ,
      d_v 
    \end{pmatrix}^\top.
  \end{equation*}
  The point representation $e$ is recovered by intersecting the three
  planes $\tau, \tau_u, \tau_v$.
\end{lem}
\begin{proof}
  First of all $e(u, v)$ lies in $\tau(u, v)$, since 
  $\la e, n\ra - d 
  = \la N^{-1} D, n\ra - d 
  = D^\top N^{-\top} n - d
  = D^\top \left(\begin{smallmatrix}1\\0\\0\end{smallmatrix}\right) - d
  = 0$.
  The partial derivative vector of $e$ differentiated by $u$ reads
  $e_u = -N^{-1} N_u N^{-1} D + N^{-1} D_u$ which leads to 
  $\la e_u, n\ra
  = (-D^\top N^{-\top} N_u^{\top} N^{-\top} + D_u^{\top} N^{-\top}) n
  =  -D^\top N^{-\top} N_u^{\top} 
     \left(\begin{smallmatrix}1\\0\\0\end{smallmatrix}\right)
     + 
     D_u^{\top} 
     \left(\begin{smallmatrix}1\\0\\0\end{smallmatrix}\right)
  =  -D^\top N^{-\top} n_u
     + 
     d_u
  =  -D^\top 
     \left(\begin{smallmatrix}0\\1\\0\end{smallmatrix}\right)
     + 
     d_u
  = 0$,
  and analogously $\la e_v, n\ra = 0$,
  which implies that $\tau$ is the tangent plane of $e$.
\end{proof}

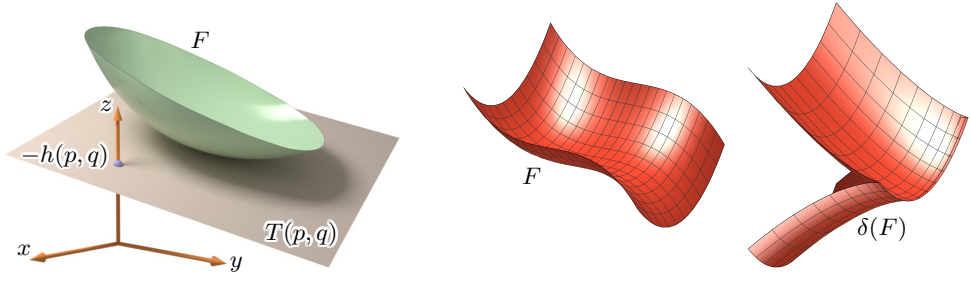
\begin{figure}[t]
  \begin{overpic}[width=.41\textwidth]{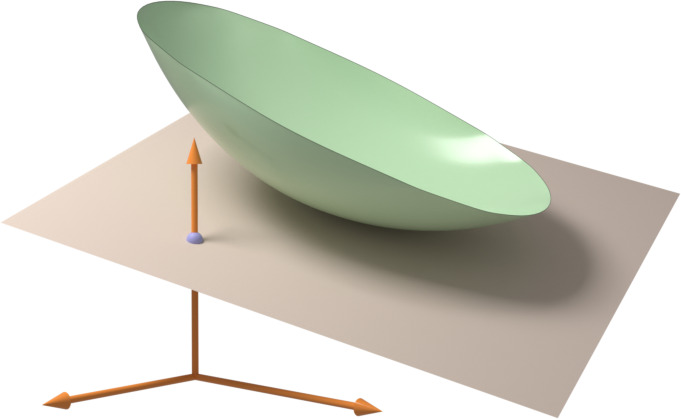}
    \put(47,56){\small$F$}
    \put(4,27){\small\contour{white}{$-h(p, q)$}}
    \put(66,7){\small\contour{white}{$T(p, q)$}}
    \put(3,3){\small$x$}
    \put(57,0){\small$y$}
    \put(24,40){\small\contour{white}{$z$}}
  \end{overpic}
  \hfill
  \begin{overpic}[width=.30\textwidth]{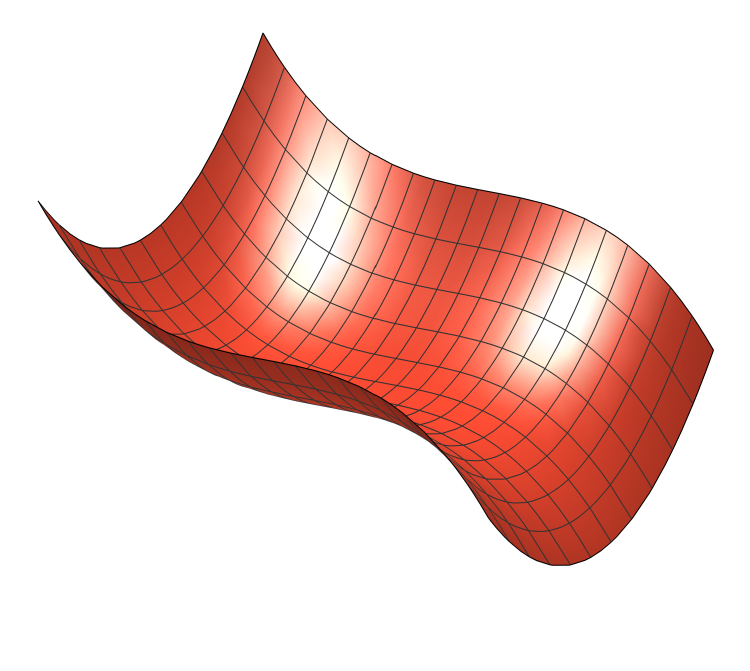}
    \put(25,29){\small$F$}
  \end{overpic}
  \begin{overpic}[width=.23\textwidth]{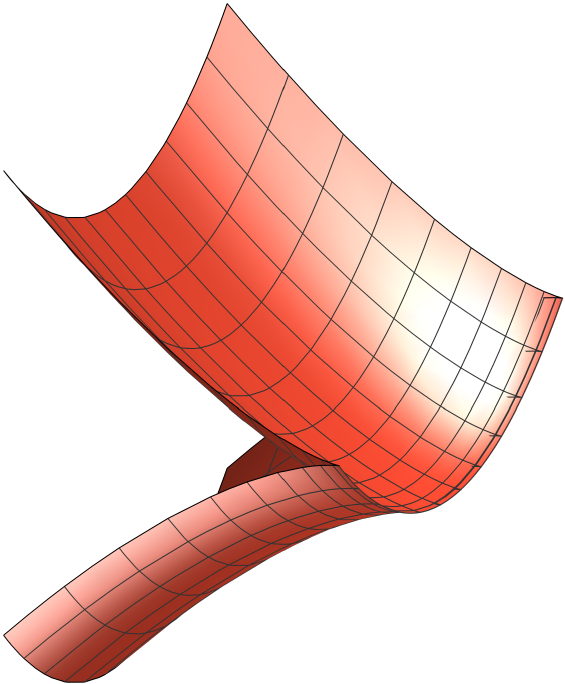}
    \put(41,12){\small$\delta(F)$}
  \end{overpic}
  \caption{\emph{Left:} 
  Illustration of the isotropic support function $h(p, q)$ which measures
  the distance on the $z$-axis between the origin and the intersection of
  the tangent plane at $(p, q)$ with the $z$-axis.
  \emph{Right:}
  A surface $F(u, v) = (u, v, \frac{1}{2} (u^2 - v^3))$ and its metric dual
  $\delta(F)$ parametrized by 
  $(u, -\frac{3}{2} v^2, \frac{1}{2} (u^2 - 2 v^3))$.
  Note that the Gaussian image of $F$ is not regular resulting in a
  cuspidal edge at $\delta(F)$.}
  \label{fig:stuetzfkt}
\end{figure}

Every admissible surface in $I^3$ with non-vanishing Gaussian curvature is
enveloped by the family of its tangent planes 
$T(u, v) \colon z = u x + v y - h(u, v)$ which serves as its dual
representation. Now $(u, v)$ are parameters of the isotropic
Gauss image and $(0, 0, -h(u, v))$ is the intersection point of the $z$-axis
with $T$ (see Figure~\ref{fig:stuetzfkt} left). Since the $z$-axis is an
\emph{isotropic normal} to $T$ we call $h(u, v)$ the \emph{isotropic
support function}.
We obtain its point representation $F(u, v)$ by applying
Lemma~\ref{lem:dualrep}:
\begin{equation}
  \label{eq:surfacefromsupp}
  F(u, v) 
  = 
  \left(\begin{smallmatrix}
    u & v & -1\\
    1 & 0 & \phantom{-}0 \\
    0 & 1 & \phantom{-}0
  \end{smallmatrix}\right)^{-1}
  \left(\begin{smallmatrix}
    h\\
    h_u\\
    h_v
  \end{smallmatrix}\right)
  =
  (h_u, h_v, h_u u + h_v v - h).
\end{equation}

\paragraph{Contact elements of a surface}

The equation of the tangent plane of a surface $F(u, v) = (u, v, f(u, v))$
at a parameter $(u, v)$ equals 
\begin{equation}
  \label{eq:tangentplane}
  f_u x + f_v y - z - (f_u u + f_v v - f) = 0.
\end{equation}
The following schematic overview illustrates the relations between surface
parametrizations, tangent planes of the dual representation and their
image under the metric duality $\delta$.

\begin{center}
\begin{tikzcd}[row sep=25, column sep=30]
  (u, v, f) 
  \arrow[r, shift left, "\text{tangent plane}"]
  \arrow[r, <-, shift right, swap, "\text{envelope}"]
  \arrow[d, <->, "\delta"]
  & 
    f_u x + f_v y - z - (f_u u + f_v v -f) = 0
  \arrow[d, <->, "\delta"] 
  \\
  u x + v y - z - f = 0
  \arrow[r, shift left, "\text{envelope}"]
  \arrow[r, <-, shift right, swap, "\text{tangent plane}"]
  & 
  (f_u, f_v, f_u u + f_v v -f) 
\end{tikzcd}
\end{center}

The arrows in the bottom row follow from Lemma~\ref{lem:dualrep} as we 
apply it to $\tau(u, v) \colon u x + v y - z - f = 0$.
Consequently, the construction of contact elements of surfaces commutes
with the metric duality $\delta$. Analogous arguments for $\nu$ imply the
following lemma. See Figure~\ref{fig:stuetzfkt} (right).

\begin{lem}
  The image of a surface $F(u, v) = (u, v, f(u, v))$ under metric duality
  is a surface $\delta(F)$ or $\nu(F)$, where points correspond to tangent
  planes and vice versa. The contact elements consisting of tangent planes
  and contact points of an admissible surface $F$ and its metric
  dual correspond via 
  \begin{equation}
    \label{eq:contact1} 
    \delta\colon (u, v, f, f_u, f_v) 
    \longleftrightarrow
    (f_u, f_v, f_u u + f_v v - f, u, v). 
  \end{equation}
  or
  \begin{equation} 
    \label{eq:contact2} 
    \nu\colon (u, v, f, f_u, f_v) 
    \longleftrightarrow
    (f_v, -f_u, f - f_u u - f_v v, -v, u). 
  \end{equation}
\end{lem}

A parametrization $f(u, v)$ of a surface is called \emph{conjugate} or
\emph{conjugate curve network} or a \emph{Q-net} if at each point $f_{uv}
\in \myspan(f_u, f_v)$. 
A parametrization $f(u, v)$ of a surface is called \emph{asymptotic} or
\emph{asymptotic curve network} or an \emph{A-net} if at each point
$f_{uu}, f_{vv} \in \myspan(f_u, f_v)$. 
Projective dualities map Q-nets to Q-nets and A-nets to A-nets.

\begin{defn}
  Two surfaces $F$ and $\bar F$ are related by a \emph{Weingarten
  transformation} if the straight line connecting $F(u, v)$ and 
  $\bar F(u, v)$ in corresponding points lies in both tangent planes and
  if in this correspondence any conjugate curve network on $F$ is
  associated with a conjugate curve network on $\bar F$.
\end{defn}

\begin{cor}
  \label{cor:weingarten}
  Any surface $F$ and its metric dual $\nu(F)$ are related by a Weingarten
  transformation.
\end{cor}

Every quintuple $E = (x, y, z, p, q)$ corresponds to a contact element but
not every two-parameter family of such quintuples $E(u, v)$ describe the
contact elements of a surface. The following lemma characterizes the
condition.

\begin{lem}
  Let $E(u, v) = (x, y, z, p, q)$ be a (sufficiently smooth) two-parameter
  family of contact elements. Then the contact elements $E$ represent
  tangent plane and contact point of a surface parametrized by 
  $(x(u, v), y(u, v), z(u, v))$ if and only if 
  \begin{equation}
    \label{eq:integrability}
    p_v x_u + q_v y_u = p_u x_v + q_u y_v.
  \end{equation}
\end{lem}
\begin{proof}
  For $E$ to describe contact elements of a surface the vector 
  $(p, q, -1)$ must be orthogonal to the tangent plane. Therefore, we must
  have 
  $
  \la 
  \Big(
  \begin{smallmatrix}\phantom{-}p\\ \phantom{-}q\\ -1\end{smallmatrix}
  \Big),
  \Big(
  \begin{smallmatrix}x_u\\ y_u\\ z_u\end{smallmatrix}
  \Big)
  \ra 
  =
  \la 
  \Big(
  \begin{smallmatrix}\phantom{-}p\\ \phantom{-}q\\ -1\end{smallmatrix}
  \Big),
  \Big(
  \begin{smallmatrix}x_v\\ y_v\\ z_v\end{smallmatrix}
  \Big)
  \ra 
  = 0$, which yields equations
  \begin{equation}
    \label{eq:system}
    z_u = p x_u + q y_u,
    \qquad
    z_v = p x_v + q y_v.
  \end{equation}
  This system of PDEs is integrable if and only if
  \begin{equation*}
    (p x_u + q y_u)_v = (p x_v + q y_v)_u,
  \end{equation*}
  which is equivalent to Equation~\eqref{eq:integrability}.
\end{proof}

\paragraph{Minkowski sum and its dual}

Let us consider two surfaces $F_1, F_2$ in Euclidean space which are
parametrized such that their tangent planes are parallel at the same
parameter $(u, v)$.
Let us furthermore denote their support functions by $h_1$
and $h_2$. Then the \emph{Minkowski sum} of $F_1$ and $F_2$ has the
support function $h_1 + h_2$ (see Figure~\ref{fig:minkowskisum} left). 

We translate this property into isotropic space where we consider two
admissible surfaces $F_1, F_2$ together with their isotropic support
functions $h_1, h_2$ which measure the height of the intersection point of
the tangent plane with the $z$-axis.
In analogy to the Euclidean case, we also call the surface that is given
by the isotropic support function $h_1 + h_2$ \emph{Minkowski sum} (see
Figure~\ref{fig:minkowskisum} center). The new surface can be computed
with the formula in Lemma~\ref{lem:dualrep}.

More generally we will consider surfaces $F^t$ corresponding to the
support function $h_1 + t h_2$ and also call the result a \emph{Minkowski
sum} (or \emph{general offset}) with parameter $t$. 
If translations are applied to the two surfaces $F_1, F_2$ then the
Minkowski sum of the two surfaces also undergoes a translation.

While the classical Minkowski sum is based on the parallelism of tangent
\emph{planes}, in $I^3$ we can ``dualize'' the Minkowski sum and base it
upon parallelism of \emph{points} (see
Definition~\ref{def:parallelpoints}). 
Hence, for two surfaces represented as graphs $F_i(u, v) = (u, v, f_i(u,
v))$, the surfaces $F^t(u, v) = (u, v, f_1(u, v) + t f_2(u, v))$ are the
dual counterparts to Minkowski sums, and we simply call them \emph{sum
surfaces} with parameter $t$ (see Figure~\ref{fig:minkowskisum} right).
In particular, $(u, v, f_1 - f_2)$ is called \emph{difference surface}.

The sum of two surfaces undergoes a $z$-parallel shearing 
if the individual surfaces undergo such shearing transformations.
These shearings, which are congruence transformations in $I^3$, are the
dual counterparts to the translation appearing at Euclidean Minkowski
sums. 

A classical Euclidean offset of a surface $f$ is obtained by moving each
point of $f$ along its normal by a constant distance, say $t$.
Equivalently, the offset is the surface given by the support function $h +
t \cdot 1$ where $h$ is the support function of $f$. Note that the support
function of the Euclidean unit sphere is constant $1$. Let us now
translate the notion of offset into isotropic geometry.

The isotropic analogon of an offset is therefore given by a support
function of the form $h + t h_s$ where $h$ denotes the support function of
the given surface $f$ and $h_s$ the support function of the unit sphere.
However, in isotropic space, taking the sum of surfaces is possible in
two ways, in the setting of parallel planes and parallel points.

In the first case we compute the tangent plane of the unit sphere of
parabolic type with equation $2 z = x^2 + y^2$ which is $u x + v y - z -
\frac{1}{2} (u^2 + v^2) = 0$ (cf.\ Equation~\eqref{eq:tangentplane}).
Its support function $h_s$ therefore reads $h_s(u, v) = (u^2 + v^2)/2$.
The \emph{plane based offset} $F^t$ is the surface represented by the
isotropic support functions $h(u, v) + t (u^2 +v^2)/2$.

In the second case the \emph{point based offsets} is given as sum surface
of $f$ and $t$ times the unit sphere, i.e., $(u, v, f(u, v) + t (u^2 +
v^2)/2)$.

\begin{figure}[t]
  \begin{overpic}[width=.42\textwidth]{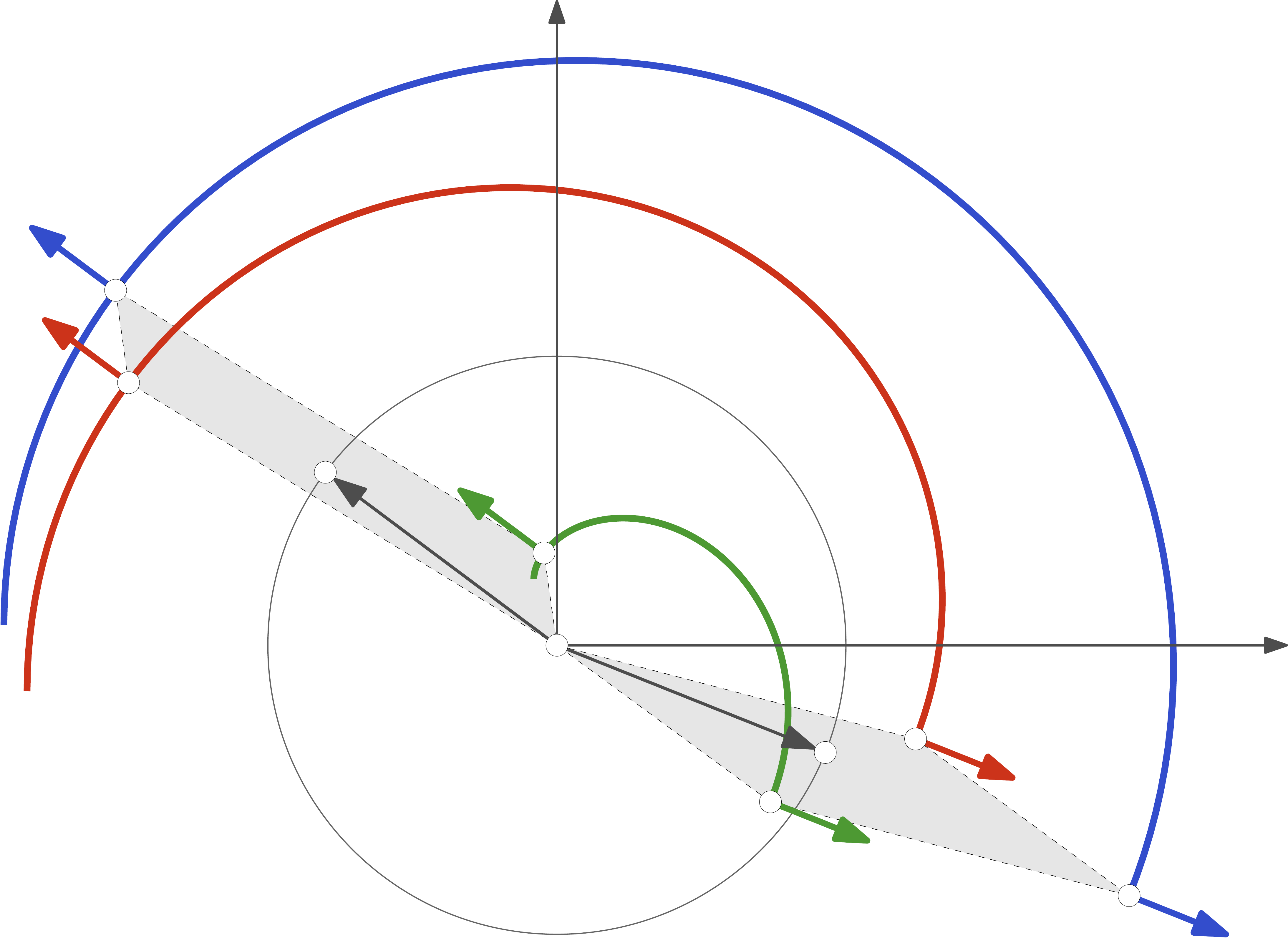}
    \put(65,48){\small$F_1$}
    \put(49,33){\small$F_2$}
    \put(70,61){\small$F_1\! +\! F_2$}
    \put(42,1){\small$S^1$}
  \end{overpic}
  \begin{overpic}[width=.35\textwidth]{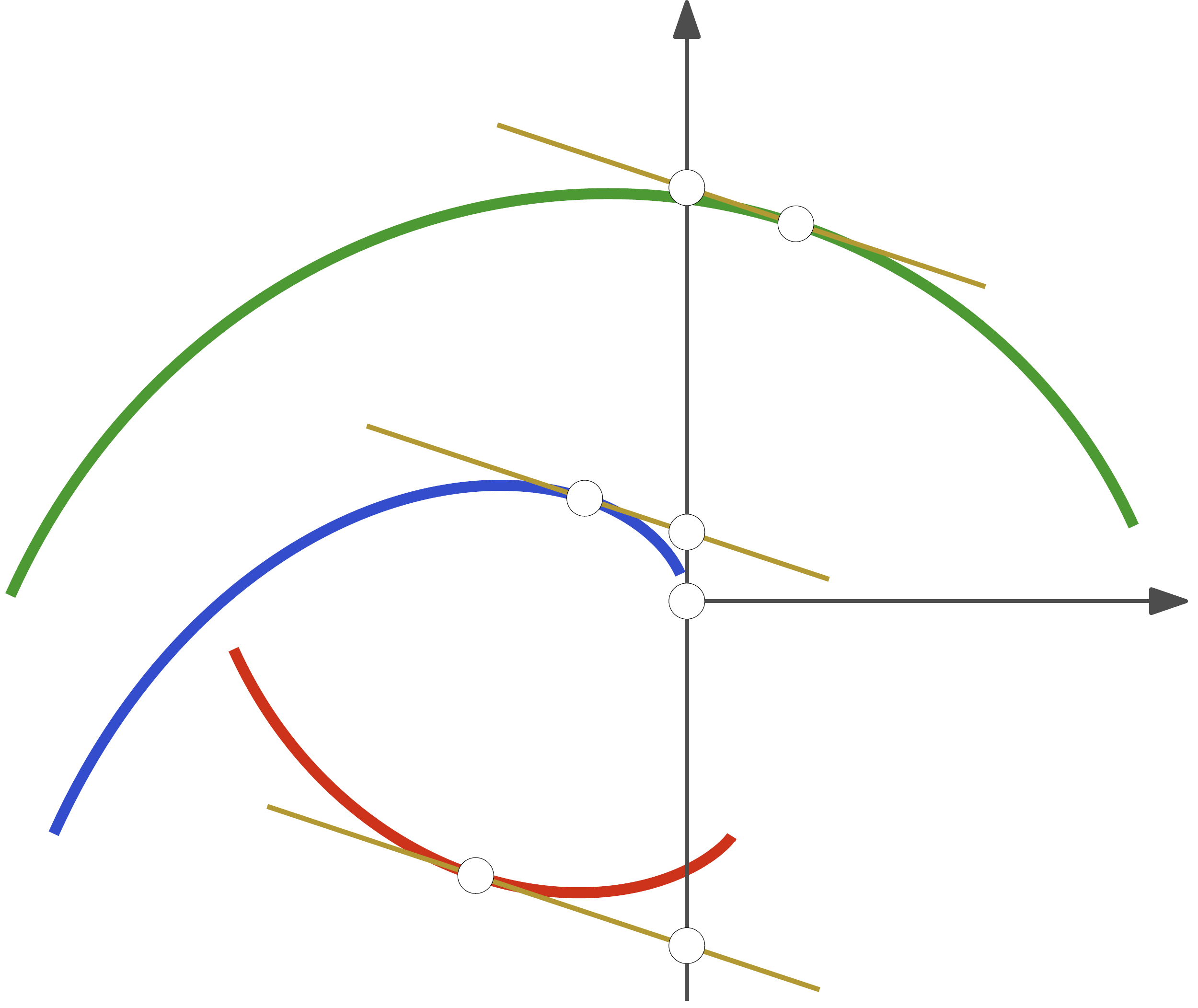}
    \put(62,12){\small$F_1$}
    \put(14,61){\small$F_2$}
    \put(30,35){\small$F_1\! +\! F_2$}
    \put(44,0){\small$-h_1$}
    \put(58,70){\small$-h_2$}
    \put(58,41){\small$-h_1\! -\! h_2$}
  \end{overpic}
  \begin{overpic}[width=.21\textwidth]{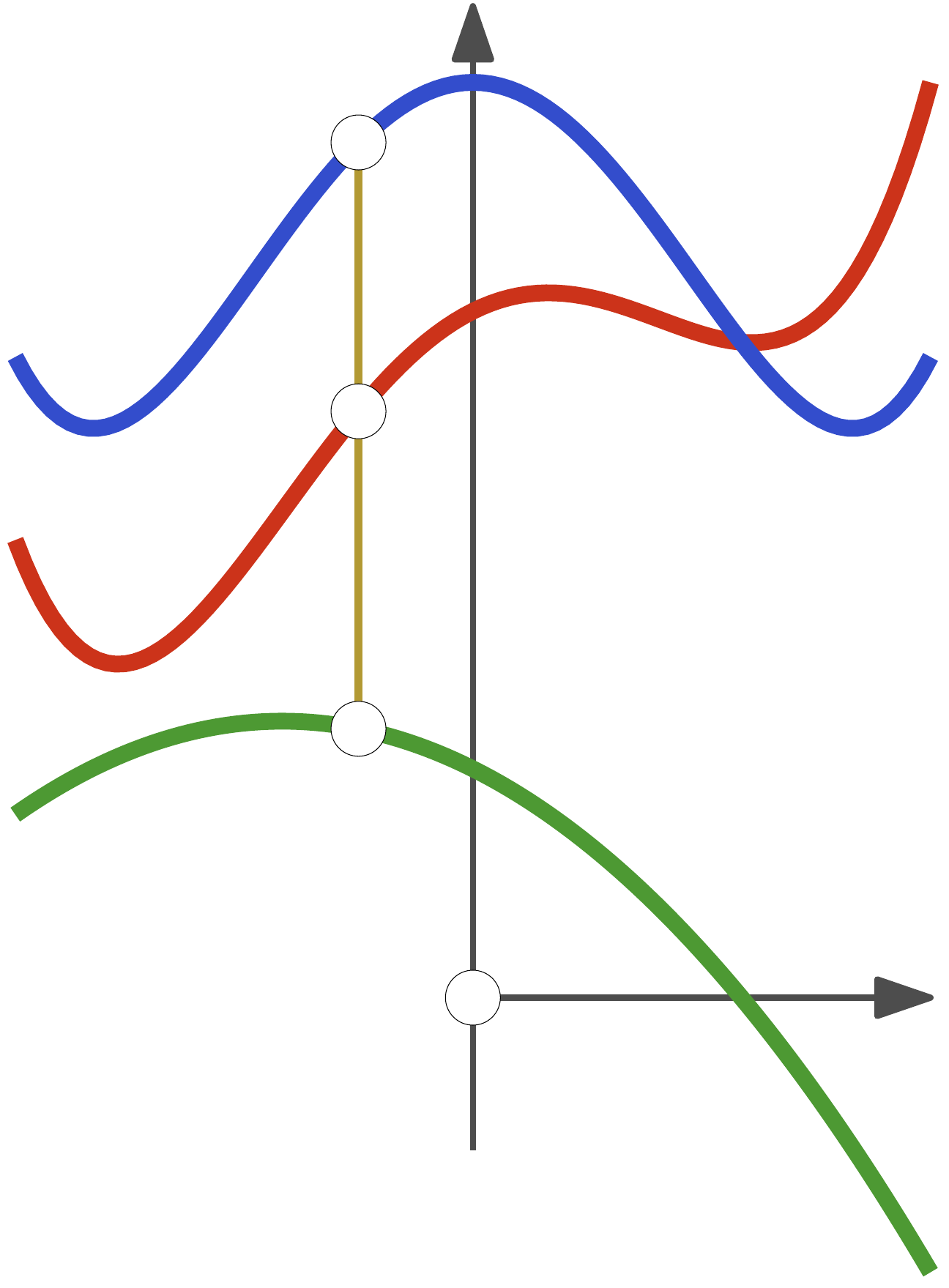}
    \put(60,86){\small$F_1$}
    \put(51,31){\small$F_2$}
    \put(-2,90){\small$F_1\! +\! F_2$}
  \end{overpic}
  \caption{Three types of Minkowski sums illustrated by curves.
  \emph{Left:} \emph{Euclidean} Minkowski sum $F_1 + F_2$ of two curves
  $F_1$ (red) and $F_2$ (green) obtained by adding points with parallel
  tangent planes and parallel normal vectors. \emph{Center:}
  \emph{Isotropic} Minkowski sum $F_1 + F_2$ of two curves $F_1$ (red) and
  $F_2$ (green). The support function of the Minkowski sum is the sum of
  the support functions of $F_1$ and $F_2$. \emph{Right:} The sum of
  functions is the dual of the Minkowski sum in the setting of
  point-parallelism.}
  \label{fig:minkowskisum}
\end{figure}
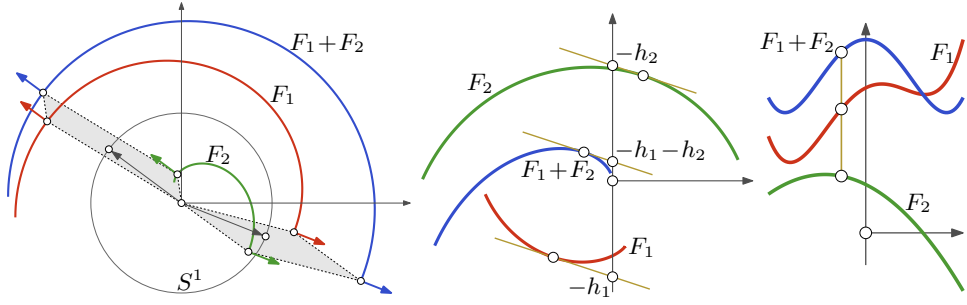

\subsection{Curvatures of surfaces in $I^3$}

The isotropic Gaussian curvature plays a crucial role in our definition of
isometric surfaces. Therefore, we first take a look at representations
of curvature notions of surfaces and their Minkowski sums as well as how
they change under dualities.

\paragraph{Change of curvatures under metric duality}

To see how curvatures transform under metric dualities $\delta$ and $\nu$,
it is sufficient to consider the transformation of an osculating
paraboloid $P(u, v) = (u, v, f(u, v)) = (\kappa_1 u^2 + \kappa_2 v^2)/2$
which has constant isotropic principal curvatures $\kappa_1, \kappa_2$
everywhere (cf.\ Example~\ref{ex:paraboloid}). 
Its contact elements have the form 
$E = (u, v, (\kappa_1 u^2 + \kappa_2 v^2)/2, \kappa_1 u, \kappa_2 v)$
which under $\delta$ and $\nu$ transform to (see
Equations~\eqref{eq:contact1} and~\eqref{eq:contact2})
\begin{equation*}
  \delta(E) 
  = 
  (\kappa_1 u, \kappa_2 v, \frac{\kappa_1 u^2 + \kappa_2 v^2}{2}, u, v),
  \quad
  \nu(E) 
  = 
  (\kappa_2 v, -\kappa_1 u, -\frac{\kappa_1 u^2 + \kappa_2 v^2}{2}, -v, u).
\end{equation*}
Then, the equations of the transformed paraboloids read
\begin{equation*}
  \delta(P) \colon 2 z = \frac{1}{\kappa_1} x^2 + \frac{1}{\kappa_2} y^2,
  \qquad
  \nu(P) \colon -2 z = \frac{1}{\kappa_2} x^2 + \frac{1}{\kappa_1} y^2.
\end{equation*}
They are related to each other by the Euclidean $\pi/2$ rotation and
the Euclidean reflection about the $xy$-plane as mentioned above. This
shows the following lemma.

\begin{lem}
  The curvature measures $\kappa_1, \kappa_2, K, H$ change under the
  metric dualities $\delta, \nu$ according to
  \begin{equation}
    \label{eq:dualcurv}
    \delta \colon 
    (\kappa_i, K, H) \to 
    \big(\frac{1}{\kappa_i}, \frac{1}{K}, \frac{H}{K}\big), 
    \quad
    \nu \colon 
    (\kappa_i, K, H) \to
    \big(-\frac{1}{\kappa_i}, \frac{1}{K}, -\frac{H}{K}\big). 
  \end{equation}
\end{lem}

\paragraph{Curvatures in terms of the support function}

The isotropic Gaussian and mean curvatures have simple representations in
terms of their support function (Prop.~\ref{prop:curvature}). We could
simply obtain these new formulas by computing them using
formulas~\eqref{eq:general-K} and~\eqref{eq:general-H} applied to the
surface parametrization~\eqref{eq:surfacefromsupp} in terms of its support
function. However, note the beautiful argument in the following proof
which relies upon metric duality and the curvature transformation just
obtained in~\eqref{eq:dualcurv}.

\begin{prop}
  \label{prop:curvature}
  Let $F$ be given by its support function $h(u, v)$.
  Then its Gaussian and mean curvature are
  \begin{equation*}
    K = \frac{1}{h_{uu} h_{vv} - h_{uv}^2},
    \qquad
    H = \frac{h_{uu} + h_{vv}}{h_{uu} h_{vv} - h_{uv}^2}. 
  \end{equation*}
\end{prop}
\begin{proof}
  From the point representation~\eqref{eq:surfacefromsupp} of a surface
  $F$ with given support function $h$ we get its contact element
  representation
  $E(u, v) = (h_u, h_v, h_u u + h_v v - h, u, v)$. By
  Equation~\eqref{eq:contact1} the contact element representation of
  $\delta(F)$ is $(u, v, h(u, v), h_u, h_v)$. Thus, curvatures $K^*$ and
  $H^*$ of $\delta(F)$ have the form (see Eqn.~\eqref{eq:graphhk})
  \begin{equation}
    \label{eq:support*}
    K^* = h_{uu} h_{vv} - h_{uv}^2,
    \qquad 
    H^* = h_{uu} + h_{vv}.
  \end{equation}
  The transformation rule~\eqref{eq:dualcurv} for curvatures under
  duality $\delta$ yields the formulae for $H = H^*/K^*$ and $K = 1/K^*$
  via support function $h$ as claimed.
\end{proof}

\begin{cor}
  Isotropic minimal surfaces correspond to support functions $h$ which are
  harmonic functions. The metric duals of minimal surfaces or constant
  Gaussian curvature surfaces have the same properties.
  The metric duals of linear Weingarten surfaces of type $H + c K = 0$ are
  constant mean curvature surfaces.
\end{cor}

\paragraph{Curvatures of Minkowski sums and their duals}

Let $F(u, v) = (u, v, f(u, v))$ and $G(u, v) = (u, v, g(u, v))$ be two
admissible surfaces. We start with the point-based Minkowski sum
$F^t(u, v) = (u, v, f(u, v) + t g(u, v))$.
Its mean and Gaussian curvatures are 
\begin{align}
    \nonumber
    H(F^t)
    &=
    H(F) + t H(G), 
    \\
    \label{eq:Ksum}
    K(F^t)
    &= 
    K(F) + 2 t K(F, G) + t^2 K(G), 
\end{align}
where, in analogy to the mixed area, we define $K(F, G)$ to be the
\emph{mixed Gaussian curvature} which reads
\begin{equation}
  \label{eq:mixedK}
  K(F, G) := \frac{1}{2} (f_{uu} g_{vv} - 2 f_{uv} g_{uv} + f_{vv} g_{uu}). 
\end{equation}
The bilinear form $K(\cdot, \cdot)$ corresponds to the quadratic form
expressing $K$ and satisfies $K(F, F) = K(F)$. 
Note that if $G = S$ is the isotropic unit sphere then in the above
expressions for point-based offsets the terms simplify to $H(S) = K(S) =
1$ and $K(F, S) = H(F)$. 

For plane-based Minkowski sums, recall that the support function of $F^t$
is $h_F + t h_G$. Equations~\eqref{eq:support*} imply that Gaussian and mean
curvatures of the dual surfaces are
\begin{equation*}
  K^*(F^t) 
  = K^*(F) + t K^*(F, G) + t^2 K^*(G),
  \qquad
  H^*(F^t) = H^*(F) + t H^*(G),
\end{equation*}
where $K^*(F)$ and $H^*(F)$ measure the Gaussian and mean curvature of the
dual surface, i.e.,  $K^*(F) = K(\delta(F))$.
Consequently, the transformation rule~\eqref{eq:dualcurv} implies
\begin{equation*}
  K(F^t) = \frac{1}{K^*(F^t)},
  \qquad
  H(F^t) = \frac{H^*(F^t)}{K^*(F^t)}.
\end{equation*}

\paragraph{Gaussian curvature and isotropic area of duals}

The total isotropic Gaussian curvature equals the isotropic surface area of
its dual. In the limit of a shrinking domain they are related
via~\eqref{eq:karea}.

\begin{lem}
  Let $F(u, v) = (u, v, f(u, v))$ be an admissible surface. 
  The Gauss curvature $K$ and the isotropic surface area of $\delta(F)$
  are related via
  \begin{equation}
    \label{eq:karea}
    K(u, v) 
    = \lim_{r \to 0} \frac{1}{r^2} \area_{E_r(u, v)}(\delta(F)),
  \end{equation}
  where $\area$ measures the isotropic area (i.e., the Euclidean area of
  the top view) and where $E_r(u, v)$ is the square with side
  length $2 r$ around $(u, v)$.
  The same holds for the metric duality $\nu$.
\end{lem}
\begin{proof}
  Equation~\eqref{eq:contact1} implies that $\delta(F) = (f_u, f_v, f_u u
  + f_v v - f)$.
  The isotropic area of the $\delta(F)$ is therefore measured as
  \begin{equation*}
    \area_{E_r(u, v)}(\delta(F))
    =
    \int_{E_r(u, v)} 
    \det\Big(\begin{smallmatrix}
      f_{uu} & f_{uv} \\
      f_{vu} & f_{vv} 
    \end{smallmatrix}\Big)
    \dd u\, \dd v
  \end{equation*}
  By the mean value theorem for integrals there is $\xi_r \in E_r(u, v)$
  such that
  \begin{equation*}
    \lim_{r \to 0} \frac{1}{r^2} \area_{E_r(u, v)}(\delta(F))
    =
    \lim_{r \to 0} \frac{1}{r^2} r^2 (f_{uu}(\xi_r) f_{vv}(\xi_r) -
    f_{uv}^2(\xi_r))  
    =
    K(u, v),
  \end{equation*}
  which concludes the proof.
\end{proof}

In the following proof we will need the notion of the mixed area of a
planar domain.

\begin{defn}
  The mixed area for planar domains $U$ bounded by curves $f, g$ given by
  their support functions $h_f, h_g \colon I \to \R$ is given by 
  $\area_U(F, G) = \int_I (h_f h_g - \dot h_f \dot h_g) \dd t$.
\end{defn}

\begin{cor}
  \label{cor:mixed}
  For the mixed Gaussian curvature we have $K(F, G) \equiv 0$ if and only
  if $\area_U(\delta(F), \delta(G)) = 0$ for all open sub domains $U$.
  The same holds for the metric duality $\nu$.
\end{cor}
\begin{proof}
  For all $(u, v)$ in the parameter domain we have
  \begin{align*}
    K(F(u, v), G(u, v)) = 0
    \overset{\eqref{eq:Ksum}}&{\Leftrightarrow}
    \frac{\partial}{\partial t} K(F^t(u, v))\Big|_{t = 0} = 0
    \\
    \overset{\eqref{eq:karea}}&{\Leftrightarrow}
    \frac{\partial}{\partial t} \lim_{r \to 0} 
    \frac{1}{r^2} \area_{E_r(u, v)}(\delta(F^t)) \Big|_{t = 0} 
    = 0.
  \end{align*}
  Since $\area(F^t) = \area(F) + 2 t \area(F, G) + t^2 \area(G)$ the
  previous line is further equivalent to
  \begin{align*}
    \lim_{r \to 0} \frac{1}{r^2} \area_{E_r(u, v)}(F, G) = 0
    &\Leftrightarrow
    \lim_{r \to 0} \frac{1}{r^2} \int_{E_r(u, v)} 
    (h_f h_g - \dot h_f \dot h_g) \dd t = 0
    \\
    &\Leftrightarrow
    h_f(u, v) h_g(u, v) - \dot h_f(u, v) \dot h_g(u, v) = 0,
  \end{align*}
  which is further equivalent to $\area_U(\delta(F), \delta(G)) = 0$ for
  any subset $U$ of the parameter domain.
\end{proof}

\section{Isometric Surfaces in $I^3$} 
\label{sec:isodef}

Since isotropic 3-space $I^3$ is based on the degenerate metric 
$ds^2 = dx^2 + dy^2$ any two surfaces are locally isometric to each other
if we consider isometries only in that metric sense
(see Figure~\ref{fig:wrong-isometry}~(a)).
To obtain a more meaningful notion of isometric maps between surfaces in
$I^3$ we add as a necessary condition an isotropic version of Gauss'
Theorema Egregium (see Definition~\ref{defn:isometric}).

\begin{SCfigure}[2.9][t]
  \begin{overpic}[width=.26\textwidth]{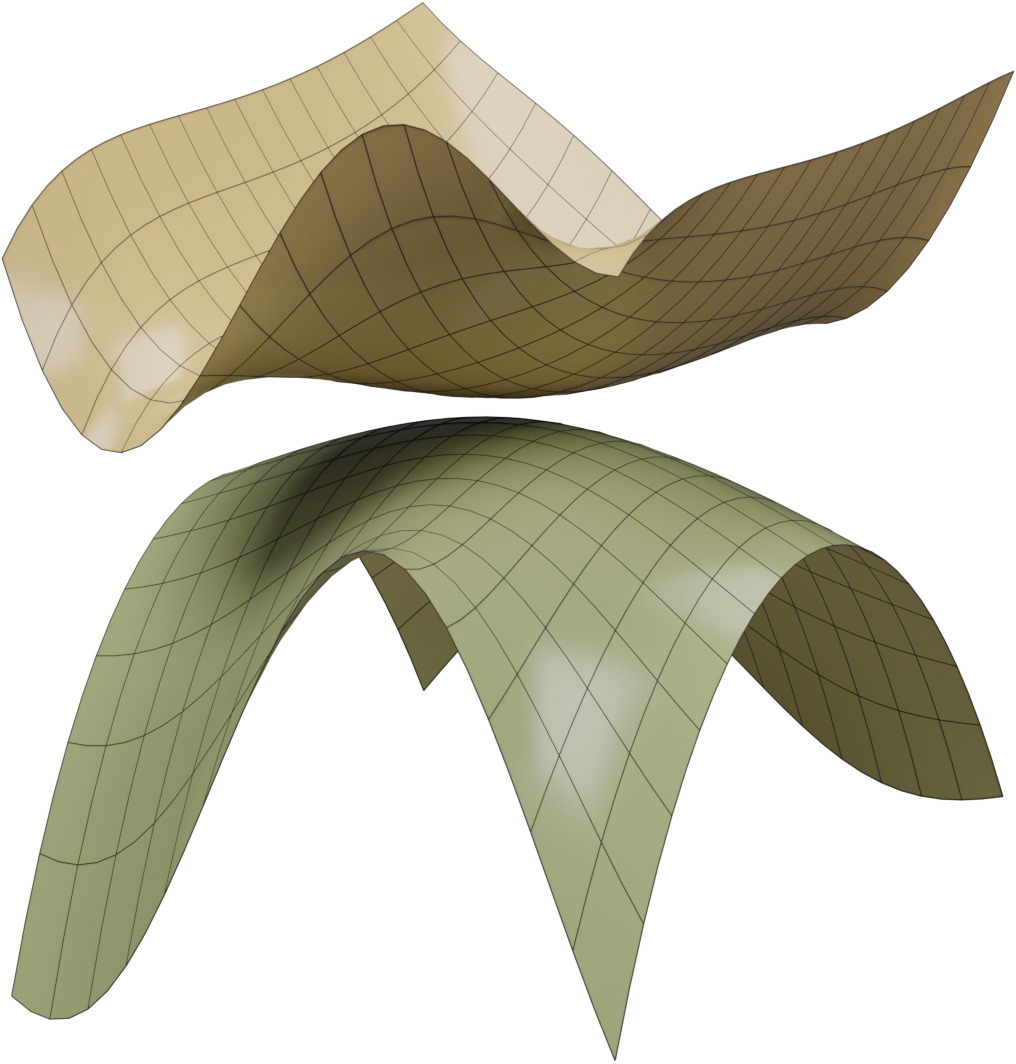}
  \end{overpic}
  \caption{
    Two isometric surfaces. The \emph{upper} surface is given by $(u, v, f
    - h)$ and the \emph{lower} surface is given by  $(u, v, f + h)$ where
    $f(u, v) = (u^2 - v^2 + \cos(1 + u) \cosh(1 + v) + \cosh v \sin u)/10$
    and $h(u, v) = (u^2 + v^2)/6$.
    }
  \label{fig:isometry-from-inf-iso}
\end{SCfigure}

\subsection{Isometric surfaces in isotropic 3-space} 

Since the metric on a surface in $I^3$ is the Euclidean metric in the top
view, an isometric mapping between two admissible surfaces $F, \bar F$
must appear as a Euclidean congruence transformation in the top view.  
We can therefore assume that two isometric surfaces $F$ and $\bar F$ are
arranged in such a way that points which correspond in the isometric
mapping have the same top view. Then the isometric mapping appears 
in the top view as identity. To obtain a sensible definition of a
non-trivial isotropic isometric maping between surfaces we add the natural
condition of equal Gaussian curvature.
\begin{defn}
  \label{defn:isometric}
  Two surfaces $F, \bar F$ with the same top view $D \subset \R^2$
  are called \emph{isometric in $I^3$}, if they have the same Gaussian
  (relative) curvature $K$ at corresponding points, i.e., at points with
  the same top view. 
\end{defn}

Expressing the surfaces as graphs of functions $F(u, v) = (u, v, f(u, v))$
and $\bar F(u, v) = (u, v, \bar f(u, v))$, equality of the isotropic
Gaussian curvature is equivalent to
\begin{equation*}
  f_{uu} f_{vv} - f_{uv}^2 = \bar f_{uu} \bar f_{vv} - \bar f_{uv}^2.
\end{equation*}

The isotropic Gaussian curvature equals the limit of the ratio between the
isotropic area of the Gaussian image (defined by parallel tangent planes)
and the isotropic surface area as the diameter of the domain goes to
zero~\cite{Sachs:1990}. See Figure~\ref{fig:wrong-isometry} (b-c) for an
illustration. Therefore, since the isotropic area stays constant under an
isometric map (or isometric deformation of the surface) also the isotropic
area of its Gauss image must stay constant. Consequently, an isometric map
of a surface induces an \emph{isotropic area preserving transformation of
the Gauss image} on the isotropic unit sphere (measured in the top view).

\paragraph{Metric duals of isometric surfaces}

Let us consider a pair of isometric surfaces $F, \bar F$ with
corresponding points being parallel points. Consequently, under metric
duality $\delta$ or $\nu$, this pair of surfaces is mapped to a pair of
surfaces $F^*, \bar F^*$ with parallel tangent planes at corresponding
points. 

Equation~\eqref{eq:contact1} implies that the metric dual of 
$F = (u, v, f(u, v))$ under $\delta$ reads $\delta(F) = (f_u, f_v, *)$
which yields a top view of the form $(f_u, f_v)$.
The isotropic Gauss image of $F$ is $\sigma(F) = (f_u, f_v, (f_u^2 +
f_v^2)/2)$ (see Def.~\ref{defn:gaussimage}). Therefore, the metric
dual $\delta(F)$ and its Gauss image $\sigma(F)$ have the same top view
$(f_u, f_v)$. The same holds for $\bar F$ whose Gauss image has the same
top view as $\delta(\bar F)$.
For $\nu$, we obtain the analogous result but with the top view of
the Gauss image $\sigma(F)$ and $\sigma(\bar F)$ related to the top view
of the duals $\nu(F)$ and $\nu(\bar F)$ by a rotation about $\pi/2$. 

\begin{prop}
  \label{prop:isometricdual}
  Let $(F, \bar F)$ be a pair of isometric surfaces (with corresponding
  points being parallel points). 
  Both pairs of metric duals $(\delta(F), \delta(\bar F))$ and 
  $(\nu(F), \nu(\bar F))$ have parallel tangent planes and equal Gaussian
  curvatures at corresponding points. The correspondences between
  $\delta(F)$ and $\delta(\bar F)$ as well as between $\nu(F)$ and
  $\nu(\bar F)$ are area preserving maps. 
\end{prop}
\begin{proof}
  The property of parallel tangent planes at corresponding points follows
  from their preimages being parallel points and the preservation of
  ``parallelity'' under metric duality.

  Since $F$ and $\bar F$ are isometric they have the same Gaussian
  curvature $K$ (Def.~\ref{defn:isometric}). The transformation
  rule~\eqref{eq:dualcurv} for curvatures implies that $\delta(F)$ and
  $\delta(\bar F)$ have equal Gaussian curvature $1/K$. The same holds for
  $\nu(F)$ and $\nu(\bar F)$.

  Since the top view of $\delta(F(u, v))$ is 
  $\varphi(u, v) := (f_u(u, v), f_v(u, v))$, the area of the top view
  of $\delta(F)$ is 
  \begin{equation*}
    \textstyle
    \iint 
    \big|\det\Big(
    \begin{smallmatrix}
      \la \varphi_u, \varphi_u \ra
      &
      \la \varphi_u, \varphi_v \ra
      \\
      \la \varphi_u, \varphi_v \ra
      &
      \la \varphi_v, \varphi_v \ra
    \end{smallmatrix}
    \Big)\big|^{1/2}
    \dd u\, \dd v
    =
    \iint |f_{uu} f_{vv} - f_{uv}^2| \dd u\, \dd v
    =
    \iint |K| \dd u\, \dd v,
  \end{equation*}
  which is the same as the area of the top view of $\delta(\bar F)$ since
  its Gaussian curvature is also $K$. The map between $\delta(F)$ and
  $\delta(\bar F)$ is therefore area preserving. The same holds for
  $\nu(F)$ and $\nu(\bar F)$.
\end{proof}

\subsection{The paratactic map} 
\label{ssec:paratactic}

Strubecker~\cite{strubecker4,Sachs:1990} studied an elegant correspondence
between surfaces in $I^3$ and area preserving maps in the plane, known as
paratactic map. 
\begin{defn}
  The \emph{paratactic map} takes contact elements $E = (x, y, z, p, q)$ to
  a pair of points $(E_l, E_r)$ in $\R^2$ via
  \begin{equation*}
    (x, y, z, p, q) 
    \longmapsto 
    \begin{array}{l}
      E_l = (x_l, y_l) = (x + q, y - p),
      \\
      E_r = (x_r, y_r) = (x - q, y + p),
    \end{array}
  \end{equation*}
  the so called \emph{left} and \emph{right image points} $E_l$ and $E_r$.
\end{defn}

Left and right image point $E_l, E_r$ arise through Clifford translations
in $I^3$, a limit case of the well known Clifford translations in elliptic
3-space. Right and left translation of the plane of the contact element
$E$ into the plane $z = 0$ maps the contact point $(x, y, z)$ to the left
and the right image point, respectively. This is another reason for
studying the paratactic map in $I^3$. 

Note that the paratactic map is not bijective. Its inverse is only
determined up to translation in isotropic direction: Suppose we are given
a pair $E_l, E_r$ in $\R^2$ and want to recover a contact element $E = (x,
y, z, p, q)$ as preimage of $E_l$ and $E_r$. Obviously, we must have
$(x, y) = (E_l + E_r)/2$ and $(q, -p) = (E_l - E_r)/2$. However, $z$ is
not encoded in the pair $E_l, E_r$.

A remarkable theorem, going back to Scheffers~\cite{scheffers-1918},
states that \emph{the map $E_l \mapsto E_r$ is area preserving if and only
if the elements $E$ are (up to translation in isotropic direction) the
surface elements of some surface $A$ (which can degenerate in some special
cases)}. 

Given an area preserving map $E_l(u, v) \mapsto E_r(u, v)$, the
contact elements of the surface $A$ are
\begin{equation}
  \label{eq:para2} 
  (x, y, z, p, q) 
  = 
  \big(\frac{x_l + x_r}{2}, \frac{y_l + y_r}{2}, z(u, v), 
       \frac{y_r - y_l}{2}, \frac{x_l - x_r}{2}\big), 
\end{equation}
where $z(u, v)$ is found by integrating the system~\eqref{eq:system}. This
system is integrable since its integrability condition
$p_v x_u + q_v y_u = p_u x_v + q_u y_v$
(Eqn.~\eqref{eq:integrability}) amounts to 
\begin{equation*}
  \partial_u x_l\,
  \partial_v y_l
  -
  \partial_v x_l\,
  \partial_u y_l
  =
  \partial_u x_r\,
  \partial_v y_r
  -
  \partial_v x_r\,
  \partial_u y_r,
\end{equation*}
which is the condition for the area preservation of the map
$E_l(u, v) \mapsto E_r(u, v)$.

Let us now take two isometric surfaces $F, \bar F$.
Proposition~\ref{prop:isometricdual} implies that the top views of
their dual surfaces $\nu(F), \nu(\bar F)$ are related by an area
preserving map. This area preserving map is given by 
$E_l = (f_y, -f_x) \mapsto E_r = (\bar f_y, -\bar f_x)$ which is the paratactic image
of 
\begin{equation*}
  (x, y, z, p, q)
  =
  \frac{1}{2} 
  (f_v + \bar f_v, -f_u - \bar f_u, \ \cdot\ , f_u - \bar f_u, f_v - \bar f_v). 
\end{equation*}

Hence, the corresponding surface $A$ has tangent planes which are parallel
to the scaled difference surface $\Phi(u, v) := (u, v, (f - \bar f)/2)$.
In Sections~\ref{sec:inf} and \ref{sec:darboux}, we will see that the two
surfaces $V$ and $A$ are relative minimal surfaces to each other.
Furthermore, they appear in a sequence of surfaces which form an isotropic
counterpart to Sauer's diagrams and the so called Darboux wreath
associated with infinitesimally flexible surfaces~\cite{sauer:1970}.

\subsection{Isotropic counterparts to well-known isometries in Euclidean
3-space} 
\label{ssec:simpleexamples}

To support the choice of our definition for isotropic isometries and to
illustrate the introduced concepts, we provide here three examples for
isotropic counterparts to well-known Euclidean results: the associated
family of minimal surfaces, Bour's theorem, and Minding isometries of
ruled surfaces.

\paragraph{Associated family of minimal surfaces}

An isotropic minimal surface is the graph $F(u, v) = (u, v, f(u, v)$ of a
harmonic function $f$. These surfaces have been studied by
Strubecker~\cite{strubecker3,strubecker4}, who also provided the isotropic
counterpart to associated families of minimal surfaces. Starting from a
complex differentiable function $f(u + i v) = x(u, v) + i y(u, v)$, the
real harmonic functions $x(u, v)$ and $y(u, v)$ define the associated
family 
\begin{equation*}
  f^t(u, v) = x(u, v) \cos t + y(u, v) \sin t.
\end{equation*}
This is a continuous isometric deformation within our definition, since
Gaussian curvature $K = -(x_{uu}^2 + y_{uv}^2)$ is preserved, as pointed
out by Strubecker~\cite{strubecker3,strubecker4}. See
Figure~\ref{fig:assoc-family} for an illustration of an associated family
of isotropic minimal surfaces with equations
$f^t(u, v) = (\frac{1}{2} \sin u \cosh v + 10) \cos t + \frac{1}{2} \cos u
\sinh v \sin t$.

\begin{figure}[t]
  \begin{overpic}[width=\textwidth]{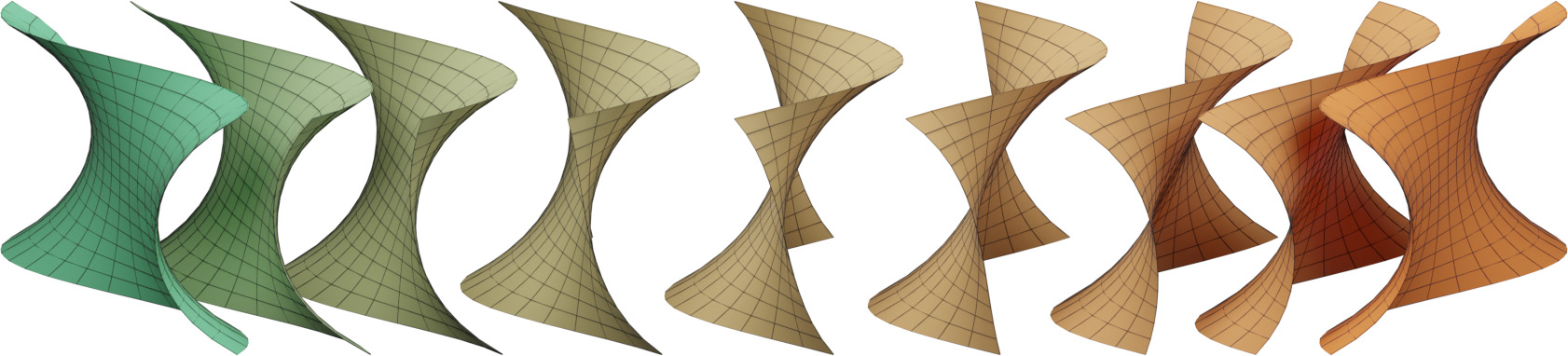}
  \end{overpic}
  \caption{Associated family of isotropic minimal surfaces. For the sake
  of ``paper-space-economy'', the isotropic direction in this image is
  horizontal instead of the usual vertical direction.}
  \label{fig:assoc-family}
\end{figure}

\paragraph{Bour's theorem}

According to E.~Bour, each rotational surface in $E^3$ is isometric to a
one-parameter family of rotational surfaces and a two-parameter family of
helical surfaces (see, e.g.,~\cite{sauer:1970}). We show the analogous
result in $I^3$.

There are ``more'' helical surfaces in $I^3$ than in $E^3$. For example a
shearing of a Euclidean helical surface in $z$-direction results in a
surface that is not a helical surface in $E^3$ any more but still in $I^3$
as it is a congruence transformation in $I^3$. However, every helical
surface in $I^3$ is isotropic isometric to a Euclidean helical surface.

Let $G \subset I^3$ be a Euclidean helical surface parametrized in the form
\begin{equation*}
  g(u, v) = (v \cos u, v \sin u, f(v) + h u).
\end{equation*}
The $v$ parameter curves ($u = {\rm const}$) are isotropic geodesics in
planes through the helical axis (i.e., $z$-axis). The $u$ parameter curves
($v = {\rm const}$) are helices and also called \emph{path curves}.
Rotational surfaces $G$ are also represented in this way with $h = 0$.
  
Bour's isometries map path curves $v = {\rm const}$ to path curves
and profiles to profiles. Thus, it is sufficient to fix the top view and
look for other helical surfaces $\bar G$ with profiles $\bar f(v)$ and
parameter $\bar h$ so that the Gaussian curvature of $\bar G$ and $G$
agree at $(u, v)$.  Inserting into Equation~\eqref{eq:general-K}, we find
\begin{equation*} 
  K = \frac{f' f'' v^3 - h^2}{v^2},
\end{equation*} 
so that equality of the Gaussian curvature $K(G) = K(\bar G)$ for two such
surfaces $G, \bar G$ requires
\begin{equation*} 
  f' f'' - \bar f' \bar f'' = \frac{h^2 - \bar h^2}{v^3}.
\end{equation*} 
Integrating yields 
\begin{equation*} 
  (f')^2 - (\bar f')^2 = \frac{\bar h^2 - h^2}{v^2} - c,
\end{equation*} 
for some integration constant $c \in \R$. Consequently, we find a surface
$\bar G$ isometric to a rotational surface $G$ (which corresponds to $h =
0$) by integration:
\begin{equation*} 
  \bar f 
  = 
  \int \varepsilon(v) \sqrt{(f')^2 - \frac{\bar h^2}{v^2} + c}\, \dd v, 
\end{equation*} 
where $\varepsilon(v) \in \{\pm 1\}$ decides the sign of the square root
but it must be chosen such that $\bar f$ is twice differentiable. Thus,
we have a family of isometric helical surfaces depending on the two
parameters $c$ and $\bar h$, which includes with $\bar h = 0$ a
one-parameter family of isometric rotational surfaces. To obtain real
solutions, $h$ and $c$ must be chosen such that the square root is real. 

\begin{ex}
  A rotational surface together with two isotropic isometric helical
  surfaces are illustrated in Figure~\ref{fig:bour-family}
  with $f(v) = -\sin v$.
  The first helical surface (Fig.~\ref{fig:bour-family} center) is given
  by
  $\bar h = 1$,
  $c_1 \approx 0.045124$,     
  and 
  $\varepsilon_1(v) = \big\{\!\!%
  \def\arraystretch{.65}%
  \begin{array}{ll}
    \scriptstyle
    \sgn(\cos v) 
    & 
    \scriptstyle
    v < 2 \pi
    \\
    \scriptstyle
    1 
    & 
    \scriptstyle
    \text{else}
  \end{array}$.
  The second helical surface (Fig.~\ref{fig:bour-family} right) is given by
  $\bar h = 1$,
  $c_2 = 0.05$,
  and 
  $\varepsilon_2(v) = 1$.
  The second one with constant $\varepsilon$ has a monotonically
  increasing meridian curve.
\end{ex}

In $I^3$ there are further nontrivial one-parameter sub-groups of
congruence transformations, generating surfaces of the form
(see~\cite{Sachs:1990}),
\begin{equation}
  \label{eq:parabolic-rot-surf}
  g(u, v) = a u^2 + b u v + f(v).
\end{equation}
These are \emph{parabolic rotational surfaces} for $a \ne 0$ and Clifford
cylinders for $a = 0, b \ne 0$. They have Gaussian curvature 
$K = g_{uu} g_{vv} - g_{uv}^2 = 2 a f'' - b^2$. Thus isometric surfaces
are characterized by $2 a f'' - b^2 = 2 \bar a \bar f''- \bar b^2$.
Consequently, to any parabolic rotational surface $G$ as
in~\eqref{eq:parabolic-rot-surf}, an isometric surface $\bar G$ has
profile curves given by 
\begin{equation*}
  \bar f(v)
  =
  \frac{a}{\bar a} f(v) + \frac{\bar b^2 - b^2}{4 \bar a} v^2 + c_1 v +
  c_2,
\end{equation*}
with integration constants $c_1, c_2 \in \R$.
These profile curves can be seen as isotropic offsets of scaled profile
curves on the original surface $G$. Parameters $c_1, c_2$ just apply an
isotropic congruence transformation, while the other two parameters
$\bar a \ne 0, \bar b$ are essential. 

Clifford cylinders have constant $K = - b^2$ independent of the profile
and thus any two of them are isometric after appropriate scaling in
$z$-direction.

\begin{ex}
  A pair of isometric parabolic rotational surfaces is depicted in 
  Figure~\ref{fig:parabolic-rot-family}, where 
  $f(v) = 2 \sin v + v^3/70$,
  $a = 1/10$,
  $b = 1/5$,
  $\bar a = 1/5$,
  $\bar b = 3/10$,
  $c_1 = c_2 = 0$.
\end{ex}

\begin{figure}[t]
  \begin{overpic}[width=\textwidth]{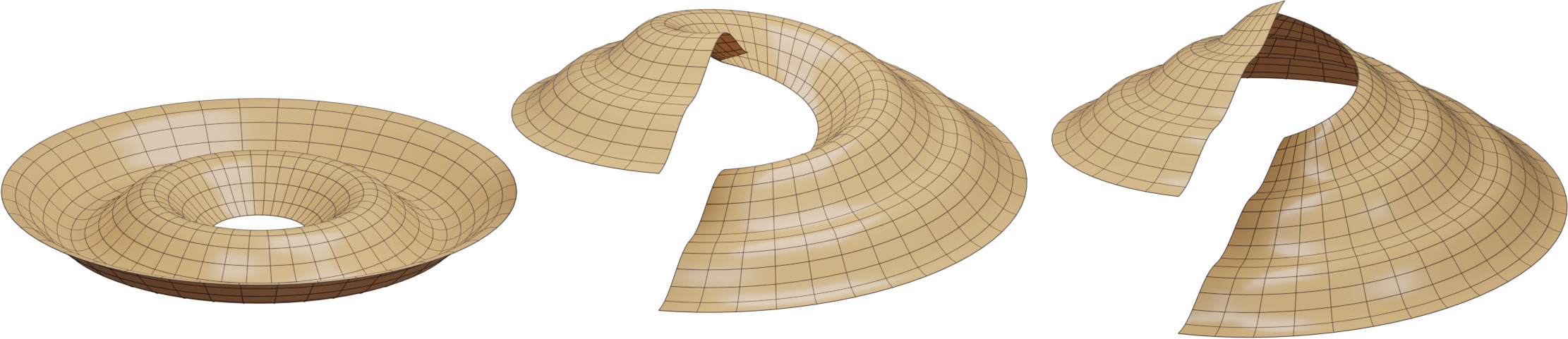}
  \end{overpic}
  \caption{A rotational surface (\emph{left}) together with two isotropic
  isometric helical surface. Note that the meridian curve of the helical
  surface on the \emph{right} is strict monotonically increasing
  (since $\varepsilon_2(v) = 1$). 
  The meridian curve of the helical surface in the \emph{center} is not
  strict monotonically increasing (since $\varepsilon_1(v) \neq
  \text{const}$).}
  \label{fig:bour-family}
\end{figure}

\begin{SCfigure}[2][t]
  \begin{overpic}[width=.38\textwidth]{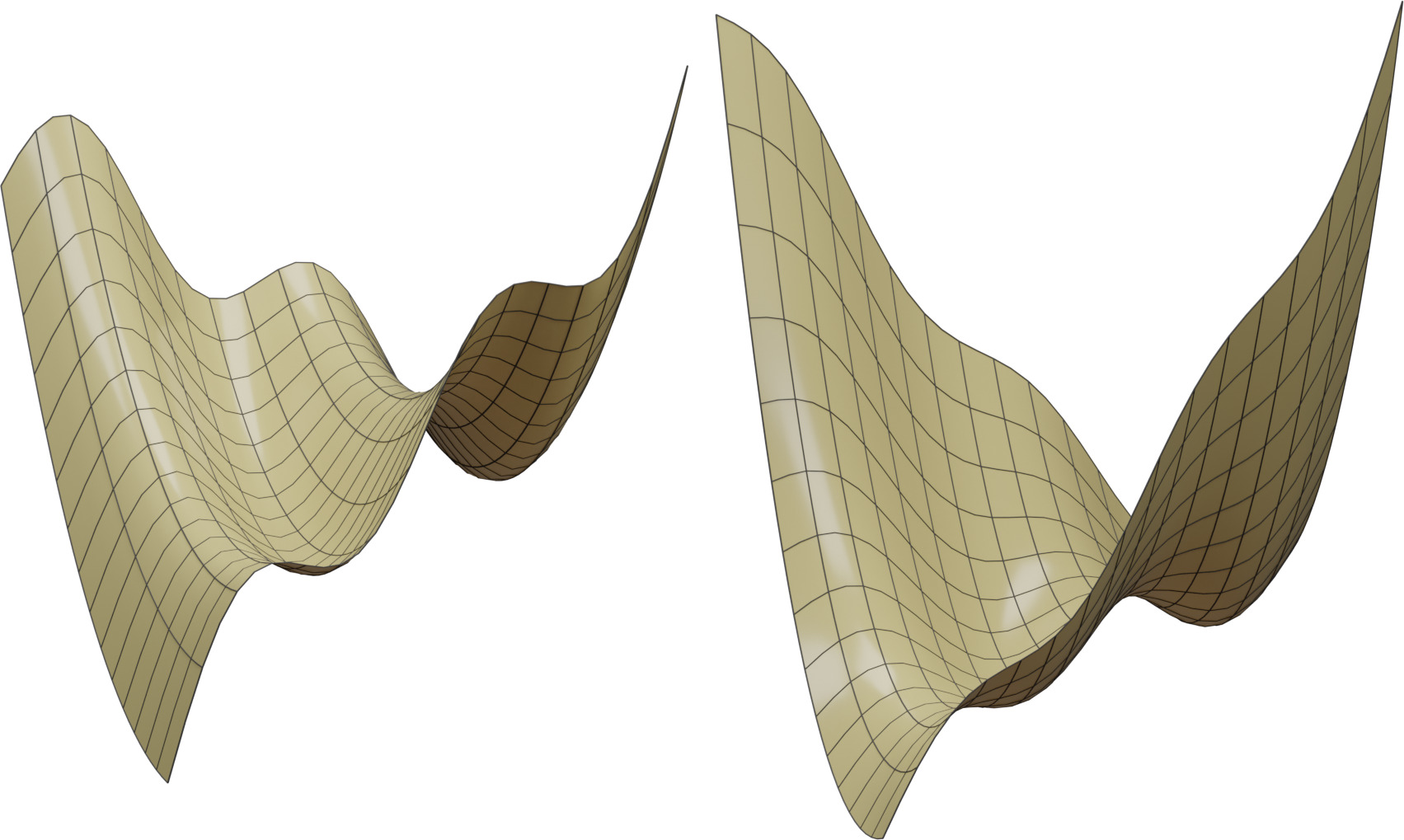}
  \end{overpic}
  \caption{A pair of isometric parabolic rotational surfaces.
  In $I^3$ there is a nontrivial one-parameter sub-group of congruence
  transformations called \emph{parabolic limit rotations}.}
  \label{fig:parabolic-rot-family}
\end{SCfigure}

\paragraph{Minding isometries of ruled surfaces}
\label{par:minding-iso}

Any sufficiently smooth ruled surface in $E^3$ can be embedded into a
continuous family of isometric ruled surfaces such that the rulings
correspond in the isometries. These ruled surfaces share striction and
curvature, but differ in their torsion (see, e.g.,~\cite{pottwall:2001}). 

\begin{defn}
  \label{defn:striction}
  The \emph{striction curve $s$} on a ruled surface in $I^3$ consists of
  those points where the tangent plane is isotropic (vertical).
  The \emph{striction} $\sigma$ is the isotropic angle measured between
  the tangent of the striction curve $s$ and the corresponding ruling and
  thus the difference between their slopes. 
\end{defn}

The striction $\sigma$ is implicitly defined by 
$\dot s/\|\dot{\tilde s}\| = e + \sigma e_3$, where $e$ is the
isotropically normalized ruling direction and where $e_3$ is the unit
vector in $z$-axis direction (see also Figure~\ref{fig:striction}).

In analogy to ruled surfaces in $E^3$, in isotropic space $I^3$ ruled
surfaces are uniquely determined by their curvature $\kappa$, their
torsion $\tau$ and striction $\sigma$~\cite[p.~198]{Sachs:1990} (up to
isotropic congruence transformations).

We are now interested in ruling preserving isometries of ruled surfaces in
$I^3$, not only because they constitute an easily accessible example, but
also due to their occurrence in our discussion of special types of
isometries in Section~\ref{sec:special}. 
We use results by Sachs~\cite{Sachs:1990} on ruled surfaces $F \subset I^3$,
which come in three types. 

For a ruled surface $F$ of \emph{type I}, the top views of the rulings are
tangents of a curve $\tilde s$ (see Fig.~\ref{fig:ruled-surface} left).
The corresponding curve $s \subset F$ is the \emph{striction curve}, along
which the tangent planes of $F$ are isotropic and envelope the isotropic
\emph{striction cylinder}. Then, at a point $p$ of a ruling $r$ which has
isotropic distance $w$ to the striction point $s \in r$, the Gaussian
curvature $K$ is computed as~\cite[p.~210]{Sachs:1990}
\begin{equation}
  \label{eq:lamarle}
  K = -\frac{\rho^2}{w^4}, 
  \qquad\text{with}\qquad
  \rho = \frac{\sigma}{\kappa},
\end{equation}
where $\rho$ is called \emph{pitch}. The same formula applies for ruled
surfaces of \emph{type II}, where the striction curve $s$ is an isotropic
line (see Fig.~\ref{fig:ruled-surface} top-right).
Note that the Gaussian curvature does not depend on the torsion $\tau \in
C^0$. Consequently, for fixed $\kappa$ and $\sigma$ and arbitrary $\tau$
the corresponding ruled surfaces are isometric to each other. We have
therefore an entire family of isometric ruled surfaces in which the
rulings correspond.

\begin{figure}[t]
  \begin{overpic}[width=.56\textwidth]{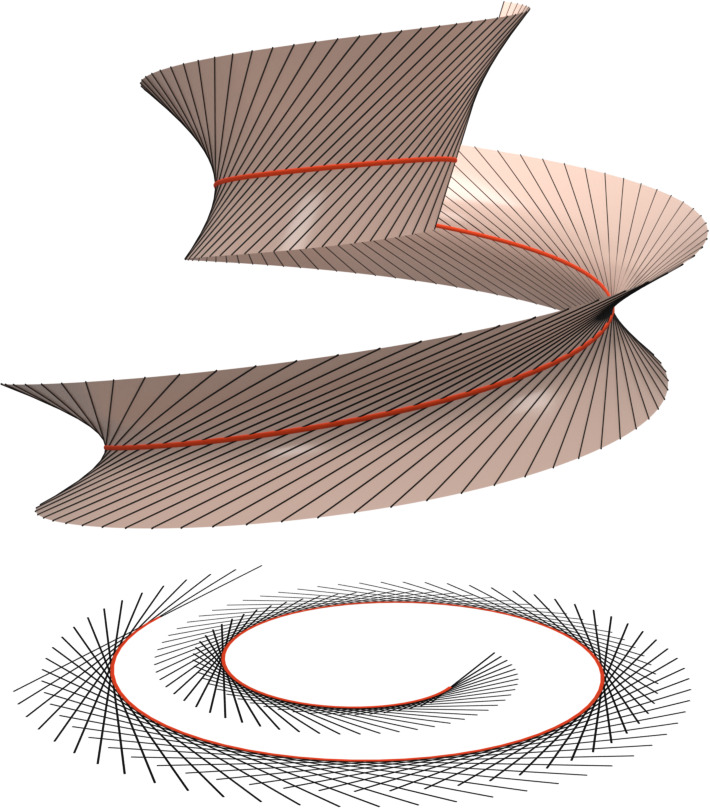}
    \put(70,14){$\tilde s$}
    \put(70,69){\contour{white}{$s$}}
  \end{overpic}
  \hfill
  \begin{minipage}[b]{.42\textwidth}
  \begin{overpic}[width=\textwidth]{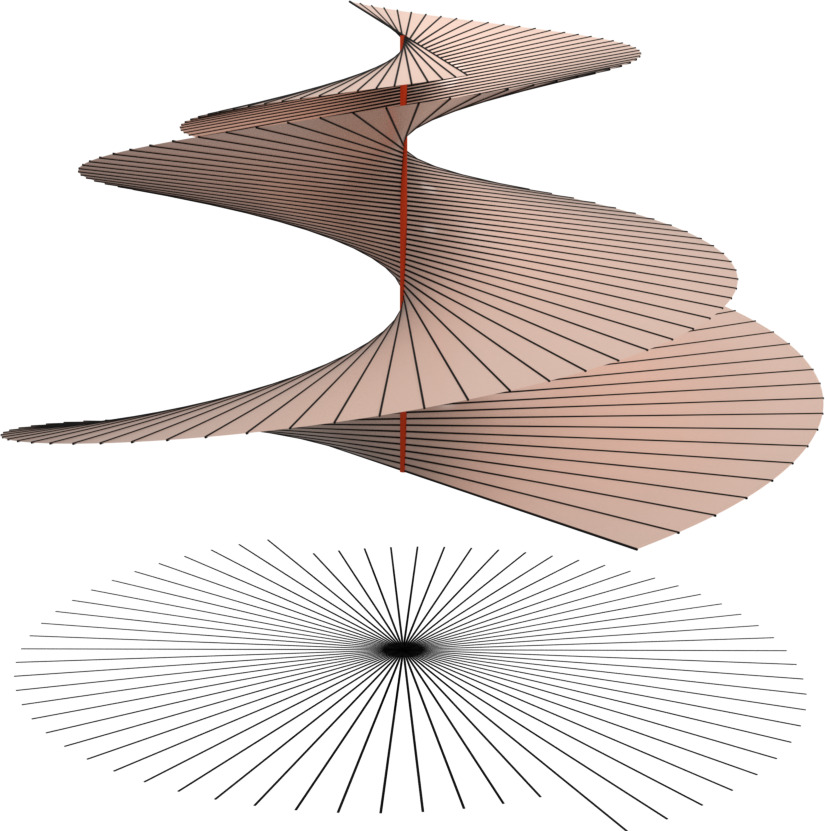}
  \end{overpic}
  \\

  \begin{overpic}[width=\textwidth]{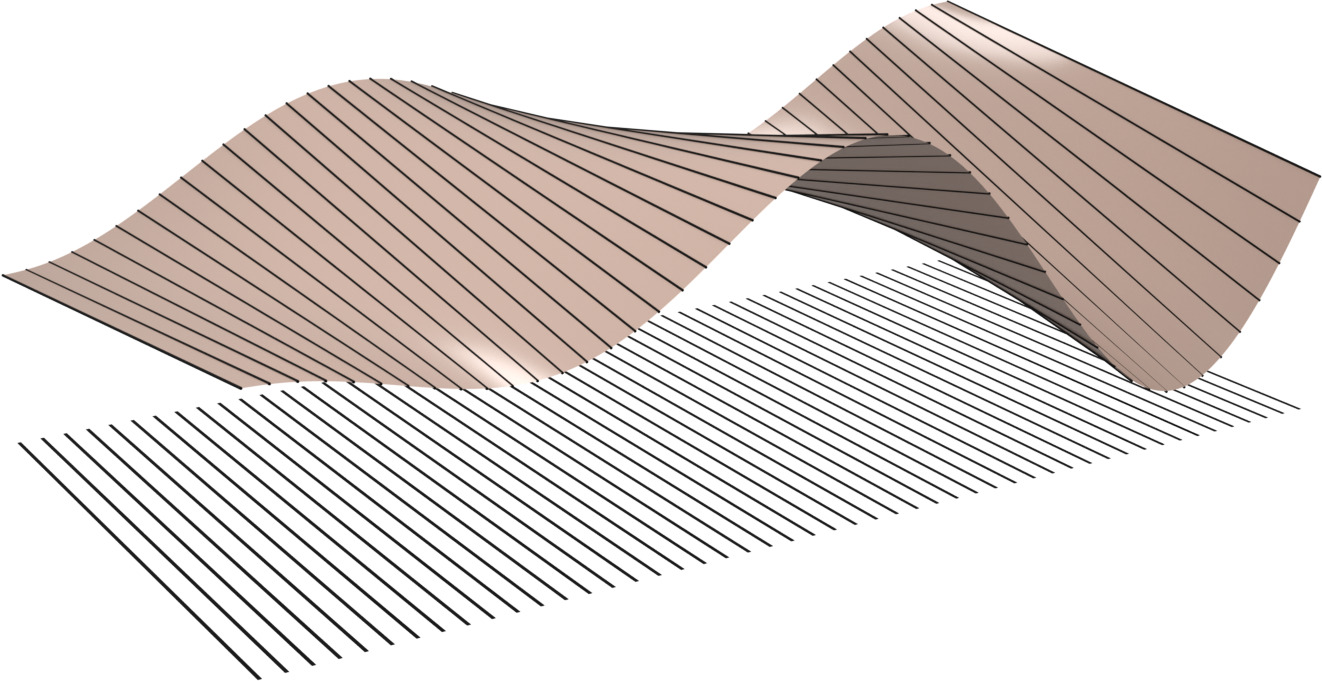}
  \end{overpic}
  \end{minipage}
  \caption{Isotropic ruled surfaces. \emph{Left:} Type I: The top views of
  the rulings are straight lines enveloping a curve $\tilde s$ -- the top
  view of the isotropic striction curve $s$ (orange). The tangent planes
  along the striction curve are isotropic planes. 
  \emph{Top-Right:} Type II: All rulings intersect an isotropic line which
  is also the striction curve. The top views of the rulings intersect in a
  common point. 
  \emph{Bottom-Right:} Type III: All rulings are parallel to a fixed
  isotropic plane. The top views of the rulings form a pencil of parallel
  lines.}
  \label{fig:ruled-surface}
\end{figure}

Minding isometries also exist for ruled surfaces of \emph{type III}, where all
rulings are parallel in the top view (see Fig.~\ref{fig:ruled-surface}
bottom-right). There, the Gaussian curvature is 
$K = -\rho^2$~\cite[p.~210]{Sachs:1990}. The pitch $\rho$
is implicitly given by $\dot e = \rho e_3$, where $e$ is again the
isotropically normalized ruling direction~\cite[p.~204]{Sachs:1990}.

\paragraph{The metric dual of Minding isometries}

The metric dual $F^* = \nu(F)$ of a ruled surface
$F$ is a ruled surface as any projective duality preserves rulings.
Along the striction curve the tangent planes are vertical and are thus
mapped by the duality to points at infinity. 
Thus, in the metric duality we have the following correspondences:
\begin{center}
\begin{tabular}{rcl}
  striction curve 
  &
  $\longleftrightarrow$ 
  &
  curve at infinity
  \\
  striction cylinder 
  &
  $\longleftrightarrow$ 
  &
  asymptotic developable.
\end{tabular}
\end{center}
The \emph{asymptotic developable} is tangent to the ruled surface along
its curve at infinity. 
Note that lines with the same top view correspond in $\nu$ to parallel
lines. The metric dual between two isometric ruled surfaces $F, \bar F$
are two ruled surfaces $F^*, \bar F^*$ with the following relation:
corresponding rulings are parallel, Gaussian curvatures at
corresponding points are equal and the map $F^* \mapsto \bar F^*$ is area
preserving. 
We will continue the discussion of isometric ruled surfaces in
Section~\ref{ssec:2positions}.

\section{Infinitesimal Isometries} 
\label{sec:inf}

A surface $G \colon U\subset \R^2 \to \R^3$ in Euclidean space is
\emph{infinitesimally flexible} if there exists an isometric deformation
of first order, i.e., if there exists a \emph{deformation vector field}
$V \colon U \to \R^3$ such that the family $G^t = G + t V$
preserves the first fundamental form $I^t$ of $G^t$ in the first order:
\begin{equation}
  \label{eq:euclinfflex}
  \frac{\dd}{\dd t} I^t\Big|_{t = 0} = 0.
\end{equation}
We will adapt this idea to give a sensible notion of infinitesimal
isotropic isometries.

\subsection{Infinitesimal deformation in $I^3$} 
\label{ssec:isodefo}

Let us consider a surface $F \colon U \subset \R^2 \to \R^3$ in isotropic
space $I^3$ and the deformation 
\begin{equation*}
  F^t = F + t V, 
\end{equation*}
with the \emph{deformation vector field} 
$V(u, v) = (n_1(u, v), n_2(u, v), n_3(u, v))$.
We are interested in the case where the deformation is
infinitesimally isometric of first order at $t = 0$.
In contrast to the Euclidean notion of infinitesimal flexibility (which is
just Equation~\eqref{eq:euclinfflex}), for infinitesimal flexibility in
$I^3$ we additionally require the preservation of the isotropic Gaussian
curvature in the first order which motivates the following definition.

\begin{defn}
  \label{defn:inffelxsmooth}
  An isotropic surface $F(u, v)$ is \emph{infinitesimally flexible} if
  there exists a non-trivial deformation vector field $V$ such that 
  \begin{equation}
    \label{eq:inf-isom}
    \frac{\dd}{\dd t} \tilde I^t\Big|_{t = 0} 
    = 
    \frac{\dd}{\dd t} 
    \small\arraycolsep=.3\arraycolsep\def\arraystretch{.9}\ensuremath{%
      \left(
      \begin{array}{ccc}
        \la \tilde F^t_u, \tilde F^t_u\ra
        &
        \la \tilde F^t_u, \tilde F^t_v\ra
        \\
        \la \tilde F^t_u, \tilde F^t_v\ra
        &
        \la \tilde F^t_v, \tilde F^t_v\ra
      \end{array}
      \right)}\Big|_{t = 0} 
      = 0
      \qquad\text{and}\qquad
      \frac{\dd}{\dd t} K(F^t)\Big|_{t = 0} = 0
  \end{equation}
  holds, where $\tilde F$ denotes the top view of $F$.
\end{defn}

For admissible surfaces in the representation of a graph $F(u, v) = (u, v,
f(u, v))$ we obtain the following expressions. For the metric part we need
that the top views are isometric of first order, i.e., $\tilde F^t$ must
satisfy the left equation of~\eqref{eq:inf-isom}:
\begin{align*}
  0 
  &=
  \frac{\dd}{\dd t} \la \tilde F^t_u, \tilde F^t_u\ra\Big|_{t = 0} 
  =
  2 \la \tilde F_u, \tilde V_u\ra 
  =
  2 \la \svs{1 \\ 0}, \svs{n_{1 u} \\ n_{2 u}}\ra 
  =
  2 n_{1 u}
  \\
  0
  &=
  \frac{\dd}{\dd t} \la \tilde F^t_v, \tilde F^t_v\ra\Big|_{t = 0} 
  =
  2 \la \tilde F_v, \tilde V_v\ra 
  =
  2 \la \svs{0 \\ 1}, \svs{n_{1 v} \\ n_{2 v}}\ra 
  =
  2 n_{2 v}
  \\
  0
  &=
  \frac{\dd}{\dd t} \la \tilde F^t_u, \tilde F^t_v\ra\Big|_{t = 0} 
  =
  \la \tilde F_u, \tilde V_v\ra + \la \tilde F_v, \tilde V_u\ra 
  =
  \la \svs{1 \\ 0}, \svs{n_{1 v} \\ n_{2 v}}\ra 
  +
  \la \svs{0 \\ 1}, \svs{n_{1 u} \\ n_{2 u}}\ra 
  =
  n_{1 v} + n_{2 u}.
\end{align*}
Consequently, 
\begin{equation}
  \label{eq:ivelocity-allg}
  V(u, v) = \sv{-k v + d_1\\\hphantom{-}k u + d_2\\\hphantom{-}n(u, v)},
\end{equation}
with $d_1, d_2, k \in \R$ and where we set $n(u, v) = n_3(u, v)$. 
Then, for the isotropic Gaussian curvature~\eqref{eq:general-K} of
$K(F^t)$ we have
\begin{equation*}
  K(F^t) =
  K(F) + t (f_{uu} n_{vv} - 2 f_{uv} n_{uv} + f_{vv} n_{uu}) + t^2 K(V).
\end{equation*}
Note that the coefficient of $t$ is the mixed Gaussian curvature $K(F, V)$
as in Equation~\eqref{eq:Ksum}.
Since any translation vector added to the velocity vector field from
Equation~\eqref{eq:ivelocity-allg} plays no role in the condition of
infinitesimal isometry we will w.l.o.g.\ only consider the simplified
version 
\begin{equation}
  \label{eq:ivelocity}
  V(u, v) = (-k v, k u, n(u, v)).
\end{equation}
Note that the top views of $F$ and $V$ are related by a rotation about
$\pi/2$. The velocity vector field $V$ will be called \emph{velocity
diagram} in Section~\ref{ssec:diagrams}. In this simplified setting where
we assume parametrization as graph $(u, v, f(u, v))$ the condition of
infinitesimal isometry boils down to 
\begin{equation} 
  \label{eq:inf} 
  K(F, V) = f_{uu} n_{vv} - 2 f_{uv} n_{uv} + f_{vv} n_{uu} = 0.
\end{equation}
This equation resembles the corresponding condition for infinitesimal
Euclidean isometries~\cite[p.~167]{sauer:1970}.
Now, Corollary~\ref{cor:mixed} implies the following proposition.
\begin{prop}
  \label{prop:flexmixed}
  An isotropic surface $F$ is infinitesimally flexible with deformation
  vector field $V$ as above if and only if 
  $\area(\delta(F), \delta(G)) \equiv 0$ or
  $\area(\nu(F), \nu(G)) \equiv 0$, where $\delta, \nu$ are the metric
  dualities.
\end{prop}

\begin{SCfigure}[1.9][t]
  \begin{overpic}[width=.49\textwidth]{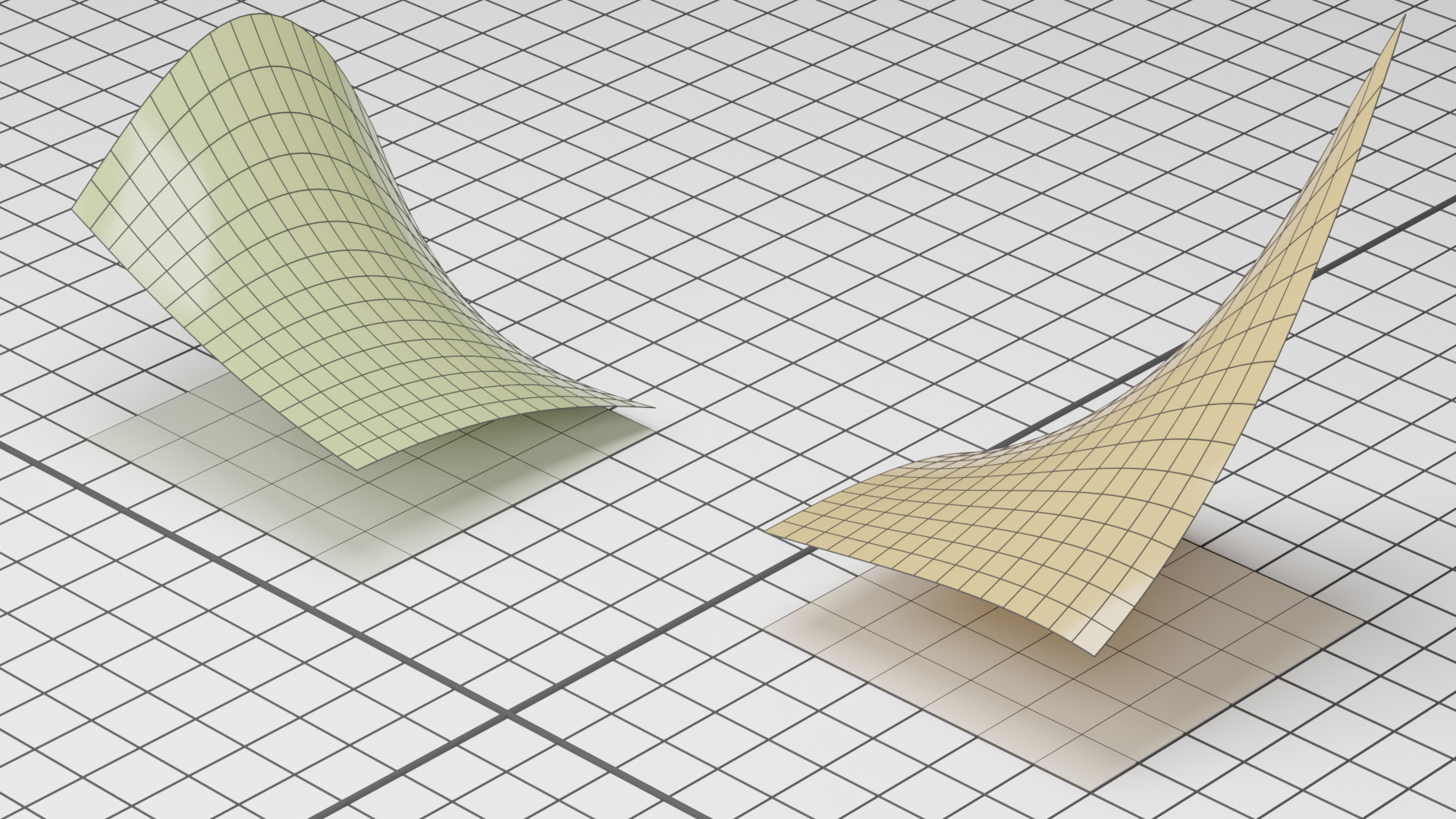}
    \put(6,50){$F$}
    \put(88,27){$V$}
  \end{overpic}
  \caption{An infinitesimally flexible surface $F$ together with its
  velocity vector field $V$. The top views of $F$ and $V$ are rotated by
  $\pi/2$. Due to the symmetry in the flexibility
  condition~\eqref{eq:inf},
  $V$ can be considered as infinitesimally flexible surface with velocity
  vector field $F$.}
\end{SCfigure}

\begin{ex}
  Let $F(u, v) = (u, v, (u^2 + v^2)/2)$ be the isotropic unit sphere of
  parabolic type. Then the differential equation for $n$ in the
  deformation vector field $V$ reads
  \begin{equation*}
    n_{uu} + n_{vv} = 0.
  \end{equation*}
  Consequently, any harmonic function $n$ yields a deformation vector field
  for the unit sphere.
  Note the difference in this example to the Euclidean sphere which is
  not infinitesimally flexible in $E^3$ (see, e.g.,~\cite{liebmann-1900}).
\end{ex}

Equation~\eqref{eq:inf} also appears in statics as Pucher's equation for
the equilibrium of a surface $F$, in the present case without external
loads applied to the interior of $F$. There, one of the surfaces, say $F$
is the surface in equilibrium, while the other one, $V(u, v) = (u, v, n(u,
v))$, is the Airy stress surface. Obviously, due to the symmetry in $f$
and $n$, equilibrium surface $F$ and stress surface $V$ can exchange their
roles. Due to the absence of external loads, equilibrium is already
present in the first two coordinates, i.e., in the top view. Hence, an
infinitesimal isometry of a surface implies a planar system in equilibrium
(elastostatic state) and its associated Airy stress surface.
Strubecker~\cite{strubecker:1962} pointed to the complete analogy between
the isotropic differential geometry of the Airy surface and the mechanical
properties of the planar elastostatic state. For example, principal
stresses and stress directions in 2D correspond to isotropic principal
curvatures and curvature directions of the stress surface. This viewpoint
has been helpful in various applications, such as the understanding of
material minimizing forms and structures~\cite{pellis-optimal-2017}. 

\begin{rem}
  Pucher's equation for equilibrium of a shell $F$ under $z$-parallel
  loads reads
  \begin{equation}
    \label{eq:pucher} 
    f_{uu} n_{vv} - 2 f_{uv} n_{uv} + f_{vv} n_{uu} = L, 
  \end{equation}
  where $L$ is the load per unit area in the top
  view~\cite{Timoshenko:1959}. It follows from our derivation, that $L$ is
  also the derivative $K'$ of isotropic Gaussian curvature $K$ of $F$ with
  respect to $t$ under a top-view preserving (for $L\ne 0$ non-isometric)
  deformation of $F$ with velocity surface $V(u, v) = (u, v, n(u, v))$.
  $V$ is the Airy stress surface of the planar stress state in the top
  view. If $F$ and $V$ agree up to a scalar factor, $f = \lambda n$, one
  speaks of a self-Airy surface $F$. Equation~\eqref{eq:pucher} then reads
  $f_{uu} f_{vv} - f_{uv}^2 = L/(2 \lambda)$. In case of constant load
  $L$, the shell $F$ possesses constant isotropic Gaussian curvature $K =
  L/(2 \lambda)$. For applications of such surfaces, studied in detail by
  Strubecker~\cite{strubecker2}, we refer to Millar et
  al.~\cite{Millar2022}. 
\end{rem}

\subsection{Associated diagrams and relation to statics} 
\label{ssec:diagrams}

Our investigation of infinitesimal isometries in $I^3$ follows Sauer's
approach to infinitesimal isometries in $E^3$~\cite{sauer:1970}. He
associates three surfaces/diagrams to a given surface that undergoes
an infinitesimal isometry. However, a straightforward translation from
$E^3$ to $I^3$ does not work.

\paragraph{Associated diagrams}

One of the three surfaces that Sauer~\cite{sauer:1970} associates with an
infinitesimal isometry of $F(u, v)$ is the \emph{velocity diagram}
(\textit{Ger.\ Verlagerungsriss}) $V(u, v)$. It is formed by the velocity
vector field of the infinitesimal isometry. Our first goal is to express
the velocity diagram in terms of two other diagrams. For that we first
look at one-parameter motions in $E^3$ (left column) and $I^3$ (right
column), respectively.

\begin{multicols}{2}
  \setlength{\columnseprule}{1pt}
  \noindent
  A \emph{Euclidean} one-parameter motion is given by
  $x \mapsto A(t) x + a(t)$
  where $A(t) \in \SO(3)$ are rotations and $a(t) \in \R^3$ is the path of
  the origin.
  The velocity vector $V(x)$ at a point $x \in E^3$ is then of the
  form~\cite[Prop.~3.4.1]{pottwall:2001}

  \begin{equation}
    \label{eq:euclideanvelicitydiag}
    V(x) = \dot x(t) = 
    \bar C + S x, 
  \end{equation}
  with $\bar C := \dot a - \dot A A^{-1} a$ and $S := \dot A A^{-1}$. 

  \columnbreak

  \noindent
  An \emph{isotropic} one-parameter motion is given by
  $x \mapsto A(t) x + a(t)$
  where 
  $A(t)\! =\!
  \Big(
  {
    \tiny\arraycolsep=.3\arraycolsep\def\arraystretch{.9}\ensuremath{%
      \begin{array}{c|c}
        \text{\footnotesize $B(t)$} & \begin{matrix}0\\0\end{matrix}\\
          \hline
          \rule{0mm}{1.3\fontcharht\font`b}
          \begin{matrix}b_1(t) & b_2(t)\end{matrix} & 1
      \end{array}}
  }
  \Big)
  $ 
  with $B(t) \in \SO(2)$, $b_1(t), b_2(t) \in \R$ and $a(t) \in \R^3$.
  The velocity vector $V(x)$ at a point $x \in I^3$ is then in analogy to
  $E^3$ of the form
  \begin{equation}
    \label{eq:isotropicvelicitydiag}
    V(x) = \dot x(t) = 
    \bar D + T x, 
  \end{equation}
  with $\bar D := \dot a - \dot A A^{-1} a$ and $T := \dot A A^{-1}$. 
\end{multicols}
\noindent
The only difference between $E^3$ and $I^3$ lies in the structure of $S$
and $T$:
\begin{multicols}{2}%
  \setlength{\columnseprule}{1pt}%
  \vspace*{-5ex}
  \begin{equation*}
  S 
  = 
  \small\arraycolsep=.3\arraycolsep\def\arraystretch{.9}\ensuremath{%
  \left(\!\!
  \begin{array}{ccc}
    \hphantom{-}0   & -c_3   & \hphantom{-}c_2\\
    \hphantom{-}c_3   & \hphantom{-}0   & -c_1\\
    -c_2 & \hphantom{-}c_1 & \hphantom{-}0
  \end{array}
  \!\right)}
  \end{equation*}

  \columnbreak
  \vspace*{-5ex}
\begin{equation}
  \label{eq:matrixt}
  T 
  = 
  \small\arraycolsep=.3\arraycolsep\def\arraystretch{.9}\ensuremath{%
  \left(\!\!
  \begin{array}{ccc}
    \hphantom{-}0   & c_3   & 0\\
    -c_3   & 0   & 0\\
    \hphantom{-}c_1 & c_2 & 0
  \end{array}
  \right),}
\end{equation}
\end{multicols}
\noindent
where matrix $S$ is skew\dash symmetric yielding $S x = C \times x$ for
$C = (c_1, c_2, c_3) \in \R^3$.

\label{frames}
Let us consider an infinitesimally flexible surface $F(u, v) = (u, v, f(u, v))$
with a corresponding velocity diagram $V(u, v) = (-k v, k u, n)$.
Furthermore, let us fix a single frame 
$(F(u_0, v_0), F_u(u_0, v_0), F_v(u_0, v_0))$ consisting of a
surface point $F$ and the two tangent vectors $F_u, F_v$. The deformation
$F^t = F + t V$ maps this frame to new frames $(F^t, F^t_x, F^t_y)$. The
fact that $V$ is a velocity diagram corresponding to an infinitesimal
isometry implies that this frame deforms isometrically of first order.
Therefore, there are one-parameter motions $\alpha(x, t) = A(t) x + a(t)$
which agree with this motion of frames of first order at $t = 0$, i.e., 
\begin{align*}
  \alpha(F(u_0, v_0), 0) 
  &= F(u_0, v_0),
  \quad
  &
  \textstyle
  \frac{\partial }{\partial t}\alpha(F_u(u_0, v_0), t)\big|_{t = 0} 
  &= V_u(u_0, v_0),
  \\
  \textstyle
  \frac{\partial }{\partial t}\alpha(F(u_0, v_0), t)\big|_{t = 0} 
  &= V(u_0, v_0),
  \quad
  &
  \textstyle
  \frac{\partial }{\partial t}\alpha(F_v(u_0, v_0), t)\big|_{t = 0} 
  &= V_v(u_0, v_0).
\end{align*}
For example in the isotropic case the one-parameter motion 
$\alpha(x, t) = A(t) x + a(t)$, where
\begin{equation*}
  A(t) = 
  \small\arraycolsep=.3\arraycolsep\def\arraystretch{.9}\ensuremath{%
  \left(
  \begin{array}{c@{\hspace{2ex}}c@{\hspace{2ex}}c}
    \hphantom{k}\cos t   & -k \sin t   & 0\\
    k \sin t   & \hphantom{-k} \cos t   & 0\\
    t n_u(u_0, v_0)    & t n_v(u_0, v_0) & 1
  \end{array}
  \right)}
  \quad\text{and}\quad
  a(t) = t (V - F(u_0, v_0))
\end{equation*}
defroms the frame attached to $F$ in the same way of first order as the
infinitesimal deformation $F^t$. Every such one-parameter motion,
Euclidean or isotropic, gives rise to such $\bar C, C$ or $\bar D, T$ from
above, which agree at $t = 0$. Consequently, for each surface point
with parameter $(u, v)$ of an infinitesimal isometry $(F, V)$ there is an
infinitesimal motion given by $\bar C(u, v)$, $C(u, v)$ or $\bar D(u, v)$,
$T(u, v)$, respectively. Hence we obtain simple formulas for the velocity
diagrams:
\begin{multicols}{2}
  \setlength{\columnseprule}{1pt}
  \centerline{$V(u, v) = \bar C(u, v) + C(u, v) \times F(u, v)$}

  \columnbreak

  \centerline{$V(u, v) = \bar D(u, v) + T(u, v)  F(u, v)$}
\end{multicols}
\noindent
In $E^3$ this representation yields the other two surfaces of the 
\emph{displacement diagrams} in the sense of~\cite{sauer:1970}. The
\emph{rotation diagram} (\textit{Ger.\ Drehriss}) $C(u, v)$ and the
\emph{translation diagram} (\textit{Ger.\
Verschiebungsriss}) $\bar C(u, v)$. 

In $I^3$ the definition of the other two displacement diagrams is more
involved as in $E^3$.
Before we develop them in $I^3$ we first note some interesting
relations between the four surfaces $F, V, C, \bar C$ in $E^3$. For
example, if $(F, V)$ are a surface and its velocity diagram to an
infinitesimal isometry with associated displacement diagrams 
$(C, \bar C)$, then also $(C, \bar C)$ can be considered as surface and
velocity diagram of an infinitesimal isometry with displacement diagrams
$(F, V)$.
Furthermore, in $E^3$ the pairs of surfaces $(F, V)$ as well as $(C, \bar
C)$ are orthogonally related (see, e.g.,~\cite{sauer:1970} and the
following definition).

\begin{defn}
  \label{defn:orthogonally}
  Two surfaces $F(u, v)$ and $V(u, v)$ are \emph{orthogonally related}
  if 
  \begin{equation*}
    \la F_u, V_u \ra = 
    \la F_v, V_v \ra = 
    \la F_u, V_v \ra + \la F_v, V_u \ra = 0.
  \end{equation*}
  They are \emph{I-orthogonally related} if their top views $\tilde F$ and
  $\tilde V$ are orthogonally related.
\end{defn}

Now we will develop the isotropic counterparts
of the displacement diagrams.

\paragraph{Velocity diagram in $I^3$}

The isotropic \emph{velocity diagram} $V(u, v)$ of an infinitesimal
deformation is given by Equation~\eqref{eq:ivelocity}. It is also known as
Airy stress surface in statics. Obviously, Equation~\eqref{eq:inf} states
the \emph{permutability of base surface $F$ and velocity diagram $V$. If
$V$ is a velocity diagram of $F$ then $F$ is a velocity diagram of $V$.}
Moreover, in view of Equations~\eqref{eq:Ksum} and~\eqref{eq:mixedK} we
have the following proposition.
\begin{prop}
  Two surfaces $F$ and $V$ are infinitesimally flexible and velocity
  diagrams of each other, if and only if they have vanishing mixed Gaussian
  curvature $K(F, V) = 0$.
\end{prop}

\paragraph{Displacement diagrams in $I^3$}

The Euclidean displacement diagrams $C, \bar C $ are read of the equation
that relates positions $F$ to velocity vectors $V$, see
Equation~\eqref{eq:euclideanvelicitydiag}. We will use the corresponding
``isotropic'' Equation~\eqref{eq:isotropicvelicitydiag} $V(u, v) = \bar
D(u, v) + T(u, v)  F(u, v)$ from above to develop sensible notions of
isotropic displacement diagrams. For that, we look at its coordinate
equations using Equations~\eqref{eq:ivelocity} and~\eqref{eq:matrixt}
\begin{equation*}
  \sv{\!-k v\\\! \hphantom{-}k u\\\! \hphantom{-}n}
  =
  \sv{\bar d_1\\ \bar d_2 \\ \bar d_3}
  +
  \left(\!\!\!\!
  \begin{array}{ccc}
    \hphantom{-}0   & c_3   & 0\\
    -c_3   & 0   & 0\\
    \hphantom{-}c_1 & c_2 & 0
  \end{array}
  \!\right)
  \sv{u\\v\\f}.
\end{equation*}
Let us fix a point $F(u, v) = (u, v, f)$. The isotropic motion
$x \mapsto \bar D + T x$ with fixed $\bar D$ and $T$ maps the point $F(u,
v) = (u, v, f)$ to its velocity vector $V(u, v) = (-k v, k u, n)$ and the
corresponding frame $(F_u, F_v)$ to $(V_u, V_v)$, i.e., we have
\begin{align*}
  V(u, v) = \bar D + T F(u, v),
  \quad
  V_u(u, v) = T F_u(u, v),
  \quad
  V_v(u, v) = T F_v(u, v),
\end{align*}
which implies 
$c_1 = n_u$, $c_2 = n_v$, $c_3 = -k$, $\bar d_1 = \bar d_2 = 0$, and 
$d_3 = n - n_u u - n_v v$. Consequently,
\begin{equation*}
  \bar D(u, v) = 
  \sv{0\\ 0\\ n - n_u u - n_v v}
  \quad\text{and}\quad
  T(u, v) = 
  \left(\!
  \begin{array}{ccc}
    0   & -k   & 0\\
    k   & \hphantom{-}0   & 0\\
    n_u & \hphantom{-}n_v & 0
  \end{array}
  \!\right),
\end{equation*}
which do not describe surfaces and can therefore not directly be
interpreted as displacement diagrams.
However, the diagrams that we define below project to $\bar D$ and $T$.

\begin{lem}
  \label{lem:g}
  Let $F = (u, v, f)$ and $V = (-k v, k u, n)$ be a pair of
  surface and velocity diagram of an infinitesimal isometry.
  Then there is a function $c(u, v)$ such that 
  \begin{equation}
    \label{eq:g}
    c_u = f_u n_{uv} - f_v n_{uu}
    \quad\text{and}\quad
    c_v = f_u n_{vv} - f_v n_{uv}.
  \end{equation}
\end{lem}
\begin{proof}
  We have to verify the integrability condition $c_{uv} = c_{vu}$ which is
  equivalent to~\eqref{eq:inf}.
\end{proof}

In what follows we choose $k = -1$ for the velocity diagram.

\begin{defn}
  \label{defn:diagrams}
  Let $F = (u, v, f)$ and $V = (-v, u, n)$ be a pair of
  surface and velocity diagram of an infinitesimal isometry.
  The \emph{isotropic displacement diagrams} consisting of a
  \emph{translation diagram} $\bar C$ and \emph{rotation diagram} $C$ are
  \begin{equation*}
    \bar C(u, v) = (-n_v, n_u, -n + n_u u + n_v v) 
    \quad\text{and}\quad
    C(u, v) = (-n_u, -n_v, c),
  \end{equation*}
  where $c$ is the function from Lemma~\ref{lem:g}.
\end{defn}

We will use in the following the Euclidean rotation around the $z$-axis
about an angle of $\pi/2$ and denote it by
\begin{equation}
  \label{eq:J}
  J =
  \left(\!
  \begin{array}{ccc}
    0   &\!   -1             &   0\\
    1   &\!   \hphantom{-}0  &   0\\
    0   &\!   \hphantom{-}0  &   1
  \end{array}
  \!\right).
\end{equation}

\begin{thm} 
  \label{thm:diagrams}
  An infinitesimal isometry of a surface $F$ in $I^3$, the associated
  velocity diagram $V$, rotation diagram $C$ and translation diagram
  $\bar C$ possess the contact element representations
  \begin{equation}
    \label{eq:diagrams}
    \begin{aligned}
      &
      F = (u, v, f, f_u, f_v),
      &&
      V = (-v, u, n, -n_v, n_u),\\
      &
      C = (-n_u, -n_v, c, f_v, -f_u),
      &&
      \bar C = (-n_v, n_u, -n + n_u u + n_v v, -v, u).
    \end{aligned}
  \end{equation}
  Functions $f, n$ satisfy Equation~\eqref{eq:inf} and $c$ satisfies
  Equation~\eqref{eq:g}. Relations between these four surfaces are
  as follows (schematically illustrated by Fig.~\ref{fig:relations}):
  \begin{enumerate}
    \item\label{itm:isoi} $F$ and $V$, as well as $C$ and $\bar C$ are
      \emph{I\dash orthogonally related}, i.e., their top views are
      orthogonally related (Def.~\ref{defn:orthogonally}).
    \item\label{itm:isov} $V$ and $\bar C$ correspond in the metric
      duality $\delta$. 
    \item\label{itm:isoiv} The tangent planes in corresponding points of
      $F$ and $J C$ are parallel to each other.
    \item\label{itm:isoiii} The surface $\bar C$ is a velocity diagram for
      an infinitesimal isometry of the surface $C$ and vice versa.
  \end{enumerate}
\end{thm}
\begin{proof}
  The contact element representations are obtained following
  Section~\ref{subsec:contactelements} from the corresponding
  parametrizations and for $C$ with the help of Lemma~\ref{lem:g}.

  Ad \ref{itm:isoi}: 
  The top views of $F$ and $V$ are related by a Euclidean rotation about
  $\pi/2$. Therefore we must have
  $\la \tilde F_u, \tilde V_u \ra 
  =
  \la \tilde F_v, \tilde V_v \ra 
  = 0
  $
  and 
  $\la \tilde F_u, \tilde V_v \ra + \la \tilde F_v, \tilde V_u \ra 
  = 0$, 
  which implies that $F$ and $V$ are I\dash orthogonally related.
  Analogously, $C$ and $\bar C$ are I\dash orthogonally related.

  Ad \ref{itm:isov}: 
  Applying Lemma~\ref{lemma1} yields metric duality between $V$ and $\bar
  C$.

  Ad \ref{itm:isoiv}: 
  The ``normalized'' normal vectors of $F$ and $C$ are easily read of from
  their contact element representations. The normal vector of $F$ is
  $(f_u, f_v, -1)$ and of $C$ it is $(f_v, -f_u, -1)$. Consequently, the
  rotation with matrix $J$ from Equation~\eqref{eq:J} rotates 
  $C$ to $J C$ with normal vector $(f_u, f_v, -1)$ equal to the one of
  $F$.

  Ad \ref{itm:isoiii}: 
  We have to show~\eqref{eq:inf-isom}, i.e., 
  $\frac{\dd}{\dd t} \tilde I^t\big|_{t = 0} = 0$, and 
  $\frac{\dd}{\dd t} K(C + t \bar C)\big|_{t = 0} = 0$. We set 
  $G := C + t \bar C$.
  Simple computations show 
  $\frac{\dd}{\dd t} \la \tilde G_u, \tilde G_u \ra\big|_{t = 0} = 
  \frac{\dd}{\dd t} \la \tilde G_u, \tilde G_v \ra\big|_{t = 0} = 
  \frac{\dd}{\dd t} \la \tilde G_v, \tilde G_v \ra\big|_{t = 0} = 0$, 
  and therefore
  $\frac{\dd}{\dd t} \tilde I^t\big|_{t = 0} = 0$.
  After lengthy but simple computations we obtain
  \begin{align*}
    &\det(G_u, G_v, G_{uu})
    =
    \ldots
    =
    (1 + t^2)
    (n_{uu} n_{vv} - n_{uv}^2)
    (t n_{uu} - n_{uu} f_{uv} + n_{uv} f_{uu})
    \\
    &\det(G_u, G_v, G_{vv})
    =
    \ldots
    =
    (1 + t^2)
    (n_{uu} n_{vv} - n_{uv}^2)
    (t n_{vv} + n_{vv} f_{uv} - n_{uv} f_{vv})
    \\
    &\det(G_u, G_v, G_{uv})
    =
    \ldots
    =
    (1 + t^2)
    (n_{uu} n_{vv} - n_{uv}^2)
    (t n_{uv} + n_{uv} f_{uv} - n_{uu} f_{vv})
    \\
    &
    \la\tilde G_u, \tilde G_u\ra
    =
    \ldots
    =
    (1 + t^2) (n_{uv}^2 + n_{uu}^2)
    \\
    &
    \la\tilde G_v, \tilde G_v\ra
    =
    \ldots
    =
    (1 + t^2) (n_{uv}^2 + n_{vv}^2)
    \\
    &
    \la\tilde G_u, \tilde G_v\ra
    =
    \ldots
    =
    (1 + t^2) (n_{uu} + n_{vv}) n_{uv}.
  \end{align*}
  Putting these expressions together yields
  \begin{align*}
    &
    \frac{\dd}{\dd t} K(C + t \bar C)\big|_{t = 0}
    =
    \frac{\dd}{\dd t} K(G)\big|_{t = 0}
    =
    \ldots
    =
    n_{uv} (f_{uu} n_{vv} - 2 f_{uv} n_{uv} + f_{vv} n_{uu})
    \\
    =\ &
    n_{uv} K(F, V) 
    = 
    0,
  \end{align*}
  which is what we wanted to show.
\end{proof}

\begin{figure}[tb]
  \hfill
  \begin{overpic}[width=.4\textwidth]{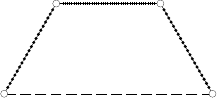}
    \put(-5,0){$V$}
    \put(101,0){$\bar C$}
    \put(78,41){$C$}
    \put(18,41){$F$}
  \end{overpic}
  \hfill
  \hfill
  \begin{overpic}[width=.15\textwidth]{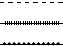}
    \put(106,-2){\small\ I-orthogonal}
    \put(106,34){\small\ $\pi/2$-parallel}
    \put(106,65){\small\ $\delta$-dual}
  \end{overpic}
  \hfill
  \hfill{}
  \caption{Relations between a surface $F$, its velocity diagram $V$,
  rotation diagram $C$ and translation diagram $\bar C$. Their relations
  are described in Theorem~\ref{thm:diagrams}.}
  \label{fig:relations}
\end{figure}
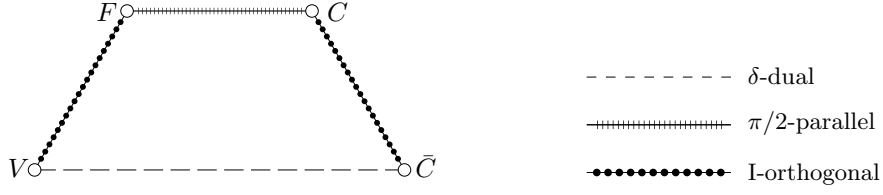

\subsection{Relation between two isometric positions} 

There is a close relation between two isometric surfaces and an
infinitesimal isometry, based on the following simple observation:

\begin{prop} 
  \label{prop:2pos} 
  Let $F = (u, v, f)$ and $\bar F = (u, v, \bar f)$ be two isometric
  surfaces in $I^3$. Then $V = (-v, u, f - \bar f)$ is a velocity diagram
  of an infinitesimal isometry of the sum surface $(u, v, f + \bar f)$
  or the middle surface $M = (u, v, (f + \bar f)/2)$.
\end{prop}
\begin{proof}
  Two surfaces $F$ and $\bar F$ with the same top view are isometric if
  and only if their Gaussian curvature is the same implying 
  $f_{uu} f_{vv} - f_{uv}^2 = \bar f_{uu} \bar f_{vv} - \bar f_{uv}^2$
  which is equivalent to 
  $(f + \bar f)_{uu} (f - \bar f)_{vv} + (f + \bar f)_{vv} (f - \bar f)_{uu} 
  - 2 (f + \bar f)_{uv} (f - \bar f)_{uv} = 0$.
  This equation is precisely Equation~\eqref{eq:inf}, the mixed Gaussian
  curvature for  the surface $(u, v, f + \bar f)$ with velocity
  diagram $(-v, u, f - \bar f)$. 
\end{proof}

By symmetry, $M$ is also a velocity diagram of an infinitesimal isometry
of $V$. However, note that $M$ is not isometric to $V$. The converse of
Proposition~\ref{prop:2pos} is the follwing.
\begin{prop}
  \label{prop:isofromdiagr} 
  Let $V = (-v, u, n)$ and $F = (u, v, f)$ be two surfaces.
  Then the two surfaces $F^+ = (u, v, f + n)$ and $F^- = (u, v, f - n)$
  are isometric to each other if and only if $V$ is a velocity diagram for
  an infinitesimal isometry of $F$. See
  Figure~\ref{fig:isometry-from-inf-iso} for an example.
\end{prop}
\begin{proof}
  This follows from a short computation as $K(F^+) = K(F^-)$ if and only
  if $K(F, V) = 0$.
\end{proof}
However, note that $F^+$ and $F^-$ are not isometric to $F$ or $V$.

\subsection{Isometric ruled surfaces} 
\label{ssec:2positions}

In this section we investigate characterizations of pairs of ruled
surfaces being isometric to each other such that the rulings correspond
in the isometry.

\begin{defn}
  \label{defn:ruling-preserving-inf-iso}
  An infinitesimal isometry is called \emph{ruling preserving} if the
  corresponding velocity diagram is also a ruled surface.
\end{defn}

The following proposition provides an example for a ruling preserving
infinitesimal isometry that will be needed in Section~\ref{sec:special}.

\begin{prop} 
  \label{prop:minding} 
  Let $F_1 = (u, v, f_1)$, $F_2 = (u, v, f_2)$ be two ruled surfaces in
  $I^3$ with the same top view of rulings, i.e., corresponding rulings lie
  in the same isotropic plane. Let us further denote the difference
  surface by $V = (-v, u, f_1 - f_2)$.
  \begin{enumerate}
    \item\label{itm:ruledi} Then $V$ is a non-planar torsal ruled surface
      if and only if $F_1$ and $F_2$ are related by a non-trivial
      isotropic Minding isometry (p.~\pageref{par:minding-iso}).
      A planar surface $V$ implies that $F_1$ and $F_2$ are congruent in
      $I^3$, which we call a trivial isometry. 
    \item\label{itm:ruledii} $V$ is a velocity diagram of an infinitesimal
      ruling preserving isometry of the middle surface 
      $M = (u, v, (f_1 + f_2)/2)$. 
    \item\label{itm:rulediii} $V$ is the tangent surface of a space
      curve if $F_1$ and $F_2$ are ruled surfaces of type~I, a cone for
      type~II and a non-isotropic cylinder for type~III.
  \end{enumerate}
\end{prop}
\begin{proof}
  Let us consider corresponding rulings (with same top view) of the
  surfaces $F_1$, $F_2$ and of $J^{-1} V$.
  These rulings together with the corresponding derivative vectors of the
  striction curve lie in the same isotropic tangent plane, illustrated by
  Figure~\ref{fig:striction}. It is therefore elementary that the
  striction $\sigma$ (Def.~\ref{defn:striction}) satisfies
  $\sigma(F_1) - \sigma(F_2) = \sigma(J^{-1} V) = \sigma(V)$.

  The curvature of the striction curve is measured in the top view and is
  therefore for all three surfaces the same. The pitch $\rho$ (cf.\
  Eqn.~\eqref{eq:lamarle} or~\cite{Sachs:1990}) is the striction divided
  by the curvature. Consequently, we have
  \begin{equation*}
    \rho(F_1) - \rho(F_2) = \rho(J^{-1} V) = \rho(V).
  \end{equation*}

  The Gaussian curvature $K$ depends on $\rho$ and for types I and II
  also on the striction distance $w$, seen in the top view as described in
  Equation~\eqref{eq:lamarle}. Thus equality of $K$ at corresponding
  points is equivalent to equality of the pitch $\rho$. 

  Now ad~\ref{itm:ruledi}: Minding isometric ruled surfaces $F_1, F_2$
  have the same Gaussian curvature $K$ and therefore $\rho(F_1) =
  \rho(F_2)$ which implies $\rho(V) = 0$. Consequently, $K(V) = 0$ and
  therefore $V$ is torsal (developable). Conversely, $\rho(V) = 0$ implies
  equal pitch and thus equal $K$ at corresponding points of $F_1$ and
  $F_2$. 

  As for~\ref{itm:ruledii}: 
  Proposition~\ref{prop:2pos} implies that $V$ is a velocity diagram of an
  infinitesimal isometry of $M$. Since $V = (-v, u, f_1 - f_2)$ it is
  also a ruled surface and therefore by
  Definition~\ref{defn:ruling-preserving-inf-iso} the infinitesimal
  isometry is ruling preserving.

  As for~\ref{itm:rulediii}: 
  The statements on the types follow immediately from the well-known
  classification of torsal ruled surfaces and their top views (as
  discussed in the section on Minding isometries starting on
  page~\pageref{par:minding-iso}, and as illustrated in
  Figure~\ref{fig:ruled-surface}).
\end{proof}

\begin{SCfigure}[2][t]
  \begin{overpic}[width=.3\textwidth]{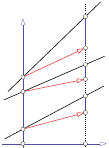}
    \put(60,3){\small$1$}
    \put(8,14){\small$F_2$}
    \put(16,32){\small$J^{-1} V$}
    \put(8,48){\small$F_1$}
    \put(59,75){\small\color{myblue}$\sigma(F_1)$}
    \put(59,49){\small\color{myblue}$\sigma(J^{-1} V)$}
    \put(59,28){\small\color{myblue}$\sigma(F_2)$}
    \put(35,73){\rotatebox{47}{\small$r_{F_1}$}}
    \put(32,48){\rotatebox{25}{\small$r_{J^{-1} V}$}}
    \put(35,26){\rotatebox{30}{\small$r_{F_2}$}}
    \put(46,58){\small\color{myred}$\dot s_{F_1}$}
    \put(41,38){\small\color{myred}$\dot s_{J^{-1} V}$}
    \put(43,15){\small\color{myred}$\dot s_{F_2}$}
  \end{overpic}
  \hspace{5mm}
  \caption{Illustration of an isotropic tangent plane of ruled surfaces
  $F_1 = (u, v, f_1)$, $F_2 = (u, v, f_2)$, and $J^{-1} V = (u, v, f_1 -
  f_2)$ which share the same top view of their rulings. The striction
  $\sigma(F_1)$ is the isotropic angle between the ruling $r_{F_1}$ and
  the tangent vector $\dot s_{F_1}$ of the arc length parametrized
  striction curve $s_{F_1}$ of $F_1$. Analogously for $F_2$ and $J^{-1} V$.
  The pitch $\rho(F_1) = \sigma(F_1)/\kappa$ is the striction divided by the
  curvature $\kappa$ of the top view of the striction curve which is the
  same for all curves with the same top view.}
  \label{fig:striction}
\end{SCfigure}
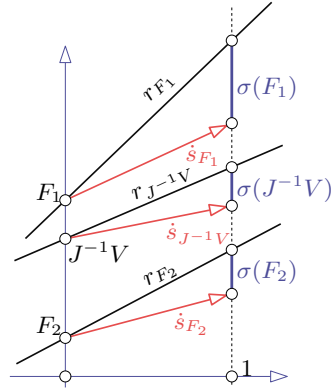

Proposition~\ref{prop:minding} immediately implies the following
corollary.

\begin{cor}
  \label{cor:minding-isometry}
  Minding isometries of a ruled surface $F$ in $I^3$ are generated by
  adding torsal ruled surfaces $R$, which have the same top view of
  rulings as $F$. The striction curve of $F$ and the regression curve of
  $R$ have the same top view. 
\end{cor}

\begin{cor}
  \label{cor:rulingpresisom}
  A velocity diagram $V$ of a ruling preserving infinitesimal isometry of
  a ruled surface $F$ is a torsal surface. 
\end{cor}
\begin{proof}
  Let $V = (-v, u, n)$ be the velocity diagram of a ruling preserving
  infinitesimal isometry of the ruled surface $F = (u, v, f)$.
  Then, by Proposition~\ref{prop:isofromdiagr}, 
  $F_1 = (u, v, f + n)$ and $F_2 = (u, v, f - n)$ are isometric ruled
  surfaces.
  Proposition~\ref{prop:minding}~\ref{itm:ruledi} implies that ``$F_1 -
  F_2$'' (which is $V$) is a torsal ruled surface.
  Furthermore, by Proposition~\ref{prop:minding}~\ref{itm:ruledii} we
  obtain that ``$F_1 - F_2$'' (which is $V$) is a velocity diagram of an
  infinitesimal ruling preserving isometry of the surface ``$F_1 + F_2$''
  (which is $F$).
\end{proof}

\begin{cor}
  \label{cor:torsaltoisom}
  Let $F$ and $J^{-1} V$ be two ruled surfaces with corresponding rulings
  having the same top view. If $V$ is a torsal surface then it is a velocity
  diagram of a ruling preserving infinitesimal isometry of $F$.
\end{cor}
\begin{proof}
  Let $F = (u, v, f)$ and $V = (-v, u, n)$ and let furthermore 
  $F_1 = (u, v, f + n)$ and $F_2 = (u, v, f - n)$.
  By Proposition~\ref{prop:minding}~\ref{itm:ruledi} $F_1$ and $F_2$ are
  related by a non-trivial isotropic Minding isometry. Consequently,
  Proposition~\ref{prop:minding}~\ref{itm:ruledii} implies that $V$ is the
  velocity diagram of a ruling preserving isometry of $F$.
\end{proof}

Proposition~\ref{prop:isometricdual} implies the following
characterization of Minding isometries in terms of metric duality.
\begin{prop}
  Two ruled surfaces $G, \bar G$ are metric duals of a pair of Minding
  isometric ruled sufaces if and only if $G$ and $\bar G$ have parallel
  corresponding rulings, the map of corresponding points is area
  preserving and the tangent planes at corresponding points are parallel.
\end{prop}

\begin{figure}[t]
  \centerline{%
  \begin{overpic}[width=.9\textwidth]{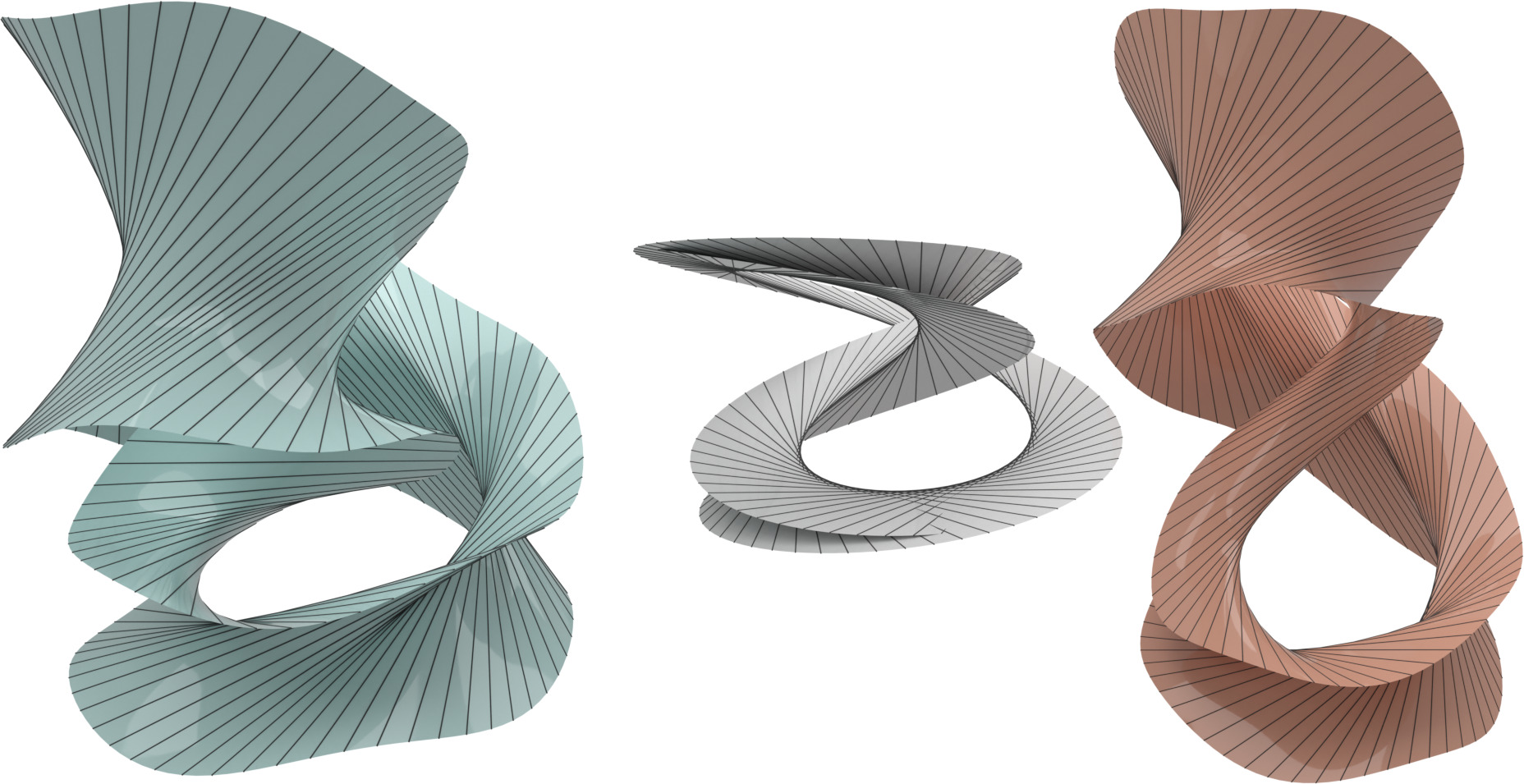}
    \put(25,0){$F$}
    \put(57,12){$R$}
    \put(75,0){$\bar F$}
  \end{overpic}}
  \centerline{%
  \begin{overpic}[width=.9\textwidth]{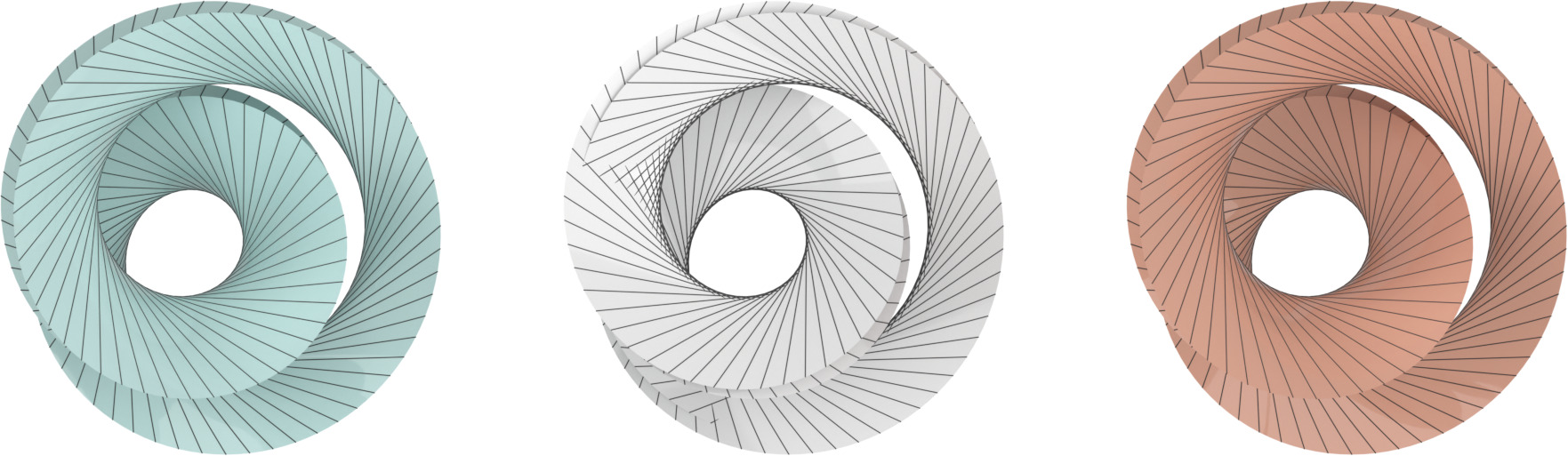}
  \end{overpic}}
  \caption{Minding isoemtry. Corollary~\ref{cor:minding-isometry} implies
  that we obtain two isometric ruled surfaces $F, \bar F$ (\emph{left} and
  \emph{right}) by adding a torsal ruled surface $R$ (\emph{center}) to
  $F$ with the same top view of rulings.
  The top views of the three surfaces (\emph{bottom-row}) are identical.}
\end{figure}

\paragraph{Discrete model of Minding isometries in $I^3$} 
\label{par:discreteminding}

It is instructive to discuss a discrete model of these Minding isometries,
in analogy to the Euclidean case (see~\cite{pottwall:2001,sauer:1970}). 
Let $r(t)$ denote the rulings of a ruled surfaces and consider two
neighbouring rulings $r(t), r(t + \varepsilon)$. Below we will always
assume $\varepsilon$ to be sufficiently close to $0$.
For type~I or~II, the top views $\tilde r(t), \tilde r(t + \varepsilon)$ 
intersect in a point $\tilde s(t)$. The isotropic line $N(t)$ through that
point intersects $r(t), r(t + \varepsilon)$ in points $s(t), s(t +
\varepsilon)$ with vertical distance $d(t)$ (which can be defined to be
the distance of the two rulings). The limit of this distance, divided by
the angle between the rulings, when $\varepsilon \to 0$ on a $C^2$
ruled surface, is the pitch $\rho(t)$ (cf.~\cite[p.~199]{Sachs:1990}). 

Consider now a discrete sequence of rulings $(r_i)_{i \in \Z}$
forming a discrete ruled surface.
A congruence transformation in $I^3$ which fixes the top view is an
isotropic shearing and just adds a linear function in isotropic direction.
This changes rulings in their respective isotropic planes but keeps their
pairwise isotropic distances $d_i$ and angles $\phi_i$ and thus the
discrete pitch $\rho_i = d_i/\phi_i$. Also discrete striction distances
$w$ do not change and thus the Gaussian curvature $K = -\rho_i^2/w^4$
(cf.~\eqref{eq:lamarle}), remains unchanged. Consequently, applying
arbitrarily many such shearings yields discrete isometric ruled surfaces.
Furthermore, the difference surface between two isometric discrete ruled
surfaces created in this way has coplanar consecutive rulings, and thus is
a discrete developable surface (in analogy to the smooth property of
Corollary~\ref{cor:minding-isometry}). We will get back to this discussion
in Section~\ref{ssec:rigidfaces}.

\subsection{From Euclidean to isotropic diagrams} 
\label{ssec:e-to-i}

A surface in equilibrium without external forces acting on its interior
has many Airy stress surfaces. We can take any direction $e_3$ as the one
in which we orthogonally project onto a plane, yielding a planar stress
state and a stress surface. Equivalent to that is the introduction of an
isotropic direction and the transfer from a Euclidean infinitesimal
isometry to the generation of associated isotropic stress surfaces.

Consider a Euclidean infinitesimal isometry of a surface $F^e$ in $E^3$
and associated diagrams $V^e, C^e, \bar C^e$, which fulfill the relation
$V^e = \bar C^e + C^e \times F^e$. 
Our goal is to generate diagrams $F, V, C, \bar C$ which share the same
transformation relations as in Theorem~\ref{thm:diagrams}
\ref{itm:isoi}-\ref{itm:isoiv} and illustrated in
Figure~\ref{fig:relations}. 
Now we introduce an isotropic direction $e_3$ and extend it to a Cartesian
frame $T = (e_1, e_2, e_3) \in \R^{3 \times 3}$.
Let us set $F = F^e$ and $C = T J^{-1} T^{-1} C^e$ so that in the frame
$T$ we have that $F$ and $J C$ have parallel tangent planes (since $F^e$
and $C^e$ have parallel tangent planes). 
We generate an isotropic velocity vector
field $V = (-\la F, e_2\ra, \la F, e_1\ra, n)$, where
\begin{equation}
  \label{eq:n}
  n = \la V^e, e_3\ra = \la \bar C^e, e_3\ra + \det(C, F, e_3).     
\end{equation}
In view of the symmetry of surface and velocity diagram $(F^e, V^e)$, (or
by simply reversing the above equation) we obtain 
$\bar C^e = V^e + F^e \times C^e$ and use it to define $\bar C$ via
\begin{equation*}
  \bar C = (\la C, e_2\ra, -\la C, e_1\ra, \bar c_3),
  \quad\text{where}\quad
  \bar c_3 := -\la \bar C^e, e_3 \ra = -n - \det(F^e, C^e, e_3). 
\end{equation*}
Consequently, the pairs $(F, V)$ and $(C, \bar C)$ are I-orthogonally
related in the frame $T$ (Def.~\ref{defn:orthogonally}).

\begin{prop}
  For any Euclidean infinitesimal isometry with diagrams $F^e$, $V^e$,
  $C^e$, $\bar C^e$ and any unit vector $e^3 \in \R^3$ the derived
  diagrams $F, V, C, \bar C$ (as above) fulfill the same relations as in
  Theorem~\ref{thm:diagrams} \ref{itm:isoi}-\ref{itm:isoiv} and
  illustrated in Figure~\ref{fig:relations}. 
\end{prop}
\begin{proof}
  What remains to show is that the resulting diagrams $V$ and $\bar C$ are
  related by metric duality $\delta $in $I^3$. 
  For simplicity, we take a Cartesian system with the canonical basis
  vectors $e_i$ and consider the surface 
  $F = F^e = (f_1, f_2, f_3)$, 
  $V = (n_1, n_2, n_3)$, 
  $C = (c_1, c_2, c_3)$, 
  $\bar C = (\bar c_1, \bar c_2, \bar c_3)$.
  Consequently, 
  \begin{equation*}
    V = (-f_2, f_1, n_3),
    \quad 
    \bar C = (c_2 , -c_1, \bar c_3),
    \quad 
    C^e = (-c_2 , c_1, c_3),
  \end{equation*}
  where $n_3 = n \overset{\eqref{eq:n}}{=} -\bar c_3 - c_1 f_1 - c_2 f_2$.
  Since $C^e$ also acts on tangent vectors of $F = F^e$, we
  have~\cite{sauer:1970}
  \begin{equation*}
    V^e_u = C^e \times F_u,
    \quad
    V^e_v = C^e \times F_v,
  \end{equation*}
  which by Equation~\eqref{eq:n} yields for the third coordinates
  \begin{equation*}
    n_u = -c_2 f_{2, u} - c_1 f_{1, u}
    \quad\text{and}\quad
    n_v = -c_2 f_{2, v} - c_1 f_{1, v}.
  \end{equation*}
  Using these relations we find a Euclidean normal vector at $V(u, v)$,
  \begin{equation*}
    V_u \times V_v = \ldots = (- c_2 d, c_1 d, d), 
    \quad\text{with}\quad
    d = f_{1, u} f_{2, v} - f_{1, v} f_{2, u}.
  \end{equation*}
  Hence, the equation of the tangent plane at $V(u, v)$ is
  \begin{equation*}
    c_2 x - c_1 y - z = -f_2 c_2 - f_1 c_1 - n_3 = \bar c_3.
  \end{equation*}
  Its image under the metric duality $\delta$ is the point 
  $(c_1, c_2, \bar c_3) \in \bar C$.
\end{proof}

\subsection{Relative minimal surfaces}

We will use in the following the Euclidean rotation around the $z$-axis
about an angle of $\pi/2$ composed with a reflection in the $xy$-plane
and denote it by
\begin{equation}
  \label{eq:L}
  L =
  \left(\!
  \begin{array}{ccc}
     \!   \hphantom{-}0  &  1   & \! \hphantom{-}0\\
     \!   -1             &  0   & \! \hphantom{-}0\\
     \!   \hphantom{-}0  &  0   & \! -1
  \end{array}
  \!\right).
\end{equation}
Applying $L$ to a contact element $E = (x, y, z, p, q)$ yilds $L E = (y,
-x, -z, -q ,p)$.

We briefly elaborate on a relation between $F$ and $L C$, namely being
\emph{relative minimal surfaces} of each other. Furthermore,
we show that the relative principal curvatures in this generalized
differential geometric perspective are closely related to the isotropic
Gaussian curvature $K$ of $V$ or $\bar C$. The surfaces $F$ and $L C$ have
contact elements
\begin{equation*}
  F = (u, v, f, f_u, f_v)
  \quad\text{and}\quad
  L C = (-n_v, n_u, -c, f_u, f_v).
\end{equation*}
Therefore, they have parallel tangent planes at corresponding points.
These tangent planes are spanned by
$F_u = (1, 0, f_u)$, $F_v = (0, 1, f_v)$
and 
\begin{equation*} 
  L C_u = (-n_{uv}, n_{uu}, -n_{uv} f_u + n_{uu} f_v),
  \qquad
  L C_v = (-n_{vv}, n_{uv}, -n_{vv} f_u + n_{uv} f_v),
\end{equation*} 
respectively. Obviously, we have
\begin{equation}
  \label{eq:relwein}
  L C_u = -n_{uv} F_u + n_{uu} F_v, 
  \qquad
  L C_v = -n_{vv} F_u + n_{uv} F_v, 
\end{equation}
which again confirms parallelism of tangent planes. 
Furthermore, Equation~\eqref{eq:relwein} yields an affine mapping between
corresponding parallel tangent planes of $F$ and $L C$ via $(F_u, F_v)
\mapsto (L C_u, L C_v)$.
If $L C$ takes the role of the normalization map (i.e., the Gauss map
in the relative differential geometry), then Equation~\eqref{eq:relwein}
describes the relative Weingarten map (depending on the definition,
possibly up to a sign, which is not important in our context). Its
matrix,
\begin{equation*}
  W 
  = 
  \left(\!\begin{array}{ll} 
    -n_{uv} & n_{uu} \\ 
    -n_{vv} & n_{uv} 
  \end{array}\!\right),
\end{equation*}
has vanishing trace, and thus we have vanishing relative mean curvature.
This implies the following proposition.
\begin{prop}
  \label{prop:relmin}
  Let $F$ be an infinitesimally flexible surface $F$ and $C$ a rotation
  diagram. Then $F$ and $L C$ are relative minimal surfaces to
  each other.   
\end{prop}

It does not matter whether we interpret $F$ or $L C$ as relative Gauss
map. If one of the two surfaces is a Euclidean sphere then the other one
is a Euclidean minimal surface. This mechanical interpretation of relative
minimal surfaces has been pointed out for the Euclidean
case~\cite[p.~245]{Blaschke1} and for general relative minimal surfaces it
follows immediately from~\cite[p.~205]{Blaschke2}. 

\begin{prop} 
  \label{prop:dualrelmin}
  Metric duality $\nu$ maps a pair of relative minimal surfaces (surfaces
  in equilibrium which are reciprocal force diagrams of each other)
  to a pair of dual relative minimal surfaces (equilibrium surface and its
  Airy surface).
\end{prop}

For the dual viewpoint of relative differential geometry in the context of
statics we point to Vouga et al.~\cite{vouga-2012-sss}. 

\begin{rem}
  Equation~\eqref{eq:inf} characterizes dual relative minimal surfaces in
  the graph representation. Hence, there is also an analogous
  \emph{characterization of relative minimal surfaces in terms of the
  isotropic support functions}. Two such surfaces $R, S$ have parallel
  tangent planes at corresponding points, written as 
  $z = u x + v y - r(u, v)$ and $z = u x + v y - s(u, v)$. Then
  application of metric duality and our discussions in
  Section~\ref{ssec:metricdual} show that the support functions $r, s$
  of relative minimal surfaces $R, S$ satisfy %
  \begin{equation*}
    r_{uu} s_{vv} - 2 r_{uv} s_{uv} + r_{vv} s_{uu} = 0. 
  \end{equation*}
\end{rem}

By definition, \emph{relative principal curvatures}
$\kappa_1^{\mathrm{rel}}, \kappa_2^{\mathrm{rel}}$ and \emph{relative
principal curvature directions} are eigenvalues and eigenvectors of $W$.
We find
\begin{equation} 
  \label{eq:relprinccurv}
  \kappa_i^{\mathrm{rel}} 
  = 
  \pm \sqrt{n_{uv}^2 - n_{uu} n_{vv}} 
  = 
  \pm \sqrt{-K}, 
\end{equation}
where $K$ is the isotropic Gaussian curvature of $V$ at $(u, v, n(u, v))$.
A similar relation appears in Euclidean differential geometry where, in
comparison, $\pm \sqrt{-K}$ are the torsions of the asymptotic
curves~\cite{strubecker3} if their curvature is not vanishing.
The relative Gaussian curvature of $F$ (with respect to Gauss map $L C$)
equals 
$K_{\mathrm{rel}} = \kappa_1^{\mathrm{rel}} \kappa_2^{\mathrm{rel}} = K$. 

Recall, that $V$ and $\bar C$ correspond in the metric duality $\delta$
(Thm.~\ref{thm:diagrams}~\ref{itm:isov}).
Consequently, the Gaussian curvatures of $V$ and $\bar C$ are
reciprocal values to each other (cf.\ Eqn.~\eqref{eq:dualcurv}).
The same holds for the relative curvatures when switching the role of
relative Gauss maps. Hence, if $F$ is the relative Gauss map, the meaning
of $K$ in Equation~\eqref{eq:relprinccurv} is the isotropic Gaussian
curvature of $\bar C$. 

\begin{rem}
  In the statics interpretation, $V$ is the stress surface. Exactly for
  structures under tension and compression, which appear for example in
  material minimizing forms~\cite{pellis-optimal-2017}, $V$ is negatively
  curved ($K < 0$) and thus the relative principal curvatures are real. 
\end{rem}

\section{The Isotropic Darboux Wreath} 
\label{sec:darboux}

We can now derive the isotropic counterpart to the Darboux wreath in $E^3$
(see, e.g.,~\cite{sauer:1970}), which is a collection of 12
infinitesimally flexible surfaces with multiple relations between them. We
will see that the isotropic counterpart consists of 6 surfaces only, and
that several relations which differ in $E^3$, agree in $I^3$.

\subsection{Constructing the isotropic Darboux wreath}

With the four surfaces $(F, V)$ and $(C, \bar C)$ in~\eqref{eq:diagrams}
we have two pairs of infinitesimal isometric surfaces together with their
velocity diagrams (cf.\ Thm.~\ref{thm:diagrams}~\ref{itm:isoiii}). 
Recall that the symmetry between a surface and its velocity diagram
implies that also $F$ is a velocity diagram for an infinitesimal isometry
of $V$. Consequently, we can construct the rotation diagram $\bar B$ and
translation diagrams $B$ corresponding to the pair $(V, F)$. Note that we
switched the bar notation in this case. Interestingly, adding these two
surfaces to the list of surface and velocity diagrams $F, V, C, \bar C$
already closes the \emph{isotropic} Darboux wreath
(Thm~\ref{thm:darbouxwreath-i}).

\begin{lem} 
  Let $(F, V)$ be a pair consisting of a velocity diagram $V = (-v, u, n)$
  for an infinitesimal isometry of a surface $F = (u, v, f)$ and let $c$
  be the function defined by Equation~\eqref{eq:inf}.
  Then $\bar B = (f_v, -f_u, -n_u f_v + n_v f_u - c)$ is a rotation diagram
  and $B = (f_u, f_v, f_u u + f_v v - f)$ is a translation diagram  
  for the reversed pair $(V, F)$.
\end{lem} 
\begin{proof} 
  We define the diagrams following Definition~\ref{defn:diagrams}.
  The definition of $B$ is straightforward.
  A corresponding rotation diagram $\bar B$ is of the form
  $\bar B = (f_v, -f_u, \bar b)$ with $\bar b$ satisfying the
  associated Equations~\eqref{eq:g} which here read
  $\bar b_u = n_v f_{uu} - n_u f_{uv}$ and  
  $\bar b_v = n_v f_{uv} - n_u f_{vv}$. The existence of $\bar b$ follows
  from Equation~\eqref{eq:inf}.
  Then, differentiating the third component of $\bar B$ from our statement
  by $u$ and using Equation~\eqref{eq:g} and yields 
  $-n_{uu} f_v - n_u f_{uv} + n_{uv} f_u + n_v f_{uu} - f_u n_{uv} + f_v
  n_{uu}$ and therefore equals $\bar b_u$. Analogously for $v$.
  Consequently, $\bar b = -n_u f_v + n_v f_u - c$ up to a constant.
\end{proof} 

\begin{prop} 
  Let $(C, \bar C)$ be the rotation and translation diagrams of the pair
  $(F, V)$ where $F$ is the infinitesimally flexible surface and $V$ a
  velocity diagram. 
  Furthermore, let $(\bar B, B)$ be the corresponding diagrams of $(V, F)$.
  Then,
  \begin{enumerate}
    \item the following three pairs are metric duals of each other:
      \begin{equation*}
        \bar C = \delta(V),
        \qquad
        B = \delta(F),
        \qquad
        \bar B = \delta(C).
      \end{equation*}
    \item $B$ and $\bar B$ are the rotation and translation diagrams for
      the pair $(F, V)$
  \end{enumerate}
\end{prop}
\begin{proof}
  Let us first write $B$ and $\bar B$ in the form of contact elements:
  \begin{equation}
    \label{eq:bbar}
    B = (f_u, f_v, f_u u + f_v v - f, u, v)
    \ \text{and}\ 
    \bar B = (f_v, -f_u, -n_u f_v + n_v f_u - c, -n_u, -n_v).
  \end{equation}
  Then Theorem~\ref{thm:diagrams}~\ref{itm:isov} and Lemma~\ref{lemma1}
  imply the dualities.

  The proof that $K(B, \bar B) = 0$ works similar to the proof of
  Theorem~\ref{thm:diagrams}~\ref{itm:isoiii} which implies that $\bar B$
  is a velocity diagram for the infinitesimally flexible surface $B$.
\end{proof} 

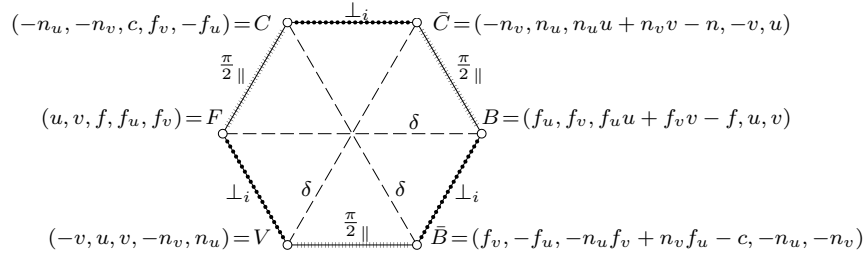
\begin{figure}[t]
  \hfill
  \hfill
  \begin{overpic}[width=.28\textwidth]{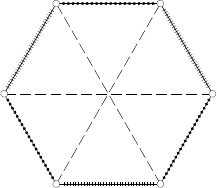}
    \put(20,02){\footnotesize\mlput{$(-v, u, v, -n_v, n_u) \!=\! V$}}
    \put(02,47){\footnotesize\mlput{$(u, v, f, f_u, f_v) \!=\! F$}}
    \put(20,82){\footnotesize\mlput{$(-n_u, -n_v, c, f_v, -f_u) \!=\! C$}}
    \put(80,82){\footnotesize\mrput{$\bar C \!=\! (-n_v, n_u, n_u u + n_v
    v - n, -v, u)$}}
    \put(98,47){\footnotesize\mrput{$B \!=\! (f_u, f_v, f_u u + f_v v - f, u, v)$}}
    \put(79,02){\footnotesize\mrput{$\bar B \!=\! (f_v, -f_u, -n_u f_v + n_v f_u - c, -n_u, -n_v)$}}
    \put(71,45){\footnotesize$\delta$}
    \put(66,20){\footnotesize$\delta$}
    \put(31,20){\footnotesize$\delta$}
    \put(3,19){\footnotesize$\perp_i$}
    \put(47,87){\footnotesize$\perp_i$}
    \put(88,19){\footnotesize$\perp_i$}
    \put(0,67){\footnotesize$\frac{\pi}{2}_\parallel$}
    \put(88,67){\footnotesize$\frac{\pi}{2}_\parallel$}
    \put(46,8){\footnotesize$\frac{\pi}{2}_\parallel$}
  \end{overpic}
  \hfill
  \hfill
  \hfill{}
  \caption{Relations in the isotropic Darboux wreath~I
  (Thm~\ref{thm:darbouxwreath-i}). The successive construction of rotation
  diagram and translation diagram closes after 3 steps resulting in 6
  surfaces in an isotropic Darboux wreath. These six surfaces are in part
  related additionally through metric duality $\delta$, I-orthogonality,
  and parallel tangent planes in corresponding points (after rotating
  about $\pi/2$).
  \\
  \includegraphics[width=.13\linewidth]{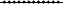}
  Infinitesimally flexible surface and velocity diagram. They are
  orthogonally related.
  \\
  \includegraphics[width=.13\linewidth]{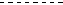}
  Correspond in the metric duality $\delta$.
  \\
  \includegraphics[width=.13\linewidth]{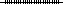}
  Parallel tangent planes in corresponding points if in each case the
  second surface is rotated by $\pi/2$ around the $z$-axis.
  }
  \label{fig:darboux1}
\end{figure}

We summarize the above in the following theorem which is depicted in
Figure~\ref{fig:darboux1}.

\begin{thm}[Darboux wreath~I]
  \label{thm:darbouxwreath-i}
   Let $F$ be an infinitesimally flexible surface with velocity diagram
   $V$. Then we have the following. 
   \begin{enumerate}
     \item The rotation diagram $C$ and translation diagram $\bar C$ of
       $(F, V)$ forms also a pair of infinitesimally flexible surfaces.
     \item The sequence of repeatedly constructing the rotation and
       translation diagrams returns to the pair $(F, V)$ and is called
       \emph{Darboux wreath}.
     \item The number of different surfaces in the Darboux wreath is 6.
     \item The three pairs $(V, F)$, $(C, \bar C)$ and $(B, \bar B)$
       consist of infinitesimally flexible surfaces together with a
       corresponding velocity diagram. They are orthogonally related
       (Def.~\ref{defn:orthogonally}).
     \item\label{itm:darbouxwreath-i-v} The three pairs $(V, \bar C)$,
       $(F, B)$ and $(C, \bar B)$ are related by the metric duality
       $\delta$.
     \item The three pairs $(F, C)$, $(\bar C, B)$ and $(\bar B, V)$ have
       parallel tangent planes in corresponding points if in each case the
       second surface is rotated by $\pi/2$ around the $z$-axis.
   \end{enumerate}
\end{thm}

Interestingly, we obtain another Darboux wreath with a different flavor if
we apply the (orientation reversing) Euclidean motion $L$ form
Equation~\eqref{eq:L} to every second surface in the Darboux wreath of
Theorem~\ref{thm:darbouxwreath-i}. It is depicted in
Figure~\ref{fig:darboux-ii}.

\begin{thm}[Darboux wreath~II]
  \label{thm:darbouxwreath-ii}
   Let $F$ be an infinitesimally flexible surface with velocity diagram
   $V$ and let $L$ denote the (orientation reversing) Euclidean motion
   form Equation~\eqref{eq:L}. Then we have the following. 
   \begin{enumerate}
     \item\label{itm:dwIIi} The three pairs $(L V, \bar C)$, $(L B, F)$
       and $(L C, \bar B)$ are related by the metric duality $\nu$.
     \item\label{itm:dwIIii} The three pairs $(L V, \bar C)$, $(L B, F)$
       and $(L C, \bar B)$ are related by Weingarten transformations.
     \item\label{itm:dwIIiii} The three pairs $(F, L C)$, $(\bar C, L B)$
       and $(\bar B, L V)$ have parallel tangent planes in corresponding
       points. They are relative minimal surfaces to each other.
     \item\label{itm:dwIIiv} The three pairs $(L V, F)$, $(L C, \bar C)$
       and $(L B, \bar B)$ each have the same top view. They are dual
       relative minimal surfaces to each other. 
   \end{enumerate}
\end{thm}
\begin{proof}
  Property~\ref{itm:dwIIi} follows directly from applying
  Lemma~\ref{lemma1}. 
  Corollary~\ref{cor:weingarten} and~\ref{itm:dwIIi} immediately
  imply~\ref{itm:dwIIii}.
  Property~\ref{itm:dwIIiii} follows directly from
  Proposition~\ref{prop:relmin} and property~\ref{itm:dwIIiv} follows from
  Proposition~\ref{prop:dualrelmin}.
\end{proof}

\begin{figure}[t]
  \hfill
  \hfill
  \begin{overpic}[width=.28\textwidth]{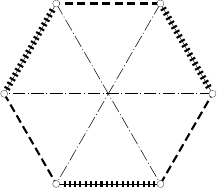}
    \put(20,02){\footnotesize\mlput{$(u, v, -n, -n_u, -n_v) \!=\! L V$}}
    \put(02,47){\footnotesize\mlput{$(u, v, f, f_u, f_v) \!=\! F$}}
    \put(20,82){\footnotesize\mlput{$(-n_v, n_u, -c, f_u, f_v) \!=\! L C$}}
    \put(80,82){\footnotesize\mrput{$\bar C \!=\! (-n_v, n_u, n_u u + n_v v - n, -v, u)$}}
    \put(99,47){\footnotesize\mrput{$L B \!=\! (f_v, -f_u, f - f_u u - f_v v, -v, u)$}}
    \put(79,02){\footnotesize\mrput{$\bar B \!=\! (f_v, -f_u, -n_u f_v + n_v f_u - c, -n_u, -n_v)$}}
    \put(71,45){\footnotesize$\nu$}
    \put(66,20){\footnotesize$\nu$}
    \put(31,20){\footnotesize$\nu$}
    \put(7,67){\footnotesize$\parallel$}
    \put(88,67){\footnotesize$\parallel$}
    \put(48,7){\footnotesize$\parallel$}
  \end{overpic}
  \hfill
  \hfill
  \hfill{}
  \caption{Relations in the isotropic Darboux wreath~II
  (Thm~\ref{thm:darbouxwreath-ii}). Applying the (orientation reversing)
  Euclidean motion $L$ form Equation~\eqref{eq:L} to every second surface
  in the isotropic Darboux wreath~I (Fig.~\ref{fig:darboux1}) 
  yields another isotropic Darboux wreath with 6 surfaces.
  These six surfaces are in part related additionally through metric
  duality $\nu$, parallel tangent planes in corresponding points, are
  (dual) relative minimal surfaces to each other and have same top views.
  \\
  \includegraphics[width=.13\linewidth]{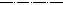}
  Correspond in the metric duality $\nu$.
  \\
  \includegraphics[width=.13\linewidth]{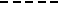}
  Have the same top view and are dual relative minimal surfaces to each
  other.
  \\
  \includegraphics[width=.13\linewidth]{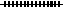}
  Have parallel tangent planes at corresponding points and are relative
  minimal surfaces to each other.
  }
  \label{fig:darboux-ii}
\end{figure}

Our definition of associated diagrams fits nicely to the statics
viewpoint. We may view $F$ as surface in equilibrium with $L V$ as Airy
stress surface. Then $L C$ is the reciprocal surface (force diagram) with
its Airy surface $\bar C$. It is known that the Airy surfaces are related
by metric duality in $I^3$
(cf.~\cite{Millar2021,vouga-2012-sss,pellis-optimal-2017}).
 
From the infinitesimal kinematics viewpoint, each surface in the wreath
has the meaning of an infinitesimally flexible surface, a velocity
diagram, a rotation diagram, a translation diagram and a metric dual.
Within statics, each surface is in equilibrium without inner loads, an
Airy surface, a reciprocal force diagram and a metric dual. From a purely
geometric perspective, each surface $S$ is relative minimal to a surface
$S_1$, dual relative minimal to $S_2$, and metric dual to $S_3$. Surfaces
$S_1, S_2, S_3$ are the ones to which the surface $S$ is connected by an
edge in the diagram of Figure~\ref{fig:darboux-ii}.

\subsection{Paratactic preimage surfaces in the isotropic Darboux wreath}

Each surface $F$ together with a velocity diagram $V$ gives rise to an
infinitesimal area preserving map in the plane which has a paratactic
preimage surface. Interestingly this paratactic preimage surface is also
included in the Darboux wreath. 

\begin{prop} 
  \label{prop:para-wreath}
  An infinitesimally flexible surface $F = (u, v, f)$ with velocity
  diagram  $V = (-v, u, n)$ defines two isometric surfaces 
  $F^- = (u, v, f - n)$ and $F^+ = (u, v, f + n)$ whose dual images
  determine an area preserving map $\nu(F^-) \mapsto \nu(F^+)$. The
  paratactic preimage surface of the top view of this map agrees with the
  rotation diagram $\bar B$ of $V$. 
\end{prop} 
\begin{proof} 
  By Proposition~\ref{prop:isofromdiagr}, the surfaces $(u, v, f - n)$ and
  $(u, v, f + n)$ are isometric to each other and by
  Proposition~\ref{prop:isometricdual} the map $\nu(u, v, f - n) \mapsto
  \nu(u, v, f + n)$ is area preserving. By Lemma~\ref{lemma1}, the top
  view of this map is described by $(f_v - n_v, -f_u + n_u) \mapsto (f_v +
  n_v, -f_u - n_u)$. We obtain the contact element representation of the
  paratactic preimage surface $(f_v, -f_u, z(u, v), -n_u, -n_v)$ by
  applying Equation~\eqref{eq:para2} for some height function $z(u, v)$. 
  Comparison with Equations~\eqref{eq:bbar} yields $z = -n_u f_v + n_v f_u
  - c$ which further implies that the paratactic preimage of this area
  preserving map is the surface $\bar B$.
\end{proof} 

Thus, all similar paratactic preimage surfaces are seen in the Darboux wreath. 
Figure~\ref{fig:wreath-para} illustrates how to find this surface.

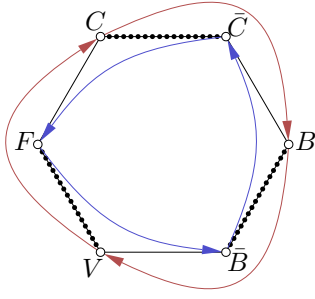
\begin{SCfigure}[2.9][tb]
  \begin{overpic}[width=.30\textwidth]{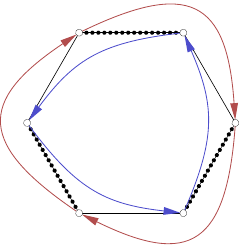}
    \put(3,48){$F$}
    \put(26,5){$V$}
    \put(75,7){$\bar B$}
    \put(98,48){$B$}
    \put(75,87){$\bar C$}
    \put(27,89){$C$}
  \end{overpic}
  \caption{Paratactic preimage surfaces in the Darboux wreath. For any
  surface in the wreath, take the neighboring velocity diagram of that
  surface (dotted edge) and then the next surface in the wreath is the
  paratactic preimage as explained in Prop.~\ref{prop:para-wreath}.}
  \label{fig:wreath-para}
\end{SCfigure}

\section{Special infinitesimally flexible nets and their discrete
counterparts} 
\label{sec:special}

In this section we investigate properties and behaviors of special
parametrizations of infinitesimally flexible surfaces.
We do this for three reasons.
First they lead to further interesting results. Second, particular
parametrizations are important in applications. 
And third, parametrizations lead to easily accessible discrete models,
which enhance the understanding of the multiple relations we have found.
Moreover, they provide new views of concepts in discrete differential
geometry, especially K{\oe}nigs nets~\cite{bobenko-2008-ddg}.

\subsection{Infinitesimal isometries of nets which induce infinitesimal
isometries of transversal ruled surfaces} 
\label{ssec:transruled}

We consider a parameterization, or net, $F(u, v)$ on a surface and
associate with it two families of ruled surfaces as follows. 

\begin{defn}
  Along an $u$-parameter curve ($v = v_0 = \text{const.}$), the tangents
  to the $v$-parameter curves form the \emph{transversal ruled surface} 
  parametrized by $R(u, t) = F(u, v_0) + t F_v(u, v_0)$. See
  Figure~\ref{fig:transversal-ruled-surfaces} (left).
  Analogously, we define transversal ruled surfaces along $v$-parameter
  curves.
\end{defn}

\begin{figure}[tb]
  \begin{overpic}[width=.2\textwidth]{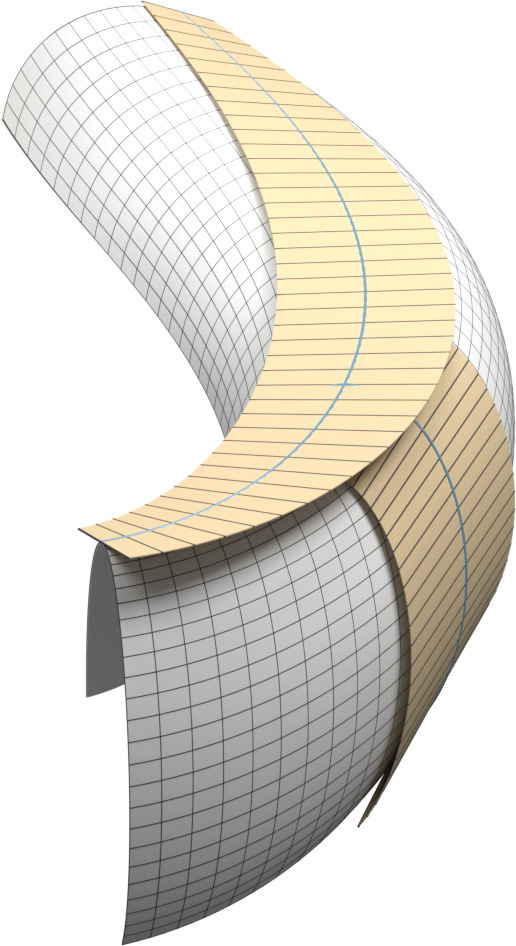}
    \put(0,50){\small$R(u, t)$}
    \put(0,80){\small$F(u, v)$}
  \end{overpic}
  \hfill
  \begin{overpic}[width=.7\textwidth]{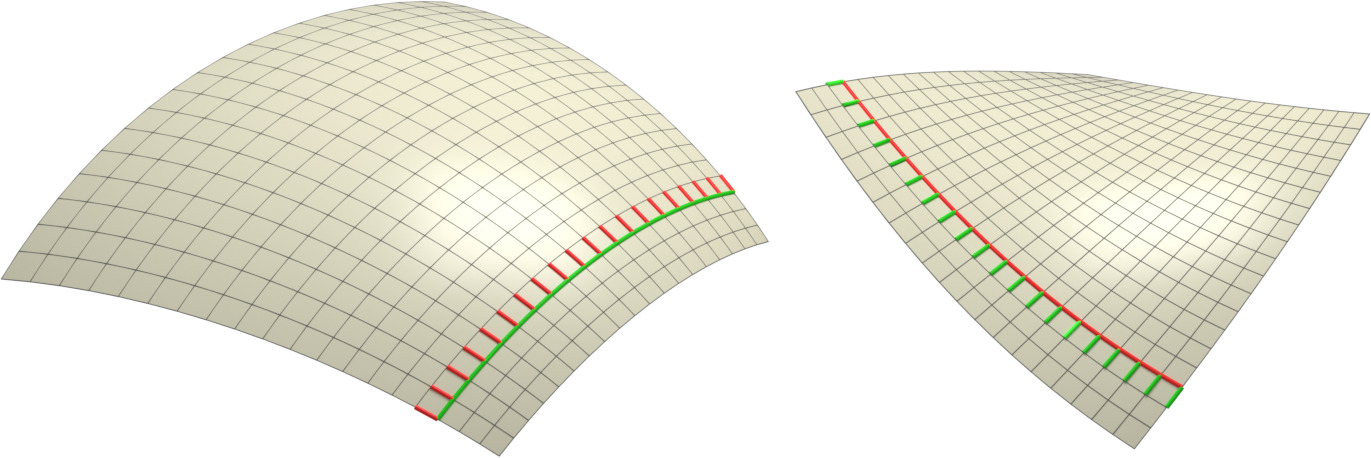}
    \put(42,3){$F$}
    \put(67,7){$G$}
  \end{overpic}
  \caption{\emph{Left:} A transversal ruled surface is formed by the
  tangents of the $v$-parameter curves along a $u$-parameter curve and
  vice versa.
  \emph{Right:} Reciprocal parallel nets $F$ and $G$. 
  The partial derivatives $F_u$ are parallel to $G_v$ along each $v$-line
  and vice versa.
  }
  \label{fig:transversal-ruled-surfaces}
\end{figure}

If $F$ undergoes a continuous isometric deformation, the transversal ruled
surfaces do in general not transform isometrically. Such isometries would
be Minding isometries, as discussed in Sections~\ref{ssec:simpleexamples}
and~\ref{ssec:2positions}. However, we will now study precisely those
infinitesimal isometries of nets $F(u, v)$ where the transversal ruled
surfaces undergo infinitesimal Minding isometries. We derive properties
of the corresponding associated diagrams.
The notion of reciprocal parallelity plays an important role.

\begin{defn}
  Let $F(u, v)$ and $G(u, v)$ parametrize two surfaces. They are called
  \emph{reciprocal parallel} if 
  $F_u \parallel G_v$ and 
  $F_v \parallel G_u$ at all points (see
  Figure~\ref{fig:transversal-ruled-surfaces} right).
\end{defn}

\begin{prop}
  \label{prop:smoothrigidfaces}
  For the diagrams $V, C, \bar C$ associated with an infinitesimally
  flexible net $F(u, v)$, the following properties are equivalent:
  \begin{enumerate}
    \item\label{itm:smoothrigidfaces-i} The infinitesimal isometry of $F$
      induces an infinitesimal isometry of all transversal ruled surfaces.
    \item\label{itm:smoothrigidfaces-ii} $V$ is a Q-net.
    \item\label{itm:smoothrigidfaces-iii} $\bar C$ is a Q-net.
    \item\label{itm:smoothrigidfaces-iv} The infinitesimal isometry of $C$
      induces an infinitesimal isometry of all transversal ruled surfaces.
    \item\label{itm:smoothrigidfaces-v} $L C$ is reciprocal parallel to $F$.
  \end{enumerate}
\end{prop}
\begin{proof}
  \ref{itm:smoothrigidfaces-i} $\Leftrightarrow$ \ref{itm:smoothrigidfaces-ii}:
  Let $F$ be an infinitesimally flexible surface with velocity diagram
  $V$. Then by Corollaries~\ref{cor:rulingpresisom}
  and~\ref{cor:torsaltoisom}, the associated transversal ruled surface
  $S(u, t) = V(u, v_0) + t V_v(u, v_0)$ is a velocity diagram for an
  infinitesimal isometry of $R(u, t) = F(u, v_0) + t F_v(u, v_0)$ if and
  only if $S$ is torsal. It is well known (see,
  e.g.,~\cite[p.~94]{sauer:1970}) that surfaces where transversal ruled
  surfaces are torsal along all parameterlines are precisely the Q-nets
  (i.e., conjugate parametrizations). Consequently, $V$ is a Q-net if and
  only~\ref{itm:smoothrigidfaces-i} holds.

  \ref{itm:smoothrigidfaces-ii} $\Leftrightarrow$
  \ref{itm:smoothrigidfaces-iii}:
  Any projective duality maps a Q-net onto a Q-net. 
  Theorem~\ref{thm:darbouxwreath-i}~\ref{itm:darbouxwreath-i-v} implies
  that $\bar C = \delta(V)$. Consequently, $V$ is a Q-net if and only if
  $\bar C$ is a Q-net.

  \ref{itm:smoothrigidfaces-iii} $\Leftrightarrow$
  \ref{itm:smoothrigidfaces-iv}:
  This follows from the symmetries in the Darboux wreath~I
  (Theorem~\ref{thm:darbouxwreath-i}) in analogy to 
  ``\ref{itm:smoothrigidfaces-i} $\Leftrightarrow$
  \ref{itm:smoothrigidfaces-ii}''.

  \ref{itm:smoothrigidfaces-ii}
  $\Leftrightarrow$
  \ref{itm:smoothrigidfaces-v}:
  Let us assume $V$ is a Q-net which is the same as $L V$ being a Q-net. 
  Equivalently all transversal ruled surfaces of $L S$ are torsal.
  The duality $\nu$ maps torsal surfaces via their enveloping tangent
  planes to parameter curves of $\bar C$. The rulings along $u$-curves
  (they have direction $L V_v$) become the tangents of $u$-curves of $\bar
  C$. Recall that the top views of a straight line and its $\nu$-image are
  parallel (Lem.~\ref{lem:topviewoflines}). Therefore, the top view of
  $L V_v$ and $\bar C_u$ are parallel and vice versa. 
  Consequently, the top views of $L V$ and $\bar C$ are reciprocal
  parallel if and only if $V$ is a Q-net.

  Since the top views of $F$ and $L V$ are the same as well as the top
  views of $\bar C$ and $L C$
  (Thm.~\ref{thm:darbouxwreath-ii}~\ref{itm:dwIIiv}), also the top views
  of $F$ and $L C$ are reciprocal parallel.
  Consequently, the top views of $F$ and $L C$ are reciprocal parallel if
  and only if $V$ is a Q-net.

  Furthermore, the parallelity of the tangent planes of $F$ and $L C$ 
  (Thm.~\ref{thm:darbouxwreath-ii}~\ref{itm:dwIIiii}) implies that $F$ is
  reciprocal parallel to $L C$ if and only if $V$ is a Q-net.
\end{proof}

\paragraph{Dual reciprocal parallel nets}

In the setting of Proposition~\ref{prop:smoothrigidfaces} we have that $F$
and $L C$ are reciprocal parallel. What is the corresponding property of
their metric duals $L B = \nu(F)$ and $\bar B = \nu(L C)$ in the Darboux
wreath~II (see Fig.~\ref{fig:darboux-ii})? We call the relation between $L
B$ and $\bar B$ \emph{dual reciprocal parallel}.

Recall that $F(u,v)$ and $L C(u,v)$ being reciprocal parallel means that
they have parallel tangent planes and $F_u \parallel L C_v$ as well as $F_v
\parallel L C_u$. The duality $\nu$ maps parallel tangent planes to points
with identical top views. Consequently, $L B (u, v)$ and $\bar B(u, v)$
have identical top views.

The points of a parameter curve $F(u, v_0)$ map to tangent planes of $L
B(u, v_0)$. Thus the tangent lines of the parameter curve map to the rulings
of the envelope of the family of planes and therefore to the line conjugate
to the tangent line of $L B(u, v_0)$.
The observation that parallel lines map to lines with the same top view
now implies the following proposition.

\begin{prop}
  Corresponding points $L B(u, v)$ and $\bar B(u, v)$ have the same top
  view. Furthermore, the conjugate directions to $L B_u$ and $\bar B_v$
  agree in the top view as well as the conjugate directions to $L B_v$ and
  $\bar B_u$ agree in the top view. 
\end{prop}

\section{Infinitesimal isometries of discrete nets, K{\oe}nigs nets and
Voss-nets} 
\label{sec:discrete}

In this section we study a discrete version of the previous section.
It leads to further insights and it is also relevant for applications.
Now, $F, V, C, \bar C, B, \bar B$ are discrete nets in the sense that 
$F \colon \Z^2 \to \R^3$, etc. Two neighboring vertices 
$F(i, j) F(i + 1, j)$ or $F(i, j) F(i, j + 1)$ form an \emph{edge} of the
net and four vertices 
$F(i, j) F(i + 1, j) F(i + 1, j + 1) F(i, j + 1)$ form a (combinatorial)
\emph{face} $\ff_{ij}$. We often use the index notation $F_{ij} =
F(i, j)$. A geometric realization of such a face is not essential for
the discrete theory. Naturally, if the face is planar we typically think
of it represented by a planar quadrilateral piece of a plane.

\subsection{Infinitesimal isometries of discrete nets with rigid faces} 
\label{ssec:rigidfaces}

There are several possible discretizations of infinitesimally flexibility
of \emph{discrete} nets. 
The following definition leads us to results which are similar to the
previous section. In analogy to the ``smooth'' definition
(Def.~\ref{defn:inffelxsmooth}) we need metric isometry and preservation
of Gauss curvature of first order. Metric isometry will be enforced by
congruent top views of corresponding nets. Infinitesimal isometric
deformation of frames (cf.\ discussion on page~\pageref{frames}) will be
enforced by infinitesimal rigid transformation of each face. And
preservation of Gauss curvature will be enforced by a discrete version of
Corollary~\ref{cor:mixed}. For that recall that the mixed area of two
quadrilaterals $P, Q$ in $\R^2$ is 
$\area(P, Q) 
= \frac{1}{4} \sum_{i = 1}^4 \det(p_i, q_{i + 1}) + \det(q_i, p_{i + 1})$
(see, e.g.,~\cite{mueller+2010}).

Note that the metric dual of a mesh $\M$ with \emph{non-planar} faces can be
understood as the collection of lines and planes obtained by dualizing the
lines and vertices of $\M$. The dual $\nu(M)$ is then a collection of
faces (``\emph{non-copunctual} vertices'', see
Fig.~\ref{fig:non-copunctal} left).

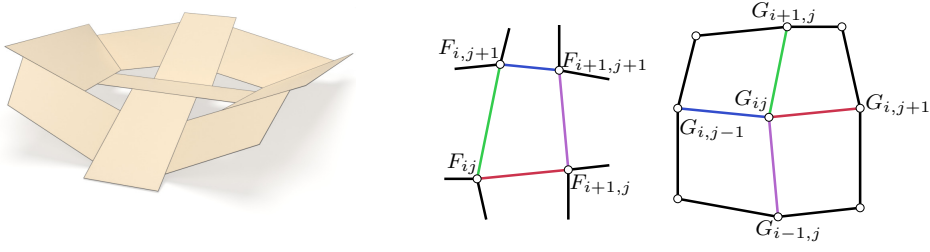
\begin{figure}[t]
  \begin{overpic}[width=.40\textwidth]{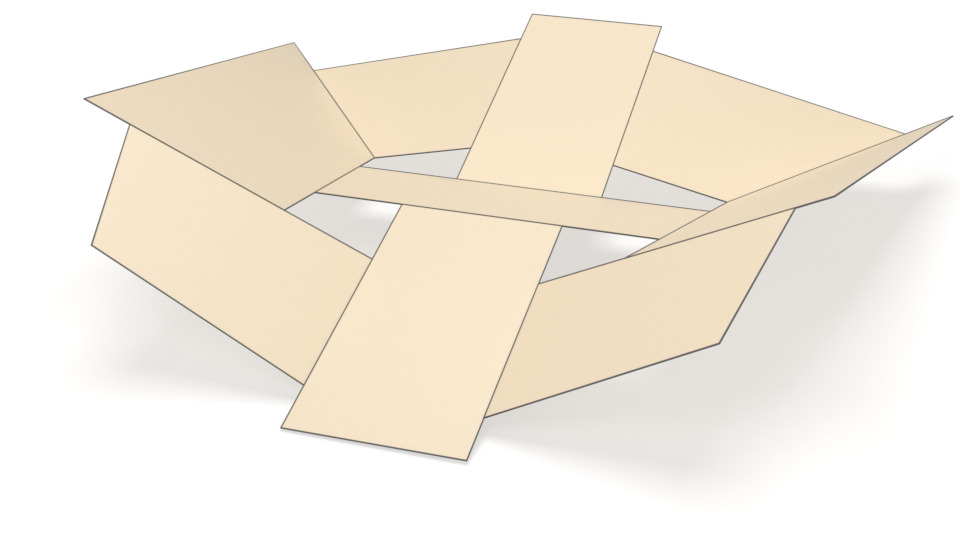}
  \end{overpic}
  \hfill
  \hfill
  \hfill
  \hfill
  \begin{overpic}[width=.44\textwidth]{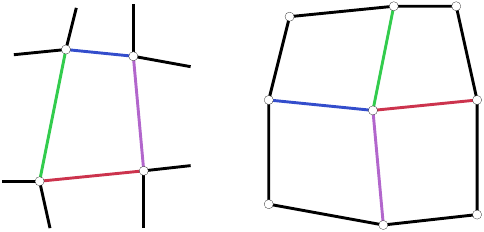}
    \put(1,13){\footnotesize$F_{ij}$}
    \put(30,8){\footnotesize$F_{i + 1, j}$}
    \put(28,37){\footnotesize$F_{i + 1, j + 1}$}
    \put(-1,40){\footnotesize$F_{i, j + 1}$}
    \put(69,28){\footnotesize$G_{ij}$}
    \put(100,27){\footnotesize$G_{i, j + 1}$}
    \put(56,21){\footnotesize$G_{i, j - 1}$}
    \put(73,49){\footnotesize$G_{i + 1, j}$}
    \put(74,-3){\footnotesize$G_{i - 1, j}$}
  \end{overpic}
  \hfill
  \hfill
  \hfill{}
  \caption{\emph{Left:} A net with edges and planar faces but with
  ``non-copunctal'' vertices. It is projective dual to a net with
  non-planar faces.
  \emph{Right:} Discrete reciprocal parallel nets. The edges around each
  face of one net are parallel to the edges around the corresponding
  vertex on the other net and vice versa. Parallel edges have the same
  color.}
  \label{fig:non-copunctal}
\end{figure}

\begin{defn}
  \label{defn:discrinfintiso}
  An isotropic net $F \colon \Z^2 \to I^3$ is \emph{infinitesimally
  flexible} if there exists a non-trivial discrete deformation vector
  field (defined on the vertices) $V \colon \Z^2 \to I^3$ such that:
  \begin{enumerate}
    \item The top view of $L V$ equals the top view of $F$.
    \item Each corresponding pair of faces $\bb, \cc$ of the dual nets
      $L B : = \nu(F)$, $\bar C := \nu(L V)$ has vanishing mixed area
      $\area(\bb, \cc) = 0$.
    \item $V$ is compatible with infinitesimal isometries of the faces in
      the following sence. For each face $\ff$ there is an infinitesimal
      isometry described by $\bar D_\ff$ and $T_\ff$ as in
      Equation~\eqref{eq:isotropicvelicitydiag}, such that 
      \begin{align*}
        V_{ij} &= \bar D_{\ff_{ij}} + T_{\ff_{ij}} F_{ij}
        &
        V_{ij} &= \bar D_{\ff_{i - 1, j}} + T_{\ff_{i - 1, j}} F_{ij}
        \\
        V_{ij} &= \bar D_{\ff_{i, j - 1}} + T_{\ff_{i, j - 1}} F_{ij}
        &
        V_{ij} &= \bar D_{\ff_{i - 1, j - 1}} + T_{\ff_{i - 1, j - 1}} F_{ij}.
      \end{align*}
  \end{enumerate}
\end{defn}

Note that since the top views of the edges of $LV$ and $F$ are the same,
the edges of $\bb$ and $\cc$ are parallel. Thus, $\area(\bb, \cc)$ is the
mixed area of two parallel quadrilaterals.

The edges along a discrete parameter line are considered discrete
tangents. The family of edges connecting two neighboring parameter lines
form \emph{discrete transversal ruled surfaces}. Consequently, per
definition in a discrete infinitesimal isometry the discrete
transversal ruled surfaces also undergo an infinitesimal isometry.

A \emph{discrete Q-net} $F$ is a net in space with the combinatorics of a
sublattice of the $\Z^2$ lattice and such that each face is planar. 

\begin{lem}
  \label{lem:rigidfaces} 
  The velocity diagram $V$ associated with a discrete infinitesimally
  flexible net $F$ is a Q-net.
\end{lem}
\begin{proof}
  Definition~\ref{defn:discrinfintiso} implies the existence of an
  infinitesimal isometry given by $\bar D_{\ff_{ij}}$ and $T_{\ff_{ij}}$
  such that all velocity vectors are compatible with that infinitesimal
  isometry. Let us therefore consider an edge $F_{ij} F_{i + 1, j}$ of
  this face $\ff_{ij}$ and the corresponding velocity vectors 
  $V_{ij} V_{i + 1, j}$. We have 
  $V_{ij} = \bar D_{\ff_{ij}} + T_{\ff_{ij}} F_{ij}$
  and
  $V_{i + 1, j} = \bar D_{\ff_{ij}} + T_{\ff_{ij}} F_{i + 1, j}$ 
  which implies
  $V_{ij} - V_{i + 1, j} = T_{\ff_{ij}} (F_{ij} - F_{i + 1, j})$.
  This holds for all edge vectors of the velocity diagram which
  corresponds to face $\ff_{ij}$. Since $T_{\ff_{ij}}$ has at most rank $2$
  the face $V_{ij} V_{i + 1, j} V_{i + 1, j + 1} V_{i, j + 1}$ must lie in
  a plane. Consequently, $V$ is a Q-net.
\end{proof}

In the spirit of the Darboux wreath let us define the \emph{discrete
translation diagram} $\bar C$ as the dual of $L V$:
\begin{equation*}
  \bar C = \nu(L V).
\end{equation*}
As a metric dual image of a Q-net, $\bar C$ is also a Q-net.

\subsection{K{\oe}nigs nets and infinitesimal isometries of discrete nets} 

Apart from projective duality, there is another ``type'' of duality for a
subclass of conjugate nets, the discrete K{\oe}nigs duality (see,
e.g.,~\cite{bobenko-2008-ddg}).
The most interesting aspect now is the characterization of the relations
between $L V$ and $F$ as well as its metric dual, the relation between
$\bar C$ and $L B$. 

\begin{defn}
  \label{defn:koenigs}
  Two discrete Q-nets are related via \emph{K{\oe}nigs duality} if
  corresponding faces have parallel corresponding edges, and parallel
  non-corresponding diagonals (see Fig.~\ref{fig:wreath-iii} right).
  A Q-net is called \emph{K{\oe}nigs net} if a K{\oe}nigs dual net exists.
\end{defn}

\begin{prop} 
  \label{prop:rigidplanarfaces}  
  Let $F : \Z^2 \to I^3$ be an infinitesimally flexible Q-net with
  velocity diagram $V$. Then the metric duals $L B = \nu(F)$ and $\bar C =
  \nu(V)$ are K{\oe}nigs nets and related via K{\oe}nigs duality.
  Thus $F$ and $V$ are metric duals of K{\oe}nigs nets and the relation
  between them is the metric dual to the K{\oe}nigs duality.
\end{prop}
\begin{proof} 
  Lemma~\ref{lem:rigidfaces} implies that the velocity diagram $V$ is a
  Q-net and therefore also $\bar C = \nu(L V)$ is a Q-net.
  Since $F$ is a Q-net also $LB = \nu(F)$ is a Q-net.

  Corresponding edges of $F$ and $L V$ have the same top view. Therefore,
  corresponding edges of $L B$ and $\bar C$ are parallel. Furthermore, the
  top views of corresponding faces $\bb$ and $\cc$ of $L B$ and $\bar C$
  are edgewise parallel quadrilaterals with the property $\area(\bb, \cc) = 0$.
  It is well-known that two parallel quadrilaterals with parallel edges
  have vanishing mixed area if and only if non-corresponding diagonals are
  parallel (see~\cite{bobenko-2008-ddg}). This is exactly the
  characterization of discrete K{\oe}nigs duality
  (Def.~\ref{defn:koenigs}). Consequently, the top views of $\bar C$ and
  $L B$ are K{\oe}nigs nets and K{\oe}nigs dual to each other. Since
  parallelity of non-corresponding diagonals is affinly invariant, this
  property holds for $\bar C$ and $L B$ which implies K{\oe}nigs duality
  for those nets as well.
\end{proof}

\begin{defn}
  Two discrete nets $F_{ij}$ and $G_{ij}$ are called \emph{reciprocal
  parallel} if 
  \begin{align*}
    \Delta_i F_{ij} &\parallel \Delta_j G_{ij}
    &\Delta_i F_{i, j + 1} &\parallel \Delta_j G_{i, j - 1}
    \\
    \Delta_j F_{ij} &\parallel \Delta_i G_{ij}
    &\Delta_j F_{i + 1, j} &\parallel \Delta_i G_{i - 1, j}.
  \end{align*}
  For an illustration see Figure~\ref{fig:non-copunctal} (right).
\end{defn}

\begin{lem}
  \label{lem:reciprocal}
  Let $F$ be an infinitesimally flexible Q-net with velocity diagram $V$. 
  There exists an A-net $L C$ (i.e., it has planar vertex stars) which is
  reciprocal parallel to $F$ and such that its top view and the top view
  of $\bar C$ are congruent.
  Furthermore, $\bar B := \nu(L C)$ is a reciprocal dual A-net to $L V$.
\end{lem}
\begin{proof}
  By Proposition~\ref{prop:rigidplanarfaces}, the nets $\bar C$ and $L B$
  are discrete K{\oe}nigs nets. In~\cite{AGAG-2023} we show that there is
  always an A-net $A$ whose top view is the same as the top view of $L B$.
  The metric duality $\nu$ maps $L B$ to $F$, lines (edges) with the
  same top view to parallel lines (edges), and A-nets to A-nets.
  Consequently, $\nu(A)$ is an A-net with edges parallel to the edges of
  $F$. It is therefore reciprocal parallel to $F$. The edges of the top
  view of $\nu(A)$ are parallel to the edges of $L B$ and therefore
  parallel to the edges of the top view of $\bar C$.

  Let us consider a non-planar face $\aaa$ of $A$ with vertices $a_{ij}$,
  etc. The vertex planes $\alpha_{ij}$ in $a_{ij}$ and $\alpha_{i + 1, j}$
  in $a_{i + 1, j + 1}$ intersect in the line $a_{i + 1, j} a_{i, j + 1}$.
  The metric duality maps the planes to points and therefore the
  intersection line to the connecting line of $\nu(\alpha_{ij})$ and
  $\nu(\alpha_{i + 1, j + 1})$. The top views of the lines remain parallel
  under $\nu$ which implies that the top views of $\nu(A)$ and $A$ have
  parallel edges and parallel non-corresponding diagonals, hence the top
  views are K{\oe}nigs dual K{\oe}nigs nets. Since the top view of $A$ is
  the top view of $L B$, the top view of $A$ is homothetic to the top view
  of $\bar C$ as the K{\oe}nigs dual is unique up to translation and
  scaling. 

  Let us denote by $L C$ the scaled version of $\nu(A)$ such that the
  top views of $L C$ and $\bar C$ are congruent. Then by the same argument
  as above $\bar B := \nu(L C)$ is an A-net reciprocal dual to $L V$.
\end{proof}

If $F$ is a Q-net then the Darboux wreath~II specializes to the one shown
in Figure~\ref{fig:wreath-iii} (left).

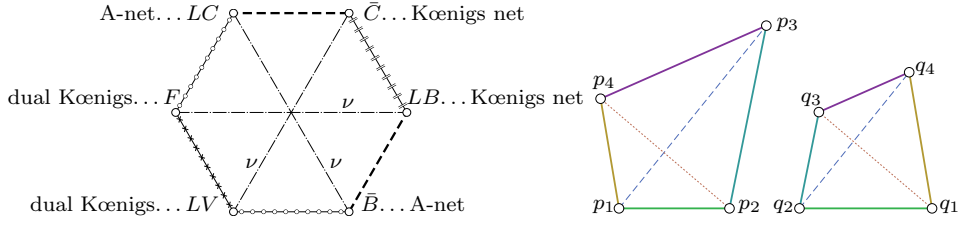
\begin{figure}[t]
  \hfill
  \hfill
  \hfill
  \hfill
  \hfill
  \begin{overpic}[width=.25\textwidth]{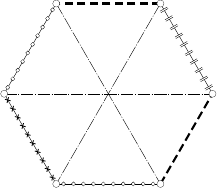}
    \put(20,02){\footnotesize\mlput{dual K{\oe}nigs\ldots $L V$}}
    \put(04,47){\footnotesize\mlput{dual K{\oe}nigs\ldots $F$}}
    \put(20,82){\footnotesize\mlput{A-net\ldots $L C$}}
    \put(80,82){\footnotesize\mrput{$\bar C$\ldots K{\oe}nigs net}}
    \put(98,47){\footnotesize\mrput{$L B$\ldots K{\oe}nigs net}}
    \put(79,02){\footnotesize\mrput{$\bar B$\ldots A-net}}
    \put(71,45){\footnotesize$\nu$}
    \put(66,20){\footnotesize$\nu$}
    \put(31,20){\footnotesize$\nu$}
  \end{overpic}
  \hfill
  \hfill
  \hfill
  \hfill
  \hfill
  \begin{overpic}[width=.36\textwidth]{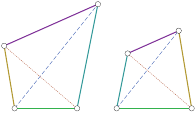}
    \put(0,2){\footnotesize$p_1$}
    \put(42,2){\footnotesize$p_2$}
    \put(53,55){\footnotesize$p_3$}
    \put(0,38){\footnotesize$p_4$}
    \put(100,2){\footnotesize$q_1$}
    \put(52,2){\footnotesize$q_2$}
    \put(60,33){\footnotesize$q_3$}
    \put(93,42){\footnotesize$q_4$}
  \end{overpic}
  \hspace{1ex}
  \caption{\emph{Left:} Relations in the isotropic Darboux wreath for a
  discrete infinitesimally flexible Q-net $F$ with rigid faces. 
  \emph{Right:} A pair of K{\oe}nigs dual faces. Corresponding edges are
  parallel (e.g., $p_2 p_3 \parallel q_2 q_3$). Non-corresponding
  diagonals are also parallel (e.g., $p_1 p_3 \parallel q_2 q_4$).
  \\
  \includegraphics[width=.15\linewidth]{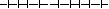}
  K{\oe}nigs duality
  \\
  \includegraphics[width=.15\linewidth]{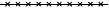}
  metric dual of K{\oe}nigs duality
  \\
  \includegraphics[width=.15\linewidth]{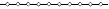}
  reciprocal parallel
  \\
  \includegraphics[width=.15\linewidth]{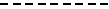}
  dual reciprocal parallel
  \\
  \includegraphics[width=.15\linewidth]{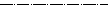}
  metric dual
  }
  \label{fig:wreath-iii}
\end{figure}

There is another relation which fits well to the smooth theory concerning
the fact that $L B$ and $\bar C$ form a pair of relative minimal surfaces.
According to the curvature theory based on mesh
parallelism~\cite{bobenko-2010-curv}, the relative mean curvature of $L
B$ w.r.t.\ $\bar C$ is in each face $\bb$ of $L B$ equal to 
$\area(\bb, \cc)/\area(\bb)$ and thus equal to zero. Hence, we see
the Q-nets $L B$ and $\bar C$ as \emph{relative minimal surfaces} of each
other. In fact, due to the parallelism, as \emph{relative principal
curvature nets}.  The metric dual relation between $F$ and $L V$ based on
point parallelism appears also in~\cite{vouga-2012-sss}.

\paragraph{The statics viewpoint}

The interpretation in terms of statics is of practical relevance. $F$ is a
net in equilibrium without external forces applied to inner vertices,
where forces act on edges only. The net $L C$ is the reciprocal force
diagram of graphic statics. The symmetry between $F$ and $L C$ is
well-known. Both are in equilibrium and one is the force diagram of the
other. The associated diagrams $V$ and $\bar C$ are the Airy stress
polyhedra, which are found already in the work of Maxwell and continue to
be used in computational structural design (see,
e.g.,~\cite{Millar2021,vouga-2012-sss}). The smooth counterpart of the
previous subsection concerns so-called truss-like continua. Most
importantly, a net on an equilibrium surface $F$, which is
itself in equilibrium (``Querspannungsnetz'' according to
Sauer~\cite{sauer:1970}) corresponds to a conjugate net on the Airy stress
surface $V$. Airy surfaces do not appear in Sauer's treatment.
In applications, one often uses properties of the smooth
counterpart to obtain good initial guesses for the computation of discrete
versions via numerical optimization (see,
e.g.,~\cite{Millar2021,vouga-2012-sss}).

\subsection{The isotropic Darboux wreath of special surfaces}
\label{ssec:darboux-wreath-examples}

In this section we will look at two special Examples of surface
classes namely isotropic linear Weingarten surfaces and Voss nets. We
investigate the remarkable nets which appear in the Darboux wreath in
these two cases.

\paragraph{Isotropic linear Weingarten surfaces} 

Inspired by recent work of Tellier et al.~\cite{TELLIER-linearWeingarten},
we take as initial surface $F(u, v) = (u, v, f(u, v))$ an isotropic linear
Weingarten surface whose isotropic curvatures $H(F), K(F)$ satisfy a
linear relation of the form
\begin{equation*}
  a H(F) + b K(F) =0,
\end{equation*}
for some $a, b \in \R$. Let 
\begin{equation*}
  n(u, v) = \frac{a}{2} (u^2 + v^2) + b f(u, v),
\end{equation*}
which implies 
\begin{align*}
  n_u = a u + b f_u,
  \quad
  n_v = a v + b f_v,
  \quad
  n_{uu} = a + b f_{uu},
  \quad
  n_{uv} = b f_{uv},
  \quad
  n_{vv} = a + b f_{vv}.
\end{align*}
Then, the surface $V(u, v) = (-v, u, n(u, v))$ is an associated velocity
diagram, since 
\begin{align*}
  K(F, V) 
  &= 
  f_{uu} n_{vv} - 2 f_{uv} n_{uv} + f_{vv} n_{uu}
  \\
  &=
  f_{uu} (a + b f_{vv}) - 2 f_{uv} b f_{uv} + f_{vv} (a + b f_{uu})
  \\
  &=
  a (f_{uu} + f_{vv}) + 2 b (f_{uu} f_{vv} - f_{uv}^2)
  =
  2 a H(F) + 2 b K(F) 
  =
  0.
\end{align*}
It follows immediately that $V$ is itself also an isotropic linear
Weingarten surface to the curvature relation
\begin{equation*}
  -a H(V) + K(V) = 0.
\end{equation*}
By Equation~\eqref{eq:dualcurv}, the surfaces $\bar C = \nu(L V)$ and $L B
= \nu(F)$ possess constant isotropic mean curvature
\begin{equation*}
  H(L B) = -\frac{H(F)}{K(F)} = \frac{b}{a}, 
  \quad\text{and}\quad 
  H(\bar C) = -\frac{H(V)}{K(V)} = -\frac{1}{a}.
\end{equation*}

The Hessian of $n$ reads $\nabla^2 n = a I + b \nabla^2 f$ where $I$
denotes the identity matrix. Consequently, the isotropic principal
curvature directions on $F$ and $L V$ have the same top view. Hence,
isotropic principal curvature lines of $F$ and $L V$ correspond to each
other in the point-parallelism. 

Let us now consider $F$ and $L V$ parametrized by curvature lines. Since a
curvature line parametrization is a Q-net,
Proposition~\ref{prop:smoothrigidfaces} implies that this infinitesimal
isometry induces infinitesimal isometries of the transversal torsal ruled
surfaces (dual K{\oe}nigs nets). Applying metric duality $\nu$, we obtain
surfaces $L B$ and $\bar C$ on which isotropic principal curvature line
nets are K{\oe}nigs nets that correspond in the K{\oe}nigs duality. Nets
$\bar B$ and $L C$ are reciprocal parallel to the orthogonal Q-nets $L V$
and $F$. Thus they are orthogonal A-nets and thus represent isotropic
minimal surfaces. 

In the discrete setting, we take as the discrete principal nets isotropic
conical nets or circular nets. Thus $F$ and $L V$ are infinitesimally
flexible isotropic principal curvature nets. Both nets have a
self-stressing mode and may be realized as a cable-net with planar faces,
provided that the surface is negatively curved and thus the equilibrium
state is stable. The metric dual nets $L B$ and $\bar C$ are also
isotropic circular or conical nets, respectively. They are K{\oe}nigs nets
in equilibrium with forces given by the reciprocal parallel A-nets $\bar
B$ and $L C$, respectively. Due to the parallelism of corresponding edges
in a reciprocal parallel pair and the fact that edges of different
parameter lines at a vertex are discrete orthogonal in the principal nets,
they are also discrete orthogonal in the A-nets. Being I-orthogonal
A-nets, the nets $\bar B$ and $L C$ are therefore discrete isotropic
minimal surfaces.

\paragraph{Isotropic Voss nets} 

Since we did not talk much about continuous (finite) isometric
deformations so far, we consider now the isotropic counterpart to a famous
example of flexible Q-nets, known as \emph{Voss nets} or briefly
\emph{V-nets}. These are discrete counterparts to surfaces with a
conjugate net of geodesics, first studied by A.~Voss~\cite{voss1888uber}.
They are reciprocal parallel nets to discrete surfaces of constant
Gaussian curvature
(\emph{K-nets})~\cite{sauer:1970,Schief2008,wunderlich-1951}.

Isotropic V-nets recently appeared in a study of 4-webs in $I^3$ which are
formed by asymptotic and geodesic curves~\cite{AGAG-2023}. There, a
special case of V-nets occurred, but most of the results hold for general
V-nets in $I^3$. We outline them and discuss all nets in the Darboux
wreath.

A discrete V-net $F$ in $I^3$ is a Q-net in which both families of
parameter lines are discrete geodesics, i.e., appear as straight lines in
the top view. To construct $V$, we prescribe the top view and two
parameter lines of different families. From these Cauchy data, planarity
of quads yields $V$, assuming that the net is simply connected. For the
V-nets in~\cite{AGAG-2023}, the top views of all parameter lines are
tangents of the same conic. 
We show that isotropic V-nets are flexible within our definition
of isometries. 

\begin{SCfigure}[2][t]
  \begin{overpic}[width=.33\textwidth]{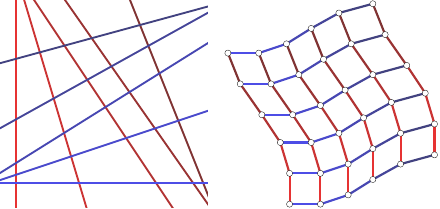}
  \end{overpic}
  \caption{The top view of the parameter curves of an isotropic V-net are
  straight lines. The top view of its metric dual is a reciprocal parallel
  net hence a translational net.}
  \label{fig:translational-net}
\end{SCfigure}

\begin{prop}
  \label{prop:v-net}
  An isotropic V-net $F$ can be embedded into a continuous family of
  isometric isotropic V-nets $F(t)$. 
\end{prop}
\begin{proof}
  We apply metric duality to $F$ and obtain $L B = \nu(F)$. The top views
  are reciprocal parallel nets.
  The reciprocal parallel net to a planar net whose parameter lines are
  straight lines, must be a translational nets
  (Fig.~\ref{fig:translational-net}) because collinear edges in a
  parameter line of one net correspond to parallel (transversal) edges
  joining neighboring parameter lines in the reciprocal parallel one. 

  To verify the existence of a continuous family of isometries we now have
  to show that the translational net $L B$ can undergo a continuous family
  of area preserving Combescure transformations, generating nets $L B(t)$.
  A Combescure transformation is defined by prescribing the parallel
  images of two parameter lines of different families, and thus we scale
  each edge of one parameter line by a factor $t$ and the edges of the
  other parameter  line by $1/t$. 
\end{proof}

Every finite isometric deformation is at each moment an infinitesimal
isometry which therefore gives rise to a Darboux wreath.

\begin{prop}
  The Darboux wreath defined by an isotropic V-net $F$ consists of the
  following surfaces and relations between them: 
  \begin{enumerate}
    \item\label{itm:1} 
      $F$ and $L V$ are isotropic V-nets with the same top view, related
      by the metric dual the to K{\oe}nigs duality. 
    \item\label{itm:2} 
      The isotropic dihedral angle along each parameter line of 
      $F$ and $L V$ is constant.
    \item\label{itm:3} 
      The nets $L B = \nu(F)$ and $\bar C = \nu(L V)$ are translational
      nets and they are Combescure transforms of each other. Even more,
      they are related by K{\oe}nigs duality and their top views are
      reciprocal diagrams to the top views of $F$ and $L V$.
    \item\label{itm:4} 
      The nets $L C$ and $\bar B$ are reciprocal parallel nets to $F$ and
      $L V$, respectively. They are isotropic K-nets, i.e., discrete
      surfaces of constant isotropic Gaussian curvature. Nets $L C$ and
      $\bar B$ are A-nets and isotropic Chebyshev nets, since they appear
      as translational nets in the top view (same as the top views of
      $\bar C$ and $L B$, respectively).
  \end{enumerate}
\end{prop}
\begin{proof}
  Lemma~\ref{lem:rigidfaces} implies that $L V$ is a Q-net. It has the
  same top view as $F$ (Def~\ref{defn:discrinfintiso}) and is therefore
  also a $V$-net which implies~\ref{itm:1}.

  By Lemma~\ref{lem:reciprocal}, $L C$ and $\bar B$ are reciprocal
  parallel to $F$ and $L V$, respectively. Hence, so are their top views.
  By the proof of Proposition~\ref{prop:v-net} the reciprocal parallel net
  to a planar net whose parameter lines are straight lines, must be a
  translational nets.
  This implies that the A-nets $L C$ and $\bar B$ appear as translational
  nets in the top view. Together with Lemma~\ref{lem:reciprocal} we
  obtain~\ref{itm:4}.

  Since the top views of $\bar C$ and $L B$ agree as well as of $L C$ and
  $\bar B$ we obtain that the Q-nets $L B$ and $\bar C$ are translational
  nets and together with Proposition~\ref{prop:rigidplanarfaces} we
  obtain~\ref{itm:3}.

  The sequences of parallel edges on $L B$ and $\bar C$ have constant
  isotropic length, implying by metric duality the constant isotropic
  dihedral angles for $F$ and $V$, hence~\ref{itm:2}.
\end{proof}

We mention that the discrete isotropic K-nets have completely analogous
properties as their smooth counterparts studied by
Strubecker~\cite{strubecker2}. They also appeared in connection with
smooth extensions of A-nets~\cite{kaef-pot-2012}. If one fills each quad
in an isotropic K-net by a bilinear patch (part of a hyperbolic
paraboloid), the union of these patches is a $C^1$ surface.

\section{Conclusion and Future Research} 
\label{sec:future}

While isotropic geometry has been systematically studied by K.~Strubecker
and inspired a large body of follow-up research, a promising definition of
isometric surfaces in $I^3$ and a study of their properties have been
missing so far. The present paper provides the fundamental definition and
basic results with the goal to fill this gap. We put special emphasis on
infinitesimal isometries, since they exhibit close  relations to concepts
in statics which recently received interest in connection with
computational structural design and architectural geometry
\cite{Millar2021,vouga-2012-sss}. Moreover, important special cases
possess elegant discrete representations which add new perspectives to
concepts in discrete differential geometry~\cite{bobenko-2008-ddg}. 

Our original motivation for studying isotropic isometries has been a
complete classification of isotropic counterparts to flexible quad meshes.
If these meshes are Q-nets, our study provides the basic approach:
Application of metric duality in $I^3$ maps a flexible Q-net in $I^3$ to a
Q-net that possesses a one-parameter family of Combescure transforms
which have the same area in corresponding faces. Such Q-nets have been
classified in the meanwhile~\cite{combescure-2024}, namely for nets with
$m \times n$ faces and not just $3 \times 3$ nets. Based on this result,
we have completed a thorough study of isotropic flexible
 Q-nets \cite{flexible-isotropic-2025}.  

There is still a lot of room for future research. We address a few of
these topics.

\begin{enumerate}
  \item We are currently working on the completion of our project on
    flexible Q-nets. Since we already have a constructive approach
    to isotropic flexible Q-nets \cite{flexible-isotropic-2025}, we
    need to transform them into Euclidean flexible meshes. For
    that, we have already developed the necessary optimization
    framework~\cite{quadmech-2024}. First numerical experiments using
    isotropic flexible nets as initial guesses have been surprisingly
    successful, which leads to the open question for a classification of
    all meshes which are flexible in isotropic and Euclidean geometry.
  \item So far, we are missing a \emph{general discrete theory of
    isometries in Euclidean and isotropic geometry} that would naturally
    extend towards infinitesimal isometries and the diagrams of Sauer
    mentioned above. The discretization of Euclidean isometries by Jiang
    et al.~\cite{isometries-2021} has been very successful in geometric
    computing, but the complete set of associated diagrams is not easily
    accessible. Related to this problem is a study of isotropic flexible
    nets with not necessarily planar faces. 
  \item Further research could address the paratactic map, in particular a
    discrete version to derive the main properties in a purely elementary
    way. Related to that is a discrete theory of general contact element
    nets, extending the special cases of principal contact element
    nets~\cite{bobenko-2008-ddg}.
\end{enumerate}

\subsection*{Acknowledgements}

The authors gratefully acknowledge the support by the Austrian Science
Fund (FWF) through grant I~4868 (DOI 10.55776/I4868).

\bibliographystyle{abbrv}
\bibliography{main}

\end{document}